\documentclass{article}
\usepackage[utf8]{inputenc}
\usepackage[utf8]{inputenc}
\usepackage[utf8]{inputenc}
\usepackage{dirtytalk}
\usepackage{tikz}
\usepackage{geometry}
\usepackage{graphicx}
\usepackage{enumitem}
\usepackage{parskip}
\usepackage{amsfonts}
\usepackage{amssymb}
\usepackage{amsmath}
\usepackage{amsthm}
\usepackage{xcolor}
\usepackage{hyperref}
\newlist{myitemize}{itemize}{1}
\setlist[myitemize,1]{leftmargin = 0.5in}
\hypersetup{colorlinks = true, linkcolor = red,  linktocpage}

\theoremstyle{plain}
\newtheorem{thm}{Theorem}[section]
\newtheorem*{thm*}{Theorem}

\newtheorem{cor}[thm]{Corollary}

\theoremstyle{definition}

\newtheorem{conj}[thm]{Conjecture}

\newtheorem{rem}[thm]{Remark}

\title{\textbf{\small{COUNTING THE NUMBER OF $\mathbb{Z}_{p}$-AND $\mathbb{F}_{p}[t]$-FIXED POINTS OF A DISCRETE DYNAMICAL SYSTEM WITH APPLICATIONS FROM ARITHMETIC STATISTICS, III}}}
\author{\footnotesize{BRIAN KINTU}}
\date{\small{\textit{To: all students who embrace learning and discovering with humility and integrity}}}

\begin{document}
\maketitle
\begin{abstract}
\small{In this follow-up paper, we again inspect a surprising relationship between the set of fixed points of a polynomial map $\varphi_{d, c}$ defined by $\varphi_{d, c}(z) = z^d + c$ for all $c, z \in \mathcal{O}_{K}$ or $\in \mathbb{Z}_{p}$ or $\in \mathbb{F}_{p}[t]$ and the coefficient $c$, where $K$ is any number field of degree $n > 1$, $p>2$ is any prime, $\mathbb{Z}_{p}$ (resp., $\mathbb{F}_{p}[t]$) is the ring of all $p$-adic integers (resp., the ring of all polynomials over a finite field $\mathbb{F}_{p}$) and $d>2$ is an integer. As in \cite{BK1,BK2} we again wish to study counting problems which are inspired by exhilarating advances in arithmetic statistics, and also by Narkiewicz on totally complex $K$-periodic points along with Adam-Fares on $\mathbb{Q}_{p}$-periodic points in arithmetic dynamics. In doing so, we then first prove that for any prime $p\geq 3$ and for any $\ell \in \mathbb{Z}_{\geq 1}$, the average number of distinct fixed points of any $\varphi_{p^{\ell}, c}$ modulo prime $p\mathcal{O}_{K}$ (modulo $p\mathbb{Z}_{p}$) is bounded or zero or unbounded as $c\to \infty$. Motivated further by $\mathbb{F}_{p}(t)$-periodic point-counting result of Benedetto, we also find that the average number in $\mathbb{F}_{p}[t]$-setting behaves in the same way as in $\mathcal{O}_{K}$-setting. More precisely, we prove that the average number of distinct fixed points of any $\varphi_{p^{\ell}, c}$ modulo  prime $\pi$ is bounded or zero or unbounded as $c$ varies; and also prove that for any prime $p\geq 5$ and for any $\ell \in \mathbb{Z}_{\geq 1}$, the average number of distinct fixed points of any $\varphi_{(p-1)^{\ell}, c}$ modulo $\pi$ (modulo $p\mathbb{Z}_{p}$) is $1$ or $2$ or $0$ as $c$ varies. Finally, we then apply density result of Bhargava-Shankar-Wang, field-counting result of Oliver Lemke-Thorne and Thunder-Widmer from arithmetic statistics, and as a result obtain counting and statistical results on irreducible polynomials, number fields and subfields of function fields arising naturally in our polynomial discrete dynamical settings.}
\end{abstract}

\begin{center}
\tableofcontents
\end{center}
\begin{center}
    \section{Introduction}\label{sec1}
\end{center}
\noindent
Consider a morphism $\varphi: {\mathbb{P}^N(K)} \rightarrow {\mathbb{P}^N(K)} $ of degree $d \geq 2$ defined on a projective space ${\mathbb{P}^N(K)}$ of dimension $N$, where $K$ is a number field. Then for any $n\in\mathbb{Z}$ and $\alpha\in\mathbb{P}^N(K)$, we call $\varphi^n = \underbrace{\varphi \circ \varphi \circ \cdots \circ \varphi}_\text{$n$ times}$ the $n^{th}$ \textit{iterate of $\varphi$}; and call $\varphi^n(\alpha)$ the \textit{$n^{th}$ iteration of $\varphi$ on $\alpha$}. By convention, $\varphi^{0}$ acts as the identity map, i.e., $\varphi^{0}(\alpha) = \alpha$ for every point $\alpha\in {\mathbb{P}^N(K)}$. The everyday philosopher may want to know (quoting here Devaney \cite{Dev}): \say{\textit{Where do points $\alpha, \varphi(\alpha), \varphi^2(\alpha), \ \cdots\ ,\varphi^n(\alpha)$ go as $n$ becomes large, and what do they do when they get there?}} So now, for any given integer $n\geq 0$ and any given point $\alpha\in {\mathbb{P}^N(K)}$, we then call the set consisting of all the iterates $\varphi^n(\alpha)$ the \textit{(forward) orbit of $\alpha$}; and which in the theory of dynamical systems we usually denote  it by $\mathcal{O}^{+}(\alpha)$.

As mentioned in the articles \cite{BK1, BK2} that one of the main goals in arithmetic dynamics (a newly emerging area of mathematics concerned with studying number-theoretic properties of discrete dynamical systems) is to classify all the points $\alpha\in\mathbb{P}^N(K)$ according to the behavior of their forward orbits $\mathcal{O}^{+}(\alpha)$. In this direction, we recall that any point $\alpha\in {\mathbb{P}^N(K)}$ is called a \textit{periodic point of $\varphi$}, whenever $\varphi^n (\alpha) = \alpha$ for some integer $n\in \mathbb{Z}_{\geq 0}$. In this case, any integer $n\geq 0$ such that the iterate $\varphi^n (\alpha) = \alpha$, is called \textit{period of $\alpha$}; and the smallest such positive integer $n\geq 1$ is called the \textit{exact period of $\alpha$}. We recall Per$(\varphi, {\mathbb{P}^N(K)})$ to denote set of all periodic points of $\varphi$; and also recall that for any given point $\alpha\in$Per$(\varphi, {\mathbb{P}^N(K)})$ the set of all iterates of $\varphi$ on $\alpha$ is called \textit{periodic orbit of $\alpha$}. In their 1994 paper \cite{Russo} and in his 1998 paper \cite{Poonen} respectively, Walde-Russo and Poonen give independently interesting examples of rational periodic points of any $\varphi_{2,c}$ defined over the field $\mathbb{Q}$; and so the interested reader may wish to revisit \cite{Russo, Poonen} to gain familiarity with the notion of periodicity of points. 

Previously in articles \cite{BK1, BK2} we (inspired by work of Bhargava-Shankar-Tsimerman (BST) and their collaborators in arithmetic statistics, and also by work of Narkiewicz in arithmetic dynamics) proved that conditioning on periodic point-counting theorem of Narkiewicz \cite{Narkie} on any $\varphi_{p^{\ell},c}$ defined over any real algebraic number field $K$ of degree $n\geq 1$, yields that the number of distinct integral fixed points of any $\varphi_{p,c}$ modulo prime $p\mathcal{O}_{K}$ equal to $3$ or $0$; from which it then followed that the average number of distinct integral fixed points of any $\varphi_{p^{\ell},c}$ modulo $p\mathcal{O}_{K}$ is also equal to $3$ or $0$ as $c\to \infty$. Moreover, we then observe in \cite{BK1, BK2} that the expected total number of distinct integral fixed points in the whole family of maps $\varphi_{p,c}$ modulo $p\mathcal{O}_{K}$ equal to $3 + 0=3$; and so independent of degree $p^{\ell}$ and $n$. So now, inspired again by work of (BST) in arithmetic statistics, and of Narkiewicz \ref{theorem 3.2.1} (which we, however, do not intend to use or condition on in all of the counting that follows in the number field setting) along with Conjecture \ref{per} of Morton-Silverman (which we only intend to understand and not (dis)prove), we revisit the setting in \cite{BK2} and then consider in Section \ref{sec2} any $\varphi_{p,c}$ defined over any number field $K$ (not necessarily real) of degree $n\geq 2$. In doing so, we then obtain unconditionally the following main theorem on any $\varphi_{p,c}$, which we state later more precisely as Theorem \ref{2.2} and which we then generalize further as Theorem \ref{2.3}; and moreover when motivated by Narkiewicz's argument on the existence of only fixed points of every map $\varphi_{p^{\ell},c}$ over $\mathbb{Q}$, we then also obtain Corollary \ref{cor2.4} as a consequence of Theorem \ref{2.3}: 

\begin{thm}\label{BB1} 
Let $K\slash \mathbb{Q}$ be any number field of degree $ n \geq 2$ with the ring of integers $\mathcal{O}_{K}$, and in which any fixed prime integer $p\geq 3$ is inert. Let $\varphi_{p, c}$ be a polynomial map defined by $\varphi_{p, c}(z) = z^p + c$ for all $c, z\in\mathcal{O}_{K}$. Then the number of distinct integral fixed points of any polynomial map $\varphi_{p,c}$ modulo $p\mathcal{O}_{K}$ is either $p$ or zero. 
\end{thm}

Inspired greatly and guided by work of (BST) and their collaborators in arithmetic statistics, and also by work of Adam-Fares \cite{Ada} on $\mathbb{Q}_{p}$-periodic points (periodic orbits) of any $\varphi_{p^{\ell}, c}$ defined over $\mathbb{Z}_{p}$ in arithmetic dynamics, we revisit the setting in Section \ref{sec2} and then prove in Section \ref{sec3} the following main theorem on any $\varphi_{p, c}$, which we state later more precisely as Theorem \ref{3.2}; and which we then generalize further as Theorem \ref{3.3}:

\begin{thm}\label{BB2} 
Let $p\geq 3$ be any fixed prime integer, and $\varphi_{p, c}$ be any polynomial map defined by $\varphi_{p, c}(z) = z^p + c$ for all $c, z\in\mathbb{Z}_{p}$. Then the number of distinct $p$-adic integral fixed points of every $\varphi_{p,c}$ modulo $p\mathbb{Z}_{p}$ is $p$ or zero. 
\end{thm}

Motivated further by that same spirit of work and activity of (BST) in arithmetic statistics and of Adam-Fares \cite{Ada}, we revisit the setting in Section \ref{sec3} and then prove in Section \ref{sec4} the following main theorem on any  $\varphi_{p-1,c}$, which we state later more precisely as Theorem \ref{4.2}; and which we then generalize more as Theorem \ref{4.3}:

\begin{thm}\label{BB3}
Let $p\geq 5$ be any fixed prime, and $\varphi_{p-1, c}$ be any polynomial map defined by $\varphi_{p-1, c}(z) = z^{p-1} + c$ for all $c, z\in\mathbb{Z}_{p}$. The number of distinct $p$-adic integral fixed points of any $\varphi_{p-1,c}$ modulo $p\mathbb{Z}_{p}$ is $1$ or $2$ or zero.
\end{thm}

\noindent Notice that the count obtained in Theorem \ref{BB3} on the number of distinct $p$-adic integral fixed points of any $\varphi_{p-1,c}$ modulo $p\mathbb{Z}_{p}$ is independent of $p$ (and so independent of the deg$(\varphi_{p-1,c})$) in each of the possibilities. Moreover, we may also observe that the expected total count (namely, $1+2+0 =3$) in Theorem \ref{BB3} on the number of distinct $p$-adic integral fixed points in the whole family of maps $\varphi_{p-1,c}$ modulo $p\mathbb{Z}_{p}$ is not only also independent of $p$ and deg$(\varphi_{p-1,c})$, but is also a constant $3$ even as $p-1\to \infty$. On the other hand, we may also notice that the count obtained in Theorem \ref{BB2} on the number of distinct $p$-adic integral fixed points of any $\varphi_{p,c}$ modulo $p\mathbb{Z}_{p}$ may depend on $p$ (and so depend on deg$(\varphi_{p,c}))$ in one of the possibilities. Consequently, the expected total count (namely, $p+0 =p$) in Theorem \ref{BB2} on the number of distinct $p$-adic integral fixed points in the whole family of maps $\varphi_{p,c}$ modulo $p\mathbb{Z}_{p}$ may not only depend on $p$, but also may grow to infinity as $p\to \infty$.

Motivated by a \say{counting-application} philosophy in arithmetic statistics and in function field number theory, along with periodic point-counting result of Benedetto in Theorem \ref{main} (also applying to any polynomial map over any rational function field $\mathbb{F}_{p}(t)$) in arithmetic dynamics, we revisit the setting in Section \ref{sec2} and then consider in Section \ref{sec5} any $\varphi_{p,c}$ over any $\mathbb{F}_{p}[t]$. In doing so, we then prove the following main theorem on any $\varphi_{p, c}$, which we state later more precisely as Theorem \ref{5.2}; and which we then generalize further as Theorem \ref{5.3}:

\begin{thm}\label{BB4} 
Let $p\geq 3$ be any fixed prime integer, and let $\pi\in \mathbb{F}_{p}[t]$ be any fixed irreducible monic polynomial of degree $m\geq 1$. Consider any family of polynomial maps $\varphi_{p, c}$ defined by $\varphi_{p, c}(z) = z^p + c$ for all polynomials $c, z\in\mathbb{F}_{p}[t]$. Then the number of distinct fixed points of any polynomial map $\varphi_{p,c}$ modulo $\pi$ is either $p$ or zero. 
\end{thm}

Motivated further by that same \say{counting-application} philosophy in arithmetic statistics and in function field number theory, along with $\mathbb{F}_{p}(t)$-periodic point-counting result of Benedetto in Theorem \ref{main} in arithmetic dynamics, we revisit the setting in Section \ref{sec5} and then prove in Section \ref{sec6} the following main theorem on any $\varphi_{p-1,c}$, which we state later more precisely as Theorem \ref{6.2}; and which we then generalize more as Theorem \ref{6.3}:

\begin{thm}\label{BB5}
Let $p\geq 5$ be any fixed prime integer, and let $\pi\in \mathbb{F}_{p}[t]$ be any fixed irreducible monic polynomial of degree $m\geq 1$. Consider any family of polynomial maps $\varphi_{p-1, c}$  defined by $\varphi_{p-1, c}(z) = z^{p-1} + c$ for all $c, z\in\mathbb{F}_{p}[t]$. Then the number of distinct fixed points of any polynomial map $\varphi_{p-1,c}$ modulo $\pi$ is $1$ or $2$ or zero.
\end{thm}

\noindent As before, we may again notice that the count obtained in Theorem \ref{BB5} on the number of distinct fixed points of any $\varphi_{p-1,c}$ modulo $\pi$ is independent of $p$ (and so independent of deg$(\varphi_{p-1,c})$) and $m$ in each of the  possibilities. Moreover, we may again also observe that the expected total count (namely, $1+2+0 =3$) in Theorem \ref{BB5} on the number of distinct fixed points in the whole family of maps $\varphi_{p-1,c}$ modulo $\pi$ is not only also independent of deg$(\varphi_{p-1,c})$ and $m$, but is also a constant $3$ even as $p-1\to \infty$ or $m\to \infty$. On the other hand, we may also notice that the count obtained in Theorem \ref{BB4} on the number of distinct fixed points of any $\varphi_{p,c}$ modulo $\pi$ may depend on $p$ (and so depend on deg$(\varphi_{p,c})$) in one of the possibilities. Consequently, as before the expected total count (namely, $p+0 =p$) in Theorem \ref{BB4} on the number of distinct fixed points in the whole family of maps $\varphi_{p,c}$ modulo $\pi$ may not only depend on $p$, but also may grow to infinity as $p \to \infty$. Mind you, we noticed earlier that this same phenomena may also occur in $\mathcal{O}_{K}$- and $\mathbb{Z}_{p}$-setting in Theorem \ref{BB1} and \ref{BB2}, respectively.

In addition, to the notion of a periodic point and a periodic orbit, we also recall that a point $\alpha\in {\mathbb{P}^N(K)}$ is called a \textit{preperiodic point of $\varphi$}, whenever the iterate $\varphi^{m+n}(\alpha) = \varphi^{m}(\alpha)$ for some integers $m\geq 0$ and $n\geq 1$. In this case, we recall that the smallest integers $m\geq 0$ and $n\geq 1$ such that $\varphi^{m+n}(\alpha) = \varphi^{m}(\alpha)$ happens, are called the \textit{preperiod} and \textit{eventual period of $\alpha$}, respectively. Again, we denote the set of preperiodic points of $\varphi$ by PrePer$(\varphi, {\mathbb{P}^N(K)})$. For any given preperiodic point $\alpha$ of $\varphi$, we then call the set of all iterates of $\varphi$ on $\alpha$, \textit{the preperiodic orbit of $\alpha$}.
Now observe for $m=0$, we have $\varphi^{n}(\alpha) = \alpha $ and so $\alpha$ is a periodic point of period $n$. Thus, the set  Per$(\varphi, {\mathbb{P}^N(K)}) \subseteq$ PrePer$(\varphi, {\mathbb{P}^N(K)})$; however, it need not be PrePer$(\varphi, {\mathbb{P}^N(K)})\subseteq$ Per$(\varphi, {\mathbb{P}^N(K)})$. In their 2014 paper \cite{Doyle}, Doyle-Faber-Krumm give nice examples (which also recovers examples in Poonen's paper \cite{Poonen}) of preperiodic points of any quadratic map $\varphi$ (where $\varphi$ is not necessarily the somewhat mostly studied $\varphi_{2,c}$ in arithmetic dynamics) defined over quadratic fields; and so the interested reader may wish to revisit \cite{Poonen, Doyle}.

In the year 1950, Northcott \cite{North} used the theory of height functions to show that not only is the set PrePer$(\varphi, {\mathbb{P}^N(K)})$ always finite, but also for a given morphism $\varphi$ the set PrePer$(\varphi, {\mathbb{P}^N(K)})$ can be computed effectively. Forty-five years later, in the year 1995, Morton and Silverman conjectured that PrePer$(\varphi, \mathbb{P}^N(K))$ can be bounded in terms of degree $d$ of $\varphi$, degree $D$ of $K$, and dimension $N$ of the space ${\mathbb{P}^N(K)}$. This celebrated conjecture is called the \textit{Uniform Boundedness Conjecture}; which we then restate here as the following conjecture:

\begin{conj} \label{silver-morton}[\cite{Morton}]
Fix integers $D \geq 1$, $N \geq 1$, and $d \geq 2$. There exists a constant $C'= C'(D, N, d)$ such that for all number fields $K/{\mathbb{Q}}$ of degree at most $D$, and all morphisms $\varphi: {\mathbb{P}^N}(K) \rightarrow {\mathbb{P}^N}(K)$ of degree $d$ defined over $K$, the total number of preperiodic points of a morphism $\varphi$ is at most $C'$, i.e., \#PrePer$(\varphi, \mathbb{P}^N(K)) \leq C'$.
\end{conj}
\noindent A special case of Conjecture \ref{silver-morton} is when the degree $D$ of a number field $K$ is $D = 1$, dimension $N$ of a space $\mathbb{P}^N(K)$ is $N = 1$, and degree $d$ of a morphism $\varphi$ is $d = 2$. In this case, if $\varphi$ is a polynomial morphism, then it is a quadratic map defined over the field $\mathbb{Q}$. Moreover, in this very special case, in the year 1995, Flynn and Poonen and Schaefer conjectured that a quadratic map has no points $z\in\mathbb{Q}$ with exact period more than 3. This conjecture of Flynn-Poonen-Schaefer \cite{Flynn} (which has been resolved for cases $n = 4$, $5$ in \cite{mor, Flynn} respectively and conditionally for $n=6$ in \cite{Stoll} is, however, still open for all integers $n\geq 7$ and moreover, which also Hutz-Ingram \cite{Ingram} gave strong computational evidence supporting it) is restated here formally as the following conjecture. Note that in this same special case, rational points of exact period $n\in \{1, 2, 3\}$ were first found in the year 1994 by Russo-Walde \cite{Russo} and also found in the year 1995 by Poonen \cite{Poonen} using a different set of techniques. We now restate the anticipated conjecture of Flynn-Poonen-Schaefer as the following conjecture:
 
\begin{conj} \label{conj:2.4.1}[\cite{Flynn}, Conj. 2]
If $n \geq 4$, then there is no $\varphi_{2,c}(z)\in \mathbb{Q}[z]$ with a $\mathbb{Q}$-point of exact period $n$.
\end{conj}
Now by assuming Conjecture \ref{conj:2.4.1} and also establishing interesting results on $\mathbb{Q}$-preperiodic points of any $\varphi_{2,c}$, in the year 1998, Poonen \cite{Poonen} then concluded that the total number of rational preperiodic points of any quadratic polynomial $\varphi_{2, c}(z)$ is at most nine. We restate here formally Poonen's result as the following corollary:
\begin{cor}\label{cor2}[\cite{Poonen}, Corollary 1]
If Conjecture \ref{conj:2.4.1} holds, then $\#$PrePer$(\varphi_{2,c}, \mathbb{Q}) \leq 9$,  for all quadratic maps $\varphi_{2, c}$ defined by $\varphi_{2, c}(z) = z^2 + c$ for all points $c, z\in\mathbb{Q}$.
\end{cor}

On still the same note of exact periods and pre(periodic) points, the next natural question that one could ask is whether the aforementioned phenomenon on exact periods and pre(periodic) points has been investigated in some other cases, namely, when $D\geq 2$, $N\geq 1$ and $d\geq 2$. In the case $D = d = 2$ and $N = 1$, then again if $\varphi$ is a polynomial map, then $\varphi$ is a quadratic map defined over a quadratic field $K = \mathbb{Q}(\sqrt{D'})$. In this case, in the years 1900, 1998 and 2006, Netto \cite{Netto}, Morton-Silverman \cite{Morton} and Erkama \cite{Erkama} resp., found independently a parametrization of a point $c$ in the field $\mathbb{C}$ of all complex points which guarantees $\varphi_{2,c}$ to have periodic points of period $M=4$. And moreover when $c\in \mathbb{Q}$, Panraksa \cite{par1} showed that one gets \textit{all} orbits of length $M = 4$ defined over $\mathbb{Q}(\sqrt{D'})$. For $M=5$, Flynn-Poonen-Schaefer \cite{Flynn} found a parametrization of a point $c\in \mathbb{C}$ that yields points of period 5; however, these periodic points are not in $K$, but rather in some other extension of $\mathbb{Q}$. In the same case $D = d = 2$ and $N = 1$, Hutz-Ingram \cite{Ingram} and Doyle-Faber-Krumm \cite{Doyle} did not find in their computational investigations points $c\in K$ for which $\varphi_{2,c}$ defined over $K$ has $K$-rational points of exact period $M = 5$. Note that to say that the above authors didn't find points $c\in K$ for which $\varphi_{2,c}$ has $K$-rational points of exact period $M=5$, is not the same as saying that such points do not exist; since it's possible that the techniques which the authors employed in their computational investigations may have been far from enabling them to decide concretely whether such points exist or not. In fact, as of the present article, we do not know whether $\varphi_{2,c}$ has $K$-rational points of exact period $5$ or not, but surprisingly from \cite{Flynn, Stoll, Ingram, Doyle} we know that for $c=-\frac{71}{48}$ and $D'=33$ the map $\varphi_{2,c}$ defined over $K = \mathbb{Q}(\sqrt{33})$ has $K$-rational points of exact period $M = 6$; and mind you, this is the only example of $K$-rational points of exact period $M=6$ that is currently known of in the whole literature of arithmetic dynamics. For $M>6$, in 2013, Hutz-Ingram [\cite{Ingram}, Prop. 2 and 3] gave strong computational evidence which showed that for any absolute discriminant $D'$ at most 4000 and any $c\in K$ with a certain logarithmic height, the map $\varphi_{2,c}$ defined over any $K$ has no $K$-rational points of exact period greater than 6. Moreover, the same authors \cite{Ingram} also showed that the smallest upper bound on the size of PrePer$(\varphi_{2,c}, K)$ is 15. A year later, in 2014, Doyle-Faber-Krumm \cite{Doyle} also gave computational evidence on 250000 pairs $(K, \varphi_{2,c})$ which not only established the same claim [\cite{Doyle}, Thm 1.2] as that of Hutz-Ingram \cite{Ingram} on the upper bound of the size of PrePer$(\varphi_{2,c}, K)$, but it also covered Poonen's claims in \cite{Poonen} on $\varphi_{2,c}$ over $\mathbb{Q}$. Three years later, in 2018, Doyle \cite{Doy} adjusted the computations in his aforementioned cited work with Faber and Krumm; and after which he then made the following conjecture on any quadratic map over any $K = \mathbb{Q}(\sqrt{D'})$:

\begin{conj}\label{do}[\cite{Doy}, Conjecture 1.4]
Let $K\slash \mathbb{Q}$ be a quadratic field and let $f\in K[z]$ be a quadratic polynomial. \newline Then, $\#$PrePer$(f, K)\leq 15$.
\end{conj}

Since Per$(\varphi, {\mathbb{P}^N(K)}) \subseteq$ PrePer$(\varphi, {\mathbb{P}^N(K)})$ and so if the size of PrePer$(\varphi, \mathbb{P}^N(K))$ is bounded above, then the size of Per$(\varphi, \mathbb{P}^N(K))$ is also bounded above and moreover bounded above by the same upper bound. Hence, we may focus on a periodic version \ref{per} of \ref{silver-morton}, and this is because in Section \ref{sec2} we consider polynomial maps of odd prime power degree $d\geq 3$ defined over the ring $\mathcal{O}_{K}$ of integers where $K$ is any number field (not necessarily real) of degree $n\geq 2$ and also consider in Section \ref{sec3} and \ref{sec5} the same maps defined independently over the ring $\mathbb{Z}_{p}$ of $p$-adic integers and over the ring $\mathbb{F}_{p}[t]$ of polynomials over a finite field $\mathbb{F}_{p}$ for any prime $p\geq 3$; and lastly consider again in Section \ref{sec4} and \ref{sec6} maps of even power degree $d\geq 4$ defined independently over $\mathbb{Z}_{p}$ and $\mathbb{F}_{p}[t]$ for any prime $p\geq 5$, resp., all in the attempt of understanding the possibility and validity of such a version of \ref{silver-morton}.

\begin{conj} \label{silver-morton 1}($(n,1)$-version of Conjecture \ref{silver-morton})\label{per}
Fix integers $n \geq 1$ and $d \geq 2$. There exists a constant $C'= C'(n, d)$ such that for all number fields $K/{\mathbb{Q}}$ of degree at most $n$, and all morphisms $\varphi: {\mathbb{P}^1}(K) \rightarrow {\mathbb{P}^1}(K)$ of degree $d$ over $K$, the total number of periodic points of a morphism $\varphi$ is at most $C'$, i.e., \#Per$(\varphi, \mathbb{P}^1(K)) \leq C'$.
\end{conj}

\subsection*{History on the Connection Between the Size of Per$(\varphi_{d, c}, K)$ and the Coefficient $c$}

In the year 1994, Walde and Russo not only proved [\cite{Russo}, Corollary 4] that for a quadratic map $\varphi_{2,c}$ defined over $\mathbb{Q}$ with a periodic point, the denominator of a rational point $c$, denoted as den$(c)$, is a square but they also proved that den$(c)$ is even, whenever $\varphi_{2,c}$ admits a rational cycle of length $\ell \geq 3$. Moreover, Walde-Russo also proved [\cite{Russo}, Cor. 6, Thm 8 and Cor. 7] that the size \#Per$(\varphi_{2, c}, \mathbb{Q})\leq 2$, whenever den$(c)$ is an odd integer. 

Three years later, in the year 1997, Call-Goldstine \cite{Call} proved that the size of PrePer$(\varphi_{2,c},\mathbb{Q})$ can be bounded above in terms of the number of distinct odd primes dividing den$(c)$. We restate formally this result of Call-Goldstine as the following theorem, in which $GCD(a, e)$ refers to the greatest common divisor of $a$, $e \in \mathbb{Z}$:

\begin{thm}\label{2.3.1}[\cite{Call}, Theorem 6.9]
Let $e>0$ be an integer and let $s$ be the number of distinct odd prime factors of e. Define $\varepsilon  = 0$, $1$, $2$, if $4\nmid e$, if $4\mid e$ and $8 \nmid e$, if $8 \mid e$, respectively. Let $c = a/e^2$, where $a\in \mathbb{Z}$ and $GCD(a, e) = 1$. If $c \neq -2$, then the total number of $\mathbb{Q}$-preperiodic points of $\varphi_{2, c}$ is at most $2^{s + 2 + \varepsilon} + 1$. Moreover, a quadratic map $\varphi_{2, -2}$ has exactly six rational preperiodic points.
\end{thm}

Eight years later, after the work of Call-Goldstine, in the year 2005, Benedetto \cite{detto} studied polynomial maps $\varphi$ of arbitrary degree $d\geq 2$ defined over an arbitrary global field $K$, and then established the following result on the relationship between the size of the set PrePre$(\varphi, K)$ and the number of bad primes of $\varphi$ in $K$:

\begin{thm}\label{main} [\cite{detto}, Main Theorem]
Let $K$ be a global field, $\varphi\in K[z]$ be a polynomial of degree $d\geq 2$ and $s$ be the number of bad primes of $\varphi$ in $K$. The number of preperiodic points of $\varphi$ in $\mathbb{P}^N(K)$ is at most $O(\text{s log s})$. 
\end{thm}
 
\noindent Since Benedetto's Theorem \ref{main} applies to any polynomial $\varphi$ of arbitrary degree $d\geq 2$ defined over any algebraic number field $K$ or over any function field, it then also follows that one can immediately apply Theorem \ref{main} to any polynomial $\varphi$ of arbitrary odd degree $d\geq 2$ defined over any algebraic number field $K$ or over any function field, and as a result obtain the upper bound predicted in Theorem \ref{main} on the number of $K$-preperiodic points. 

Five years after the work of Benedetto, in the year 2010, Narkiewicz's \cite{Narkie1} proved that any polynomial map $\varphi_{d,c}$ of odd prime-power degree $d = p^{\ell}$ with $\ell\geq 1$ defined over any totally complex extension $K\slash \mathbb{Q}$ of degree $n$ where $K$ does not contain $p$-\text{th} roots of unity, the length of $K$-cycles of a polynomial map $\varphi_{d,c}$ are bounded by a certain constant $B = B(K, p)$ depending only on $K$ and $p$; and moreover if $p$ exceeds $2^n$, then the bound depends only on degree $n$ of $K$. We restate here more formally Narkiewicz's result as the following:

\begin{thm} \label{theorem 3.2.1}[\cite{Narkie1}, Theorem]
Let $K$ be a totally complex extension of $\mathbb{Q}$ of degree $n>1$, denote by $R$ its ring of integers and $D$ be the maximal order of a primitive root of unity contained in $K$. Let $p$ be a prime not dividing $D$, and put $F(X) = X^n + c\in K[X]$ with $n = p^k$ with $k\geq 1$ and $c\neq 0$. Then the lengths of cycles of $F$ in $K$ are bounded by a constant $B = B(K, p)$. If $p>2^n$, then this constant can be taken to be $n2^{n+1}(2^n-1)$.
\end{thm}\noindent Now recall in arithmetic dynamics, and more generally in classical dynamical systems that we can always identify any $K$-orbit of any map, say $\varphi_{p,c}$, with any $K$-cycle of the same map. So then, if we were working under the assumptions in Theorem \ref{theorem 3.2.1}, then by Theorem \ref{theorem 3.2.1} the total number of distinct points in any $K$-orbit is bounded by a constant $B$ depending on only $K$ and $p$; and moreover $B = n2^{n+1}(2^n -1)$ whenever $p>2^n$. 

Seven years later, after some work of Benedetto and other several authors working on non-archimedean dynamics, in the year 2012, Adam-Fares [\cite{Ada}, Proposition 15] studied the dynamical system $(K, \ x^{p^{\ell}}+c)$ where $K$ is a local field equipped with a discrete valuation and $\ell\in \mathbb{Z}^{+}$. In the case $K = \mathbb{Q}_{p}$, they showed that the polynomial $\varphi_{p^{\ell},c}(x) = x^{p^{\ell}} + c$ where $c\in \mathbb{Z}_{p}$, either has $p$ fixed points or a periodic orbit of exact period $p$ in $\mathbb{Q}_{p}$.

Three years after \cite{Narkie}, in 2015, Hutz \cite{Hutz} developed an algorithm determining effectively all $\mathbb{Q}$-preperiodic points of a morphism defined over a given number field $K$; from which he then made the following conjecture: 

\begin{conj} \label{conjecture 3.2.1}[\cite{Hutz}, Conjecture 1a]
For any integer $n > 2$, there is no even degree $d > 2$ and no point $c \in \mathbb{Q}$ such that the polynomial map $\varphi_{d, c}$ has rational points of exact period $n$.
Moreover, \#PrePer$(\varphi_{d, c}, \mathbb{Q}) \leq 4$. 
\end{conj}

\noindent On the note whether any theoretical progress has yet been made on Conjecture \ref{conjecture 3.2.1}, more recently, Panraksa \cite{par2} proved among many other results that the quartic polynomial $\varphi_{4,c}(z)\in\mathbb{Q}[z]$ has rational points of exact period $n = 2$. Moreover, he also proved that $\varphi_{d,c}(z)\in\mathbb{Q}[z]$ has no rational points of exact period $n = 2$ for any $c \in \mathbb{Q}$ with $c \neq -1$ and $d = 6$, $2k$ with $3 \mid 2k-1$. The interested reader may find these mentioned results of Panraksa in his unconditional Thms 2.1, 2.4 and also see his Thm 1.7 conditioned on the abc-conjecture in \cite{par2}.

Twenty-eight years later, after the work of Walde-Russo, in the year 2022, Eliahou-Fares proved [\cite{Shalom2}, Theorem 2.12] that the denominator of a rational point $-c$, denoted as den$(-c)$ is divisible by 16, whenever $\varphi_{2,-c}$ defined by $\varphi_{2, -c}(z) = z^2 - c$ for all $c, z\in \mathbb{Q}$ admits a rational cycle of length $\ell \geq 3$. Moreover, they also proved [\cite{Shalom2}, Proposition 2.8] that the size \#Per$(\varphi_{2, -c}, \mathbb{Q})\leq 2$, whenever den$(-c)$ is an odd integer. Motivated by \cite{Call}, Eliahou-Fares \cite{Shalom2} also proved that the size of Per$(\varphi_{2, -c}, \mathbb{Q})$ can be bounded above by using information on den$(-c)$, namely, information in terms of the number of distinct primes dividing den$(-c)$. Moreover, they in \cite{Shalom1} also showed that the upper bound is four, whenever $c\in \mathbb{Q^*} = \mathbb{Q}\setminus\{0\}$. We restate here their results as:

\begin{cor}\label{sha}[\cite{Shalom2, Shalom1}, Cor. 3.11 and Cor. 4.4, respectively]
Let $c\in \mathbb{Q}$ such that den$(c) = d^2$ with $d\in 4 \mathbb{N}$. Let $s$ be the number of distinct primes dividing $d$. Then, the total number of $\mathbb{Q}$-periodic points of $\varphi_{2, -c}$ is at most $2^s + 2$. Moreover, for $c\in \mathbb{Q^*}$ such that the den$(c)$ is a power of a prime number. Then, $\#$Per$(\varphi_{2, c}, \mathbb{Q}) \leq 4$.
\end{cor}

\noindent The purpose of this article is to once again inspect further the above connection in the case of maps $\varphi_{p^{\ell}, c}$ defined, first over the ring of integers $\mathcal{O}_{K}$ where $K$ is any number field (not necessarily real) of degree $n\geq 2$, then defined independently over $\mathbb{Z}_{p}$ and $\mathbb{F}_{p}[t]$ for any given prime $p\geq 3$ and $\ell \in \mathbb{Z}^{+}$; and then also consider the case of maps $\varphi_{(p-1)^{\ell}, c}$ defined independently over $\mathbb{Z}_{p}$ and $\mathbb{F}_{p}[t]$ for any prime $p\geq 5$ and $\ell \in \mathbb{Z}^{+}$; and doing all of this from a spirit that's inspired and guided by some of the many striking developments in arithmetic statistics.

\section{On Number of Integral Fixed Points of any Family of Polynomial Maps $\varphi_{p^{\ell},c}$}\label{sec2}

In this section, we wish to count the number of distinct integral fixed points of any polynomial map $\varphi_{p^{\ell},c}$ modulo fixed prime ideal $p\mathcal{O}_{K}$ for any given prime $p\geq 3$ and for any integer $\ell \geq 1$. With that in mind, we let $p\geq 3$ be any prime, $\ell\geq 1$ be any integer and $c\in \mathcal{O}_{K}$ be any point, and then define fixed point-counting function 
\begin{equation}\label{N_{c}}
N_{c}(p) := \# \bigg\{ z\in \mathcal{O}_{K} / p\mathcal{O}_{K} : \varphi_{p^{\ell},c}(z) - z \equiv 0 \ \text{(mod $p\mathcal{O}_{K}$)}\bigg\}.
\end{equation}\noindent Setting $\ell =1$ and so the map $\varphi_{p^{\ell}, c} = \varphi_{p,c}$, we then first prove the following theorem and its generalization \ref{2.2}:
\begin{thm} \label{2.1}
Let $K\slash \mathbb{Q}$ be any number field of degree $n \geq 2$ with the ring of integers $\mathcal{O}_{K}$, and in which $3$ is inert. Let $\varphi_{3, c}$ be a cubic map defined by $\varphi_{3, c}(z) = z^3 + c$ for all $c, z\in\mathcal{O}_{K}$, and let $N_{c}(3)$ be the number defined as in \textnormal{(\ref{N_{c}})}. Then $N_{c}(3) = 3$ for every coefficient $c\in 3\mathcal{O}_{K}$; otherwise we have $N_{c}(3) = 0$ for every point $c \not \in 3\mathcal{O}_{K}$.
\end{thm}

\begin{proof}
Let $f(z)= \varphi_{3, c}(z)-z = z^3 - z + c$ and note that for every coefficient $c\in 3\mathcal{O}_{K}$, reducing $f(z)$ modulo prime ideal $3\mathcal{O}_{K}$, we then obtain $f(z)\equiv z^3 - z$ (mod $3\mathcal{O}_{K}$); and from which it then follows that the reduced cubic polynomial $f(z)$ modulo $3\mathcal{O}_{K}$ factors as $z(z-1)(z+1)$ over a finite field $\mathcal{O}_{K}\slash 3\mathcal{O}_{K}$ of order $3^{[K:\mathbb{Q}]} = 3^n$. But now, it then follows by the factor theorem that $z\equiv -1, 0, 1$ (mod $3\mathcal{O}_{K}$) are roots of $f(z)$ modulo $3\mathcal{O}_{K}$ in $\mathcal{O}_{K}\slash 3\mathcal{O}_{K}$. Moreover, since the reduced univariate polynomial $f(x)$ modulo $3\mathcal{O}_{K}$ is of degree $3$ over a field $\mathcal{O}_{K}\slash 3\mathcal{O}_{K}$ and so $f(x)$ modulo $3\mathcal{O}_{K}$ can have at most three roots in $\mathcal{O}_{K}\slash 3\mathcal{O}_{K}$(even counted with multiplicity), we then conclude $\#\{ z\in \mathcal{O}_{K} / 3\mathcal{O}_{K} : \varphi_{3,c}(z) - z \equiv 0 \text{ (mod $3\mathcal{O}_{K}$)}\} = 3$ and so $N_{c}(3) = 3$. To see $N_{c}(3) = 0$ for every coefficient $c \not \in 3\mathcal{O}_{K}$, we first note that since $c\not\in 3\mathcal{O}_{K}$, then this also means  $c\not \equiv 0$ (mod $3\mathcal{O}_{K}$). So now, recall from a well-known fact about subfields of finite fields that every subfield of $\mathcal{O}_{K}\slash 3\mathcal{O}_{K}$ is of order $3^r$ for some positive integer $r\mid n$, we then obtain the inclusion $\mathbb{F}_{3}\hookrightarrow\mathcal{O}_{K}\slash 3\mathcal{O}_{K}$ of fields, where $\mathbb{F}_{3}$ is a field of order 3; and moreover recall (as a fact) that $z^3 = z$ for every element $z\in \mathbb{F}_{3} \subset\mathcal{O}_{K}\slash 3\mathcal{O}_{K}$. But now we note $z^3 - z + c\not \equiv 0$ (mod $3\mathcal{O}_{K}$) for every element $z\in \mathbb{F}_{3} \subset\mathcal{O}_{K}\slash 3\mathcal{O}_{K}$, and so $f(z)\not \equiv 0$ (mod $3\mathcal{O}_{K}$) for every point $z\in \mathbb{F}_{3} \subset\mathcal{O}_{K}\slash 3\mathcal{O}_{K}$. Otherwise, if $f(\alpha) \equiv 0$ (mod $3\mathcal{O}_{K}$) for some point $\alpha\in \mathcal{O}_{K}\slash 3\mathcal{O}_{K}\setminus \mathbb{F}_{3}$ and for every coefficient $c\not\in 3\mathcal{O}_{K}$. In this case, we then note that together with the three roots proved earlier, it then follows that the reduced cubic polynomial $f(x)$ modulo $3\mathcal{O}_{K}$ has in total more than three roots in $\mathcal{O}_{K}\slash 3\mathcal{O}_{K}$; and which then yields a contradiction. This then means that $f(x)=\varphi_{3,c}(x)-x$ has no roots in $\mathcal{O}_{K} / 3\mathcal{O}_{K}$ for every coefficient $c\not \in 3\mathcal{O}_{K}$; and so we then conclude $N_{c}(3) = 0$. This then completes the whole proof, as required.
\end{proof} 
We now wish to generalize Theorem \ref{2.1} to any polynomial map $\varphi_{p, c}$ for any given prime $p\geq 3$. More precisely, we prove that the number of distinct integral fixed points of any  $\varphi_{p, c}$ modulo $p\mathcal{O}_{K}$ is either $p$ or zero:

\begin{thm} \label{2.2}
Let $K\slash \mathbb{Q}$ be any number field of degree $ n \geq 2$ with the ring of integers $\mathcal{O}_{K}$, and in which any fixed prime integer $p\geq 3$ is inert. Let $\varphi_{p, c}$ be defined by $\varphi_{p, c}(z) = z^p + c$ for all $c, z\in\mathcal{O}_{K}$, and let  $N_{c}(p)$ be defined as in \textnormal{(\ref{N_{c}})}. Then $N_{c}(p) = p$ for every coefficient $c\in p\mathcal{O}_{K}$; otherwise $N_{c}(p) = 0$ for every point $c \not \in p\mathcal{O}_{K}$. 
\end{thm}
\begin{proof}
By applying a similar argument as in the Proof of Theorem \ref{2.1}, we then obtain the count as desired. That is, let $f(z)= z^p - z + c$ and note that for every coefficient $c\in p\mathcal{O}_{K}$, reducing $f(z)$ modulo  prime ideal $p\mathcal{O}_{K}$, it then follows $f(z)\equiv z^p - z$ (mod $p\mathcal{O}_{K}$); and so the reduced polynomial  $f(z)$ modulo $p\mathcal{O}_{K}$ is now a polynomial defined over a finite field $\mathcal{O}_{K}\slash p\mathcal{O}_{K}$ of order $p^{[K:\mathbb{Q}]} =p^n$. So now, by again the same fact mentioned earlier in the Proof of Theorem \ref{2.1} about subfields of finite fields, we again have the inclusion $\mathbb{F}_{p}\hookrightarrow \mathcal{O}_{K}\slash p\mathcal{O}_{K}$ of fields, where $\mathbb{F}_{p}$ is a field of order $p$; and moreover it is a well-known about polynomials over finite fields that the monic polynomial $h(x)= x^p -x$ vanishes at every point $z\in \mathbb{F}_{p}$; and so $z^p = z$ for every element $z\in \mathbb{F}_{p}\subset \mathcal{O}_{K}\slash p\mathcal{O}_{K}$. But now $f(z)\equiv 0$ (mod $p\mathcal{O}_{K}$) for every point $z\in \mathbb{F}_{p}\subset \mathcal{O}_{K}\slash p\mathcal{O}_{K}$. Moreover, since $f(x)$ modulo $p\mathcal{O}_{K}$ is of degree $p$ over a field $\mathcal{O}_{K}\slash p\mathcal{O}_{K}$ and so $f(x)$ modulo $p\mathcal{O}_{K}$ can have at most $p$ roots in $\mathcal{O}_{K}\slash p\mathcal{O}_{K}$(even counted with multiplicity), we then conclude $\#\{ z\in \mathcal{O}_{K} / p\mathcal{O}_{K} : \varphi_{p,c}(z) - z \equiv 0 \text{ (mod $p\mathcal{O}_{K}$)}\} = p$ and so $N_{c}(p) = p$. To see $N_{c}(p) = 0$ for every coefficient $c \not \in p\mathcal{O}_{K}$, we again first note that since $c\not\in p\mathcal{O}_{K}$, then $c\not \equiv 0$ (mod $p\mathcal{O}_{K}$). But now recall (as a fact) that $z^p - z = 0$ for every element $z\in \mathbb{F}_{p}\subset \mathcal{O}_{K}\slash p\mathcal{O}_{K}$, it then follows that $z^p - z + c\not \equiv 0$ (mod $p\mathcal{O}_{K}$) for every element $z\in \mathbb{F}_{p}\subset \mathcal{O}_{K}\slash p\mathcal{O}_{K}$; and so the reduced polynomial $f(z)\not \equiv 0$ (mod $p\mathcal{O}_{K}$) for every point $z\in \mathbb{F}_{p}\subset \mathcal{O}_{K}\slash p\mathcal{O}_{K}$. 
Otherwise, if $f(\alpha) \equiv 0$ (mod $p\mathcal{O}_{K}$) for some point $\alpha\in \mathcal{O}_{K}\slash p\mathcal{O}_{K}\setminus \mathbb{F}_{p}$ and for every coefficient $c\not\in p\mathcal{O}_{K}$, then by the same reasoning as in the Proof of the second part of Theorem \ref{2.1}, we then obtain a similar contradiction. This then means $f(x)=\varphi_{p,c}(x)-x$ has no roots in $\mathcal{O}_{K} / p\mathcal{O}_{K}$ for every coefficient $c\not \in p\mathcal{O}_{K}$; and so we then conclude $N_{c}(p) = 0$. This then completes the whole proof, as desired.
\end{proof}

Now we wish to generalize Theorem \ref{2.2} further to any $\varphi_{p^{\ell}, c}$ for any given prime $p\geq 3$ and integer $\ell \geq 1$. That is, we prove that the number of distinct integral fixed points of any $\varphi_{p^{\ell}, c}$ modulo $p\mathcal{O}_{K}$ is also $p$ or zero:

\begin{thm} \label{2.3}
Let $K\slash \mathbb{Q}$ be any number field of degree $ n \geq 2$ with ring $\mathcal{O}_{K}$, and in which any fixed prime $p\geq 3$ is inert. Let $\ell \in \mathbb{Z}_{\geq 1}$, and $\varphi_{p^{\ell}, c}$ be defined by $\varphi_{p^{\ell}, c}(z) = z^{p^{\ell}} + c$ for all $c, z\in\mathcal{O}_{K}$. Let $N_{c}(p)$ be as in \textnormal{(\ref{N_{c}})}. For any $c\in p\mathcal{O}_{K}$, we have $N_{c}(p) = p$ if $\ell \in \{1, p\}$ or $2\leq N_{c}(p) \leq \ell$ if $\ell \not \in \{1, p\}$; or else $N_{c}(p) = 0$ for all $c \not \in p\mathcal{O}_{K}$. 
\end{thm}

\begin{proof}
As before, let $f(z)= z^{p^{\ell}} - z + c$ and note that for every coefficient $c\in p\mathcal{O}_{K}$, reducing $f(z)$ modulo prime $p\mathcal{O}_{K}$, it then follows $f(z)\equiv z^{p^{\ell}} - z$ (mod $p\mathcal{O}_{K}$); and so $f(z)$ modulo $p\mathcal{O}_{K}$ is now a polynomial defined over a finite field $\mathcal{O}_{K}\slash p\mathcal{O}_{K}$. So now, as before we may recall the inclusion $\mathbb{F}_{p}\hookrightarrow \mathcal{O}_{K}\slash p\mathcal{O}_{K}$ of fields and also that $z^p = z$ for every $z\in \mathbb{F}_{p}\subset \mathcal{O}_{K}\slash p\mathcal{O}_{K}$; and so $z^{p^{\ell}} = (z^p)^{^\ell} = z^{\ell}$ for every $z\in \mathbb{F}_{p}$ and for every $\ell \in \mathbb{Z}_{\geq 1}$. It then follows $f(z)\equiv z^{\ell} - z$ (mod $p\mathcal{O}_{K}$) for every point $z\in \mathbb{F}_{p}\subset \mathcal{O}_{K}\slash p\mathcal{O}_{K}$ and for every $\ell \in \mathbb{Z}_{\geq 1}$. Now if $\ell = 1$ or $\ell = p$, then this yields  $f(z)\equiv z - z$ (mod $p\mathcal{O}_{K}$) or $f(z)\equiv z^p - z$ (mod $p\mathcal{O}_{K}$) for every $z\in \mathbb{F}_{p}$; and so $f(z)\equiv 0$ (mod $p\mathcal{O}_{K}$) for every point $z\in \mathbb{F}_{p}\subset\mathcal{O}_{K}\slash p\mathcal{O}_{K}$ and so we then conclude $N_{c}(p) = p$. Otherwise, if $\ell \in \mathbb{Z}^{+}\setminus\{1, p\}$ for any prime $p$, then we may first observe $z$ and $z-1$ are linear factors of $f(z)\equiv z(z-1)(z^{\ell-2}+z^{\ell-3}+\cdots +z+1)$ (mod $p\mathcal{O}_{K}$) for every point $z\in \mathbb{F}_{p}\subset \mathcal{O}_{K}\slash p\mathcal{O}_{K}$; and so $f(z)$ modulo $p\mathcal{O}_{K}$ vanishes at $z\equiv 0, 1$ (mod $p\mathcal{O}_{K}$) in $\mathcal{O}_{K}\slash p\mathcal{O}_{K}$. This then means $\#\{ z\in \mathcal{O}_{K} / p\mathcal{O}_{K} : \varphi_{p^{\ell},c}(z) -z \equiv 0 \ \text{(mod $p\mathcal{O}_{K}$)}\}\geq 2$ with a strict inequality depending on whether the other non-linear factor of $f(z)$ modulo $p\mathcal{O}_{K}$ vanishes or not on $\mathcal{O}_{K}\slash p\mathcal{O}_{K}$. Now since $\kappa(z)=z^{\ell-2}+z^{\ell-3}+\cdots +z+1$ (mod $p\mathcal{O}_{K}$) is a univariate polynomial of degree $\ell-2$ over a field $\mathcal{O}_{K}\slash p\mathcal{O}_{K}$, then $\kappa(z)$ has $\leq \ell-2$ roots in $\mathcal{O}_{K}\slash p\mathcal{O}_{K}$(even counted with multiplicity). But now we note that $2\leq \#\{ z\in \mathcal{O}_{K} / p\mathcal{O}_{K} : \varphi_{p^{\ell},c}(z) -z \equiv 0 \ \text{(mod $p\mathcal{O}_{K}$)}\}\leq (\ell-2) + 2 = \ell$ and so $2\leq N_{c}(p) \leq \ell$. Finally, we now show $N_{c}(p) = 0$ for every coefficient $c \not \in p\mathcal{O}_{K}$ and for every $\ell \in \mathbb{Z}_{\geq 1}$. To see this, let's for the sake of a contradiction, suppose $z^{p^{\ell}} - z + c \equiv 0$ (mod $p\mathcal{O}_{K}$) for some $z\in \mathcal{O}_{K} / p\mathcal{O}_{K}$ and for every $c \not \in p\mathcal{O}_{K}$ and every $\ell \geq 1$. But then observe that if $\ell = 1$ or $\ell = p$ and so $z^{p^{\ell}} - z =  0$ (mod $p\mathcal{O}_{K}$) for every $z\in \mathbb{F}_{p}\subset \mathcal{O}_{K} / p\mathcal{O}_{K}$, it then follows $c\equiv 0$ (mod $p\mathcal{O}_{K}$) and so $c\in p\mathcal{O}_{K}$; and from which we then obtain a contradiction. Otherwise, observe that if $\ell \in \mathbb{Z}^{+}\setminus\{1, p\}$, then since in this case we proved earlier that $2\leq N_{c}(p) \leq \ell$ whenever $c\in p\mathcal{O}_{K}$, then together with the above root of $f(z)$ modulo $p\mathcal{O}_{K}$ obtained when $c\not \in p\mathcal{O}_{K}$, it then follows that the total number of distinct roots of $f(z)\equiv z^{\ell}-z +c$ (mod $p\mathcal{O}_{K}$) can be more than $\ell=$ deg$(f(z)  \text{ mod } p\mathcal{O}_{K})$; which is clearly impossible, since every univariate nonzero polynomial over any field can have at most its degree roots in that field. It then overall follows $f(x)=\varphi_{p^{\ell},c}(x)-x$ has no roots in $\mathcal{O}_{K} / p\mathcal{O}_{K}$ for every coefficient $c\not \in p\mathcal{O}_{K}$ and for every $\ell \in \mathbb{Z}_{\geq 1}$; and so we then conclude $N_{c}(p) = 0$. This then completes the whole proof, as desired. 
\end{proof}

Finally, motivated by Narkiewicz's argument on existence of only fixed points of any $\varphi_{p^{\ell},c}$ over any real algebraic number field, we then restrict on the ordinary integers $\mathbb{Z}\subset \mathcal{O}_{K}$ and then get the following consequence of Theorem \ref{2.3} on the number of integral fixed points of any $\varphi_{p^{\ell},c}$ (mod $p$) for any prime $p\geq 3$ and any $\ell \in \mathbb{Z}_{\geq 1}$: 

\begin{cor} \label{cor2.4}
Let $p\geq 3$ be any fixed prime, and $\ell \geq 1$ be any integer. Let $\varphi_{p^{\ell}, c}$ be defined by $\varphi_{p^{\ell}, c}(z) = z^{p^{\ell}} + c$ for all $c, z\in\mathbb{Z}$, and let $N_{c}(p)$ be defined as in \textnormal{(\ref{N_{c}})} with $\mathcal{O}_{K} / p\mathcal{O}_{K}$ replaced with $\mathbb{Z}\slash p\mathbb{Z}$. Then for every coefficient $c=pt$, we have $N_{c}(p) = p$ if $\ell \in \{1,p\}$ or $2\leq N_{c}(p) \leq \ell$ if $\ell \in \mathbb{Z}^{+}\setminus\{1, p\}$; otherwise $N_{c}(p) = 0$ for any $c\neq pt$.
\end{cor}

\begin{proof}
By applying a similar argument as in the Proof of Theorem \ref{2.3}, we then obtain the count as desired. That is, let $f(z)= z^{p^{\ell}} - z + c$ and note that for every coefficient $c = pt$, reducing $f(z)$ modulo $p$, it then follows $f(z)\equiv z^{p^{\ell}} - z$ (mod $p$); and so $f(z)$ modulo $p$ is now a polynomial defined over a finite field $\mathbb{Z}\slash p\mathbb{Z}$ of order $p$. Now recall by Fermat's Little Theorem that $z^p = z$ and so $z^{p^{\ell}} = z^{\ell}$ for every $z\in \mathbb{Z}\slash p \mathbb{Z}$ and for every $\ell \in \mathbb{Z}_{\geq 1}$, it then follows $f(z)\equiv z^{\ell} - z$ (mod $p$) for every $z\in \mathbb{Z}\slash p\mathbb{Z}$. Now if $\ell = 1$ or $\ell = p$, then $f(z)\equiv z - z$ (mod $p$) or $f(z)\equiv z^p - z$ (mod $p$) for every $z\in \mathbb{Z}\slash p\mathbb{Z}$; and so $f(z)\equiv 0$ (mod $p$) for every $z\in \mathbb{Z}\slash p\mathbb{Z}$ and so $N_{c}(p) = p$. Otherwise, if $\ell \in \mathbb{Z}^{+}\setminus\{1, p\}$, then observe $z$ and $z-1$ are linear factors of $f(z)\equiv z(z-1)(z^{\ell-2}+z^{\ell-3}+\cdots +z+1)$ (mod $p$) for every $z\in \mathbb{Z}\slash p\mathbb{Z}$; and so $f(z)$ modulo $p$ vanishes at $z\equiv 0, 1$ (mod $p$). But now we note that $\#\{ z\in \mathbb{Z} / p\mathbb{Z} : \varphi_{p^{\ell},c}(z) - z \equiv 0 \ \text{(mod $p$)}\}\geq 2$ with a strict inequality depending on whether the non-linear factor of $f(z)$ modulo $p$ vanishes or not on $\mathbb{Z}\slash p\mathbb{Z}$. Moreover, since $\kappa(z)=z^{\ell-2}+z^{\ell-3}+\cdots +z+1$ (mod $p$) is a univariate polynomial of degree $\ell-2$ over a field $\mathbb{Z}\slash p\mathbb{Z}$, then $\kappa(z)$ has $\leq \ell-2$ roots in $\mathbb{Z}\slash p\mathbb{Z}$ (even counted with multiplicity). But now we note $2\leq \#\{ z\in \mathbb{Z}\slash p\mathbb{Z} : \varphi_{p^{\ell},c}(z) - z \equiv 0 \ \text{(mod $p$)}\}\leq (\ell-2) + 2 = \ell$ and so $2\leq N_{c}(p) \leq \ell$. Finally, we now show $N_{c}(p) = 0$ for every coefficient $c \neq pt$ and for every $\ell \in \mathbb{Z}_{\geq 1}$. Let's for the sake of a contradiction, suppose $z^{p^{\ell}} - z + c \equiv 0$ (mod $p$) for some point $z\in \mathbb{Z} / p\mathbb{Z}$ and for every $c \neq pt$ and for every $\ell \in \mathbb{Z}_{ \geq 1}$. But then if $\ell \in \{1,p\}$ and so $z^{p^{\ell}} - z =  0$ (mod $p$) for every $z\in \mathbb{Z}\slash p\mathbb{Z}$, it then follows $c\equiv 0$ (mod $p$); and so a contradiction. Otherwise, if $\ell \in \mathbb{Z}^{+}\setminus\{1, p\}$, then since in this case we proved earlier that $2\leq N_{c}(p) \leq \ell$ when $c = pt$, then together with the root of $f(z)$ modulo $p$ assumed when $c\neq pt$, it then follows that the total number of distinct roots of $f(z)\equiv z^{\ell}-z +c$ (mod $p$) can be more than $\ell=$ deg$(f(z)  \text{ mod } p)$; and which is not possible, since every univariate nonzero polynomial over any field can only have at most its degree roots in that field. It then overall follows $f(x)=\varphi_{p^{\ell},c}(x)-x$ has no roots in $\mathbb{Z}\slash p\mathbb{Z}$ for every coefficient $c\neq pt$ and for every $\ell \in \mathbb{Z}_{\geq 1}$; and hence we conclude $N_{c}(p) = 0$. This then completes the whole proof, as desired.
\end{proof}
 
\begin{rem}\label{re2.5}
With Theorem \ref{2.3} at our disposal, we may then to each distinct integral fixed point of $\varphi_{p^{\ell},c}$ associate an integral fixed orbit. So then, a dynamical translation of Theorem \ref{2.3} is the claim that the number of distinct integral fixed orbits that any $\varphi_{p^{\ell},c}$ has when iterated on the space $\mathcal{O}_{K} / p\mathcal{O}_{K}$ is $p$ or bounded between 2 and $\ell$ or is zero. As we've already mentioned in Introduction \ref{sec1} that the count obtained in Theorem \ref{2.3} may on one hand depend on $p$ or $\ell$ (and so depend on deg$(\varphi_{p^{\ell},c})$) and not on the degree $n=[K:\mathbb{Q}]$; and on the other hand, the count may be absolutely independent of $p$, $\ell$ (and hence absolutely independent of the degree of $\varphi_{p^{\ell},c}$) and $n$. Hence, we may then have that the function $N_{c}(p)\to \infty$ or $N_{c}(p)\in [2, \ell]$ or $N_{c}(p)\to 0$ as $p\to \infty$; a somewhat interesting phenomenon that we do also find in Remark \ref{re3.4} and \ref{re5.4} in Section \ref{sec3} and \ref{sec5}, respectively.
\end{rem}

\section{The Number of $\mathbb{Z}_{p}\slash p\mathbb{Z}_{p}$-Fixed Points of any Family of Polynomial Maps $\varphi_{p^{\ell},c}$}\label{sec3}

As in Section \ref{sec2}, we in this section also wish to count the number of distinct $p$-adic integral fixed points of any $\varphi_{p^{\ell},c}$ modulo prime ideal $p\mathbb{Z}_{p}$ for any given prime $p\geq 3$ and for any integer $\ell \geq 1$. As before, let $p\geq 3$ be any prime, $\ell \geq 1$ be any integer and $c\in \mathbb{Z}_{p}$ be any $p$-adic integer, and then define fixed point-counting function 
\begin{equation}\label{X_{c}}
X_{c}(p) := \#\bigg\{ z\in \mathbb{Z}_{p} / p\mathbb{Z}_{p} : \varphi_{p^{\ell},c}(z) - z \equiv 0 \ \text{(mod $p\mathbb{Z}_{p}$)}\bigg\}.
\end{equation}\noindent Again, setting $\ell =1$ and so $\varphi_{p^{\ell}, c} = \varphi_{p,c}$, we then first prove the following theorem and its generalization \ref{3.2}:

\begin{thm} \label{3.1}
Let $\varphi_{3, c}$ be a cubic map defined by $\varphi_{3, c}(z) = z^3 + c$ for all $c, z\in\mathbb{Z}_{3}$, and let $X_{c}(3)$ be the number defined as in \textnormal{(\ref{X_{c}})}. Then $X_{c}(3) = 3$ for any coefficient $c\in 3\mathbb{Z}_{3}$; otherwise $X_{c}(3) = 0$ for any coefficient $c \not \in 3\mathbb{Z}_{3}$.
\end{thm}
\begin{proof}
Let $f(z)= \varphi_{3, c}(z)-z = z^3 - z + c$ and note that for every coefficient $c\in 3\mathbb{Z}_{3}$, reducing $f(z)$ modulo prime ideal $3\mathbb{Z}_{3}$, it then follows that $f(z)\equiv z^3 - z$ (mod $3\mathbb{Z}_{3}$); from which it then also follows that the reduced polynomial $f(z)$ modulo $3\mathbb{Z}_{3}$ factors as $z(z-1)(z+1)$ over a finite field $\mathbb{Z}_{3}\slash 3\mathbb{Z}_{3}$ of order $3$. But now it then follows by the factor theorem that $z \equiv -1, 0, 1$ (mod $3\mathbb{Z}_{3}$) are roots of $f(z)$ modulo $3\mathbb{Z}_{3}$. Moreover, since the reduced univariate polynomial $f(x)$ modulo $3\mathbb{Z}_{3}$ is of degree $3$ over a field $\mathbb{Z}_{3}\slash 3\mathbb{Z}_{3}$ and so (from a well-known fact) the reduced polynomial $f(x)$ modulo $3\mathbb{Z}_{3}$ can only have at most three roots in $\mathbb{Z}_{3}\slash 3\mathbb{Z}_{3}$ (even counted with multiplicity), we then conclude $\#\{ z\in \mathbb{Z}_{3} / 3\mathbb{Z}_{3} : \varphi_{3,c}(z) - z \equiv 0 \ \text{(mod $3\mathbb{Z}_{3}$)}\} = 3$ and so $X_{c}(3) = 3$. To see $X_{c}(3) = 0$ for every coefficient $c \not \in 3\mathbb{Z}_{3}$, we first note that since the coefficient $c\not \in 3\mathbb{Z}_{3}$, then this also means $c\not \equiv 0$ (mod $3\mathbb{Z}_{3}$). But now recalling (as a fact) that $z^3=z$ for every element $z\in \mathbb{Z}_{3}/ 3\mathbb{Z}_{3}$, we then note that $z^3 - z + c\not \equiv 0$ (mod $3\mathbb{Z}_{3}$) for every element $z\in \mathbb{Z}_{3}/ 3\mathbb{Z}_{3}$; and so the reduced polynomial $f(z)\not \equiv 0$ (mod $3\mathbb{Z}_{3}$) for every point $z\in \mathbb{Z}_{3}/ 3\mathbb{Z}_{3}$. This then means that $f(x)=\varphi_{3,c}(x) - x$ has no roots in $\mathbb{Z}_{3}/ 3\mathbb{Z}_{3}$ for every coefficient $c\not \in 3\mathbb{Z}_{3}$; and so we then conclude $X_{c}(3) = 0$, as needed. This then completes the whole proof, as desired.
\end{proof} 
We now wish to generalize Theorem \ref{3.1} to any polynomial map $\varphi_{p, c}$ for any given prime $p\geq 3$. More precisely, we prove that the number of distinct $p$-adic integral fixed points of any $\varphi_{p, c}$ modulo $p\mathbb{Z}_{p}$ is $p$ or zero:

\begin{thm} \label{3.2}
Let $p\geq 3$ be any fixed prime, and $\varphi_{p, c}$ be defined by $\varphi_{p, c}(z) = z^p + c$ for all $c, z\in\mathbb{Z}_{p}$. Let $X_{c}(p)$ be defined as in \textnormal{(\ref{X_{c}})}. Then $X_{c}(p) = p$ for every coefficient $c\in p\mathbb{Z}_{p}$; otherwise $X_{c}(p) = 0$ for every point $c \not \in p\mathbb{Z}_{p}$.
\end{thm}
\begin{proof}
By applying a similar argument as in the Proof of Theorem \ref{3.1}, we then obtain the count as desired. That is, let $f(z)=\varphi_{p, c}(z)-z= z^p - z + c$ and note that for every coefficient $c\in p\mathbb{Z}_{p}$, reducing $f(z)$ modulo prime ideal $p\mathbb{Z}_{p}$, it then follows $f(z)\equiv z^p - z$ (mod $p\mathbb{Z}_{p}$); and so the reduced polynomial $f(z)$ modulo $p\mathbb{Z}_{p}$ is now a polynomial defined over a finite field $\mathbb{Z}_{p}\slash p\mathbb{Z}_{p}$ of order $p$. Now recall a well-known fact about polynomials over finite fields that the monic  polynomial $h(x) = x^p - x$ vanishes at every $z\in \mathbb{Z}_{p}\slash p\mathbb{Z}_{p}$; and so $z^p = z $  for every $z\in \mathbb{Z}_{p}\slash p\mathbb{Z}_{p}$. But then $f(z)\equiv 0$ (mod $p\mathbb{Z}_{p}$) for every $z\in \mathbb{Z}_{p}\slash p\mathbb{Z}_{p}$. Moreover, since $f(x)$ modulo $p\mathbb{Z}_{p}$ is of degree $p$ over a field $\mathbb{Z}_{p}\slash p\mathbb{Z}_{p}$ and so $f(x)$ modulo $p\mathbb{Z}_{p}$ can only have at most $p$ roots in $\mathbb{Z}_{p}\slash p\mathbb{Z}_{p}$ (even counted with multiplicity), we then conclude $\#\{ z\in \mathbb{Z}_{p}\slash p\mathbb{Z}_{p} : \varphi_{p,c}(z) - z \equiv 0 \ \text{(mod $p\mathbb{Z}_{p}$)}\} = p$ and so $X_{c}(p) = p$. To see $X_{c}(p) = 0$ for every coefficient $c \not \in p\mathbb{Z}_{p}$, we note that since $c\not \in p\mathbb{Z}_{p}$, then $c\not \equiv 0$ (mod $p\mathbb{Z}_{p}$). But now since $z^p = z$ for every $z\in \mathbb{Z}_{p}\slash p\mathbb{Z}_{p}$, we then note $z^p - z + c\not \equiv 0$ (mod $p\mathbb{Z}_{p}$) for every element $z\in \mathbb{Z}_{p}\slash p\mathbb{Z}_{p}$; and so $f(z)\not \equiv 0$ (mod $p\mathbb{Z}_{p}$) for every point $z\in \mathbb{Z}_{p}/ p\mathbb{Z}_{p}$. This then means $f(x)=\varphi_{p,c}(x) - x$ has no roots in $\mathbb{Z}_{p}/ p\mathbb{Z}_{p}$ for every coefficient $c\not \in p\mathbb{Z}_{p}$; and so we then conclude $X_{c}(p) = 0$. This then completes the whole proof, as desired.
\end{proof}

Finally, we now wish to generalize Theorem \ref{3.2} further to any $\varphi_{p^{\ell}, c}$ for any prime $p\geq 3$ and any $\ell \in \mathbb{Z}_{\geq 1}$. That is, we prove the number of distinct $p$-adic integral fixed points of any $\varphi_{p^{\ell}, c}$ modulo $p\mathbb{Z}_{p}$ is also $p$ or zero:

\begin{thm} \label{3.3}
Let $p\geq 3$ be any fixed prime integer, and $\ell \geq 1$ be any integer. Let $\varphi_{p^{\ell}, c}$ be defined by $\varphi_{p^{\ell}, c}(z) = z^{p^{\ell}}+c$ for all $c, z\in\mathbb{Z}_{p}$, and let $X_{c}(p)$ be the number as in \textnormal{(\ref{X_{c}})}. Then for every coefficient $c\in p\mathbb{Z}_{p}$, we have $X_{c}(p) = p$ if $\ell \in \{1,p\}$ or $2\leq X_{c}(p) \leq \ell$ if $\ell \in \mathbb{Z}^{+}\setminus\{1, p\}$; otherwise $X_{c}(p) = 0$ for every  point $c\not\in p\mathbb{Z}_{p}$.
\end{thm}

\begin{proof}
As before, let $f(z)= z^{p^{\ell}} - z + c$ and note that for every coefficient $c\in p\mathbb{Z}_{p}$, reducing $f(z)$ modulo $p\mathbb{Z}_{p}$, it then follows $f(z)\equiv z^{p^{\ell}} - z$ (mod $p\mathbb{Z}_{p}$); and so $f(z)$ modulo $p\mathbb{Z}_{p}$ is now a polynomial defined over a finite field $\mathbb{Z}_{p}\slash p\mathbb{Z}_{p}$. So now, as before observe  $f(z)\equiv z^{\ell} - z$ (mod $p\mathbb{Z}_{p}$) for every $z\in \mathbb{Z}_{p}\slash p\mathbb{Z}_{p}$ and for every $\ell \in \mathbb{Z}_{\geq 1}$, since $z^p = z$ and so $z^{p^{\ell}} = z^{\ell}$ for every element $z\in \mathbb{Z}_{p}\slash p\mathbb{Z}_{p}$. Now if $\ell = 1$ or $\ell = p$, then this yields $f(z)\equiv z - z$ (mod $p\mathbb{Z}_{p}$) or $f(z)\equiv z^p - z$ (mod $p\mathbb{Z}_{p}$) for every $z\in \mathbb{Z}_{p}\slash p\mathbb{Z}_{p}$; and so the reduced polynomial $f(z)\equiv 0$ (mod $p\mathbb{Z}_{p}$) for every point $z\in \mathbb{Z}_{p}\slash p\mathbb{Z}_{p}$ and so we then conclude $X_{c}(p) = p$. Otherwise, if exponent $\ell \in \mathbb{Z}^{+}\setminus\{1, p\}$, then observe $z$ and $z-1$ are linear factors of $f(z)\equiv z(z-1)(z^{\ell-2}+z^{\ell-3}+\cdots +z+1)$ (mod $p\mathbb{Z}_{p}$) for every point $z\in \mathbb{Z}_{p}\slash p\mathbb{Z}_{p}$; and so $f(z)$ modulo $p\mathbb{Z}_{p}$ vanishes at $z\equiv 0, 1$ (mod $p\mathbb{Z}_{p}$). This then means that $\#\{ z\in \mathbb{Z}_{p} / p\mathbb{Z}_{p} : \varphi_{p^{\ell},c}(z) - z \equiv 0 \ \text{(mod $p\mathbb{Z}_{p}$)}\}\geq 2$ with a strict inequality depending on whether the non-linear factor of $f(z)$ modulo $p\mathbb{Z}_{p}$ vanishes or not on $\mathbb{Z}_{p}\slash p\mathbb{Z}_{p}$. Now since $\kappa(z)=z^{\ell-2}+z^{\ell-3}+\cdots +z+1$ (mod $p\mathbb{Z}_{p}$) is a univariate polynomial of degree $\ell-2$ over a field $\mathbb{Z}_{p}\slash p\mathbb{Z}_{p}$, it then follows that $\kappa(z)$ has $\leq \ell-2$ roots in $\mathbb{Z}_{p}\slash p\mathbb{Z}_{p}$(even counted with multiplicity). But now together with the above two roots, we then note that $2\leq \#\{ z\in p\mathbb{Z}_{p}\slash p\mathbb{Z}_{p} : \varphi_{p^{\ell},c}(z) - z \equiv 0 \ \text{(mod $p\mathbb{Z}_{p}$)}\}\leq (\ell-2) + 2 = \ell$ and so $2\leq X_{c}(p) \leq \ell$. Finally, we now show $X_{c}(p) = 0$ for every coefficient $c \not \in p\mathbb{Z}_{p}$ and for every $\ell \in \mathbb{Z}_{\geq 1}$. To see this, let's for the sake of a contradiction, suppose $z^{p^{\ell}} - z + c \equiv 0$ (mod $p\mathbb{Z}_{p}$) for some point $z\in \mathbb{Z}_{p} / p\mathbb{Z}_{p}$ and for every $c \not \in p\mathbb{Z}_{p}$ and every $\ell \in \mathbb{Z}_{\geq 1}$. But then notice that if $\ell = 1$ or $\ell = p$ and so $z^{p^{\ell}} - z =  0$ (mod $p\mathbb{Z}_{p}$) for every $z\in \mathbb{Z}_{p}\slash p\mathbb{Z}_{p}$, it then follows that $c\equiv 0$ (mod $p\mathbb{Z}_{p}$) and so $c\in p\mathbb{Z}_{p}$; and hence yielding a contradiction. Otherwise, notice that if $\ell \in \mathbb{Z}^{+}\setminus\{1, p\}$, then since in this case we proved in the first part that $2\leq X_{c}(p) \leq \ell$ when $c\in p\mathbb{Z}_{p}$, then together with the above root of $f(z)$ modulo $p\mathbb{Z}_{p}$ assumed when $c\not \in p\mathbb{Z}_{p}$, it then follows that the total number of distinct roots of the reduced polynomial $f(z)\equiv z^{\ell}-z + c$ (mod $p\mathbb{Z}_{p}$) can be more than $\ell=$ deg$(f(z)  \text{ mod } p\mathbb{Z}_{p})$; and which is not possible, since every univariate nonzero polynomial over any field can only have at most its degree roots in that field. It then overall follows that $f(x)=\varphi_{p^{\ell},c}(x) - x$ has no roots in $\mathbb{Z}_{p}/ p\mathbb{Z}_{p}$ for every coefficient $c\not \in p\mathbb{Z}_{p}$ and for every $\ell \in \mathbb{Z}_{\geq 1}$; and so we then conclude $X_{c}(p) = 0$. This then completes the whole proof, as desired.
\end{proof}

\begin{rem}\label{re3.4}
With now Theorem \ref{3.3}, we may to each distinct $p$-adic integral fixed point of $\varphi_{p^{\ell},c}$ associate $p$-adic integral fixed orbit. In doing so, we then obtain a dynamical translation of Theorem \ref{3.3}, namely, that the number of distinct $p$-adic integral fixed orbits that any $\varphi_{p^{\ell},c}$ has when iterated on the space $\mathbb{Z}_{p} / p\mathbb{Z}_{p}$ is $p$ or bounded between $2$ and $\ell$ or zero. Again observe in Theorem \ref{3.3} that we've obtained a count that on one hand may depend on $p$ or $\ell$ (and so depend on deg$(\varphi_{p^{\ell},c})$); and on the other hand, it may be absolutely independent of $p$ and $\ell$ (and so independent of $p^{\ell}$). As a result, we may have $X_{c}(p)\to \infty$ or $X_{c}(p)\in [2, \ell]$ or $X_{c}(p)\to 0$ as $p\to \infty$; a somewhat interesting phenomenon coinciding precisely with what we remark(ed) about in Remark \ref{re2.5} and \ref{re5.4}, however, differing significantly from a phenomenon that we remark about in Remark \ref{re4.4} and \ref{re6.4}. 
\end{rem}

\section{On Number of $\mathbb{Z}_{p}\slash p\mathbb{Z}_{p}$-Fixed Points of any Family of Polynomial Maps $\varphi_{(p-1)^{\ell},c}$}\label{sec4}

As in Section \ref{sec3}, we in this section also wish to count the number of distinct $p$-adic integral fixed points of any $\varphi_{(p-1)^{\ell},c}$ modulo prime $p\mathbb{Z}_{p}$ for any given prime $p\geq 5$ and for any integer $\ell \geq 1$. As before, we let $p\geq 5$ be any prime, $\ell \geq 1$ be any integer and $c\in \mathbb{Z}_{p}$ be any $p$-adic integer, and then define fixed point-counting function 
\begin{equation}\label{Y_{c}}
Y_{c}(p) := \#\bigg\{ z\in \mathbb{Z}_{p} / p\mathbb{Z}_{p} : \varphi_{(p-1)^{\ell},c}(z) - z \equiv 0 \ \text{(mod $p\mathbb{Z}_{p}$)}\bigg\}.
\end{equation}\noindent Again, setting $\ell =1$ and so $\varphi_{(p-1)^{\ell}, c} = \varphi_{p-1,c}$, we first prove the following theorem and its generalization \ref{4.2}:

\begin{thm} \label{4.1}
Let $\varphi_{4, c}$ be defined by $\varphi_{4, c}(z) = z^4 + c$ for all $c, z\in\mathbb{Z}_{5}$, and $Y_{c}(5)$ be as in \textnormal{(\ref{Y_{c}})}. Then $Y_{c}(5) = 1$ or $2$ for any coefficient $c\equiv 1 \ (mod \ 5\mathbb{Z}_{5})$ or $c\in 5\mathbb{Z}_{5}$, resp.; otherwise $Y_{c}(5) = 0$ for any point $c\equiv -1\ (mod \ 5\mathbb{Z}_{5})$. 
\end{thm}
\begin{proof}
Let $g(z)= \varphi_{4,c}(z)-z = z^4 - z + c$ and note that for every coefficient $c\in 5\mathbb{Z}_{5}$, reducing $g(z)$ modulo prime ideal $5\mathbb{Z}_{5}$, we then obtain $g(z)\equiv z^4 - z$ (mod $5\mathbb{Z}_{5}$); and so the reduced polynomial $g(z)$ modulo $5\mathbb{Z}_{5}$ is now a polynomial defined over a finite field $\mathbb{Z}_{5}\slash 5\mathbb{Z}_{5}\cong \mathbb{Z}\slash 5\mathbb{Z}$ of order $5$. So now, recall from a well-known fact about polynomials over finite fields that the quartic monic polynomial $h(x)=x^4-1$ vanishes at every nonzero element in $\mathbb{Z}_{5}\slash 5\mathbb{Z}_{5}$; and so $z^4 =1$ for every nonzero element $z\in \mathbb{Z}_{5}\slash 5\mathbb{Z}_{5}$. But then $g(z)\equiv 1 - z$ (mod $5\mathbb{Z}_{5}$) for every nonzero point $z\in \mathbb{Z}_{5}\slash 5\mathbb{Z}_{5}$, and so $g(z)$ modulo $5\mathbb{Z}_{5}$ has a root in $\mathbb{Z}_{5}\slash 5\mathbb{Z}_{5}$, namely, $z\equiv 1$ (mod $5\mathbb{Z}_{5}$). Moreover, since $z$ is a linear factor of the reduced polynomial $g(z)\equiv z(z^3 - 1)$ (mod $5\mathbb{Z}_{5}$), it then follows $z\equiv 0$ (mod $5\mathbb{Z}_{5}$) is also a root of $g(z)$ modulo $5\mathbb{Z}_{5}$. But now we then conclude $\#\{ z\in \mathbb{Z}_{5} / 5\mathbb{Z}_{5} : \varphi_{4,c}(z) - z \equiv 0 \ \text{(mod $5\mathbb{Z}_{5}$)}\} = 2$ and so $Y_{c}(5) = 2$. To see $Y_{c}(5) = 1$ for every coefficient $c\equiv 1$ (mod $5\mathbb{Z}_{5}$), we note that with $c\equiv 1$ (mod $5\mathbb{Z}_{5}$) and since also $z^4 = 1$ for every $z\in \mathbb{Z}_{5}\slash 5\mathbb{Z}_{5}\setminus \{0\}$, then reducing $g(z)= \varphi_{4,c}(z)-z = z^4 - z + c$ modulo $5\mathbb{Z}_{5}$, it then follows $g(z)\equiv 2 - z$ (mod $5\mathbb{Z}_{5}$); and so $g(z)$ modulo $5\mathbb{Z}_{5}$ has a root in $\mathbb{Z}_{5}\slash 5\mathbb{Z}_{5}$, namely, $z\equiv 2$ (mod $5\mathbb{Z}_{5}$) and so we then conclude $Y_{c}(5) = 1$. Finally, to see $Y_{c}(5) = 0$ for every coefficient $c\equiv -1$ (mod $ 5\mathbb{Z}_{5})$, we note that since $c \equiv -1$ (mod $5\mathbb{Z}_{5}$) and since also (as a fact) $z^4 = 1$ for every element $z\in \mathbb{Z}_{5}\slash 5\mathbb{Z}_{5}\setminus \{0\}$, it then follows $z^4 - z + c\equiv -z$ (mod $5\mathbb{Z}_{5}$); and so the reduced polynomial $g(z) \equiv -z$ (mod $5\mathbb{Z}_{5}$). But now we note that $z\equiv 0$ (mod $5\mathbb{Z}_{5}$) is a root of $g(z)$ modulo $5\mathbb{Z}_{5}$ for every coefficient $c\equiv -1$ (mod $5\mathbb{Z}_{5}$) and for every coefficient $c\equiv 0$ (mod $5\mathbb{Z}_{5}$) as seen from the first part; and which is impossible, since $-1\not \equiv 0$ (mod $5\mathbb{Z}_{5}$). This then means that the quartic monic polynomial $g(x)=\varphi_{4,c}(x)-x$ has no roots in $\mathbb{Z}_{5} / 5\mathbb{Z}_{5}$ for every coefficient $c\equiv -1$ (mod $5\mathbb{Z}_{5}$); from which we then conclude $Y_{c}(5) = 0$, as also needed. This then  completes the whole proof, as required.
\end{proof} 
We now wish to generalize Theorem \ref{4.1} to any $\varphi_{p-1, c}$ for any given prime $p\geq 5$. More precisely, we prove that the number of distinct $p$-adic integral fixed points of every map  $\varphi_{p-1, c}$ modulo $p\mathbb{Z}_{p}$ is $1$ or $2$ or zero:

\begin{thm} \label{4.2}
Let $p\geq 5$ be any fixed prime integer, and let $\varphi_{p-1, c}$ be a polynomial map defined by $\varphi_{p-1, c}(z) = z^{p-1}+c$ for all $c, z\in\mathbb{Z}_{p}$. Let $Y_{c}(p)$ be the number defined as in \textnormal{(\ref{Y_{c}})}. Then $Y_{c}(p) = 1$ or $2$ for every coefficient $c\equiv 1 \ (mod \ p\mathbb{Z}_{p})$ or $c\in p\mathbb{Z}_{p}$, respectively; otherwise the number $Y_{c}(p) = 0$ for every coefficient $c\equiv -1\ (mod \ p\mathbb{Z}_{p})$.
\end{thm}
\begin{proof}
By applying a similar argument as in the Proof of Theorem \ref{4.1}, we then obtain the count as desired. That is, let $g(z)=\varphi_{p-1,c}(z)-z= z^{p-1} - z + c$ and note that for every coefficient $c\in p\mathbb{Z}_{p}$, reducing $g(z)$ modulo prime $p\mathbb{Z}_{p}$, it then follows $g(z)\equiv z^{p-1} - z$ (mod $p\mathbb{Z}_{p}$); and so the reduced polynomial $g(z)$ modulo $p\mathbb{Z}_{p}$ is now a polynomial defined over a finite field $\mathbb{Z}_{p}\slash p\mathbb{Z}_{p}$. Now recall a well-known fact about polynomials over finite fields that the monic polynomial $h(x) = x^{p-1}-1$ vanishes at every nonzero element $z\in \mathbb{Z}_{p}\slash p\mathbb{Z}_{p}$; and so $z^{p-1}=1$ for every element nonzero $z\in \mathbb{Z}_{p}\slash p\mathbb{Z}_{p}$. But then $g(z)\equiv 1 - z$ (mod $p\mathbb{Z}_{p}$) for every nonzero point $z\in \mathbb{Z}_{p}\slash p\mathbb{Z}_{p}$; and so $g(z)$ modulo $p\mathbb{Z}_{p}$ has a root in $\mathbb{Z}_{p}\slash p\mathbb{Z}_{p}$, namely, $z\equiv 1$ (mod $p\mathbb{Z}_{p}$). Moreover, since $z$ is a linear factor of $g(z)\equiv z(z^{p-2} - 1)$ (mod $p\mathbb{Z}_{p}$), it then follows $z\equiv 0$ (mod $p\mathbb{Z}_{p}$) is also a root of $g(z)$ modulo $p\mathbb{Z}_{p}$. But now we then conclude $\#\{ z\in \mathbb{Z}_{p} / p\mathbb{Z}_{p} : \varphi_{p-1,c}(z) - z \equiv 0 \ \text{(mod $p\mathbb{Z}_{p}$)}\} = 2$ and so $Y_{c}(p) = 2$. To see $Y_{c}(p) = 1$ for every coefficient $c\equiv 1$ (mod $p\mathbb{Z}_{p}$), we note that with $c\equiv 1$ (mod $p\mathbb{Z}_{p}$) and since also $z^{p-1} = 1$ for every $z\in (\mathbb{Z}_{p}\slash p\mathbb{Z}_{p})^{\times}=\mathbb{Z}_{p}\slash p\mathbb{Z}_{p}\setminus \{0\}$, reducing $g(z)= \varphi_{p-1,c}(z)-z = z^{p-1} - z + c$ modulo $p\mathbb{Z}_{p}$, it then follows $g(z)\equiv 2 - z$ (mod $p\mathbb{Z}_{p}$) and so $g(z)$ modulo $p\mathbb{Z}_{p}$ has a root in $\mathbb{Z}_{p}\slash p\mathbb{Z}_{p}$, namely, $z\equiv 2$ (mod $p\mathbb{Z}_{p}$); and so $Y_{c}(p) = 1$. Finally, to see $Y_{c}(p) = 0$ for every coefficient $c\equiv -1$ (mod  $p\mathbb{Z}_{p})$, we note that since $c \equiv -1$ (mod $p\mathbb{Z}_{p}$) and since also $z^{p-1} = 1$ for every $z\in (\mathbb{Z}_{p}\slash p\mathbb{Z}_{p})^{\times}$, it then follows $z^{p-1} - z + c\equiv -z$ (mod $p\mathbb{Z}_{p}$); and so $g(z) \equiv -z$ (mod $p\mathbb{Z}_{p}$). But now notice $z\equiv 0$ (mod $p\mathbb{Z}_{p}$) is a root of $g(z)$ modulo $p\mathbb{Z}_{p}$ for every coefficient $c\equiv -1$ (mod $p\mathbb{Z}_{p}$) and for every coefficient $c\equiv 0$ (mod $p\mathbb{Z}_{p}$) as seen from the first part; and which is not possible, since $-1\not \equiv 0$ (mod $p\mathbb{Z}_{p}$). This then means that $g(x)=\varphi_{p-1,c}(x)-x$ has no roots in $\mathbb{Z}_{p} / p\mathbb{Z}_{p}$ for every coefficient $c\equiv -1$ (mod $p\mathbb{Z}_{p}$); and so we then conclude $Y_{c}(p) = 0$. This then completes the whole proof, as desired.
\end{proof}

Finally, we wish to generalize Theorem \ref{4.2} further to any $\varphi_{(p-1)^{\ell}, c}$ for any prime $p\geq 5$ and any $\ell\in \mathbb{Z}_{\geq 1}$. That is, we prove the number of distinct $p$-adic integral fixed points of any $\varphi_{(p-1)^{\ell}, c}$ modulo $p\mathbb{Z}_{p}$ is $1$ or $2$ or $0$:

\begin{thm} \label{4.3}
Let $p\geq 5$ be any fixed prime integer, and $\ell \geq 1$ be any integer. Let $\varphi_{(p-1)^{\ell}, c}$ be defined by $\varphi_{(p-1)^{\ell}, c}(z) = z^{(p-1)^{\ell}}+c$ for all $c, z\in\mathbb{Z}_{p}$. Let $Y_{c}(p)$ be the number as in \textnormal{(\ref{Y_{c}})}. Then $Y_{c}(p) = 1$ or $2$ for every coefficient $c\equiv 1 \ (mod \ p\mathbb{Z}_{p})$ or $c\in p\mathbb{Z}_{p}$, respectively; otherwise we have $Y_{c}(p) = 0$ for every $c\equiv -1\ (mod \ p\mathbb{Z}_{p})$.
\end{thm}

\begin{proof}
By applying a similar argument as in the Proof of Theorem \ref{4.2}, we then obtain the count as desired. As before, let $g(z)= z^{(p-1)^{\ell}} - z + c$ and note that for every coefficient $c\in p\mathbb{Z}_{p}$, reducing $g(z)$ modulo $p\mathbb{Z}_{p}$, we then obtain $g(z)\equiv z^{(p-1)^{\ell}} - z$ (mod $p\mathbb{Z}_{p}$); and so the reduced polynomial $g(z)$ modulo $p\mathbb{Z}_{p}$ is now a polynomial defined over a finite field $\mathbb{Z}_{p}\slash p\mathbb{Z}_{p}$. Now since $h(x)= x^{p-1}-1$ vanishes at every $z\in (\mathbb{Z}_{p}\slash p\mathbb{Z}_{p})^{\times}$ and so $z^{p-1} = 1$ for every element $z\in (\mathbb{Z}_{p}\slash p\mathbb{Z}_{p})^{\times}$, it then follows $z^{(p-1)^{\ell}} = 1$ for every $z\in (\mathbb{Z}_{p}\slash p\mathbb{Z}_{p})^{\times}$ and for every $\ell \in \mathbb{Z}_{\geq 1}$. But now $g(z)\equiv 1 - z$ (mod $p\mathbb{Z}_{p}$) for every point $z\in (\mathbb{Z}_{p}\slash p\mathbb{Z}_{p})^{\times}$, and so $g(z)$ modulo $p\mathbb{Z}_{p}$ has a root in $\mathbb{Z}_{p}\slash p\mathbb{Z}_{p}$, namely, $z\equiv 1$ (mod $p\mathbb{Z}_{p}$). Moreover, since $z$ is a linear factor of $g(z)\equiv z(z^{(p-1)^{\ell}-1} - 1)$ (mod $p\mathbb{Z}_{p}$), it then follows $z\equiv 0$ (mod $p\mathbb{Z}_{p}$) is also a root of $g(z)$ modulo $p\mathbb{Z}_{p}$. But now we then conclude $\#\{ z\in \mathbb{Z}_{p} / p\mathbb{Z}_{p} : \varphi_{(p-1)^{\ell},c}(z) - z \equiv 0 \ \text{(mod $p\mathbb{Z}_{p}$)}\} = 2$ and so $Y_{c}(p) = 2$. To see $Y_{c}(p) = 1$ for every coefficient $c\equiv 1$ (mod $p\mathbb{Z}_{p}$) and for every $\ell \in \mathbb{Z}_{\geq 1}$, we note that since $c\equiv 1$ (mod $p\mathbb{Z}_{p}$) and since also $z^{(p-1)^{\ell}} = 1$ for every element $z\in (\mathbb{Z}_{p}\slash p\mathbb{Z}_{p})^{\times}$ and for every integer $\ell \geq 1$, reducing $g(z)= z^{(p-1)^{\ell}} - z + c$ modulo $p\mathbb{Z}_{p}$, it then follows $g(z)\equiv 2 - z$ (mod $p\mathbb{Z}_{p}$) and so $g(z)$ modulo $p\mathbb{Z}_{p}$ has a root in $\mathbb{Z}_{p}\slash p\mathbb{Z}_{p}$, namely, $z\equiv 2$ (mod $p\mathbb{Z}_{p}$); and so we then conclude $Y_{c}(p) = 1$. Finally, to see $Y_{c}(p) = 0$ for every coefficient $c\equiv -1$ (mod $ p\mathbb{Z}_{p})$ and for every $\ell \in \mathbb{Z}_{\geq 1}$, we note that since $c \equiv -1$ (mod $p\mathbb{Z}_{p}$) and since also $z^{(p-1)^{\ell}} = 1$ for every element $z\in (\mathbb{Z}_{p}\slash p\mathbb{Z}_{p})^{\times}$ and for every $\ell \in \mathbb{Z}_{\geq 1}$, it then follows $z^{(p-1)^{\ell}} - z + c\equiv -z$ (mod $p\mathbb{Z}_{p}$); and so $g(z) \equiv -z$ (mod $p\mathbb{Z}_{p}$). But now we note that $z\equiv 0$ (mod $p\mathbb{Z}_{p}$) is a root of $g(z)$ modulo $p\mathbb{Z}_{p}$ for every coefficient $c\equiv -1$ (mod $p\mathbb{Z}_{p}$) and for every coefficient $c\equiv 0$ (mod $p\mathbb{Z}_{p}$) as seen from the first part; and which is not possible, since again $-1\not \equiv 0$ (mod $p\mathbb{Z}_{p}$). This then means that $g(x)=\varphi_{(p-1)^{\ell},c}(x)-x$ has no roots in $\mathbb{Z}_{p} \slash p\mathbb{Z}_{p}$ for every coefficient $c\equiv -1$ (mod $p\mathbb{Z}_{p}$) and for every $\ell \in \mathbb{Z}_{\geq 1}$; from which we then conclude $Y_{c}(p) = 0$. This then completes the whole proof, as desired.
\end{proof}

\begin{rem}\label{re4.4}
With now Theorem \ref{4.3}, we may then to each distinct $p$-adic integral fixed point of $\varphi_{(p-1)^{\ell},c}$ associate $p$-adic integral fixed orbit. In doing so, we then also obtain a dynamical translation of Theorem \ref{4.3}, namely, that the number of distinct $p$-adic integral fixed orbits that any $\varphi_{(p-1)^{\ell},c}$ has when iterated on the space $\mathbb{Z}_{p} / p\mathbb{Z}_{p}$ is $1$ or $2$ or $0$. Furthermore, as noted in Intro.\ref{sec1} that the count obtained in each of the cases $c\equiv 1, 0, -1$ (mod $p\mathbb{Z}_{p})$ in Theorem \ref{4.3} is independent of $p$ and $\ell$ (and so independent of deg$(\varphi_{(p-1)^{\ell},c})$). Moreover, we may also notice that the expected total count (namely, $1 + 2 + 0 =3$) of distinct $p$-adic integral fixed points (fixed orbits) in the whole family of maps $\varphi_{(p-1)^{\ell},c}$ modulo $p\mathbb{Z}_{p}$ is not only also independent of $p$ and $\ell$ (and so independent of deg$(\varphi_{(p-1)^{\ell},c})$), but is also a constant $3$ even as $(p-1)^{\ell}\to \infty$; a somewhat interesting phenomenon coinciding precisely with what we remark(ed) about in [\cite{BK2}, Remark 3.5] and also here in Remark \ref{re6.4}, however, differing significantly from a phenomenon that we remark(ed) about in Remark \ref{re2.5}, \ref{re3.4} and \ref{re5.4}. 
\end{rem}

\section{The Number of $\mathbb{F}_{p}[t]\slash (\pi)$-Fixed Points of any Family of Polynomial Maps $\varphi_{p^{\ell},c}$}\label{sec5}

As in Section \ref{sec2} and \ref{sec3}, we in section also wish to count the number of distinct fixed points of any $\varphi_{p^{\ell},c}$ modulo fixed irreducible monic polynomial $\pi \in \mathbb{F}_{p}[t]$ of degree $m\geq 1$, where $p\geq 3$ is any given prime and $\ell \geq 1$ is any integer. To do so, we again let $p\geq 3$ be any prime, $\ell \geq 1$ be any integer, $c\in \mathbb{F}_{p}[t]$ be any point and $\pi\in \mathbb{F}_{p}[t]$ be any fixed irreducible monic polynomial of degree $m\geq 1$, and then define fixed point-counting function
\begin{equation}\label{N_{ct}}
N_{c(t)}(\pi, p) := \#\bigg\{ z\in \mathbb{F}_{p}[t]\slash (\pi) : \varphi_{p^{\ell},c}(z) - z \equiv 0 \ \text{(mod $\pi$)}\bigg\}.
\end{equation} \noindent Again, setting $\ell =1$ and so $\varphi_{p^{\ell}, c} = \varphi_{p,c}$, we then first prove the following theorem and its generalization \ref{5.2}:
\begin{thm} \label{5.1}
Let $\varphi_{3, c}$ be a cubic map defined by $\varphi_{3, c}(z) = z^3 + c$ for all $c, z\in\mathbb{F}_{3}[t]$, and let $N_{c(t)}(\pi, 3)$ be defined as in \textnormal{(\ref{N_{ct}})}. Then $N_{c(t)}(\pi, 3) = 3$ for every coefficient $c\in (\pi)$; otherwise $N_{c(t)}(\pi, 3) = 0$ for every $c \not \in (\pi)$.
\end{thm}
\begin{proof}
Let $f_{c(t)}(z)= \varphi_{3, c}(z)-z = z^3 - z + c$ and note that for every coefficient $c$ in a maximal ideal $(\pi) := \pi\mathbb{F}_{3}[t]$, reducing $f_{c(t)}(z)$ modulo prime $\pi$, we then obtain $f_{c(t)}(z)\equiv z^3 - z$ (mod $\pi$); and from which it then also follows $f_{c(t)}(z)$ modulo $\pi$ factors as $z(z-1)(z+1)$ over a finite field $\mathbb{F}_{3}[t] / (\pi)$ of order $3^{\text{deg}(\pi)} = 3^m$. But now it then follows by the factor theorem that $z\equiv -1, 0, 1$ (mod $\pi$) are roots of $f_{c(t)}(z)$ modulo $\pi$ in $\mathbb{F}_{3}[t] / (\pi)$; and so every $\mathbb{F}_{3}[t]$-solution of $f_{c(t)}(z) = 0$ is of the form $z_{1} = \pi s -1$, $z_{2} = \pi r$ and $z_{3} = \pi u+1$ for some elements $s, r, u\in \mathbb{F}_{3}[t]$. Moreover, since the reduced univariate polynomial $f_{c(t)}(x)$ modulo $\pi$ is of degree $3$ over a field $\mathbb{F}_{3}[t] / (\pi)$ and so $f_{c(t)}(x)$ modulo $\pi$ can only have at most $3$ roots in $\mathbb{F}_{3}[t] / (\pi)$ (even counted with multiplicity), we then conclude $\#\{ z\in \mathbb{F}_{3}[t]\slash (\pi) : \varphi_{3,c}(z) - z \equiv 0 \ \text{(mod $\pi$)}\} = 3$ and so $N_{c(t)}(\pi, 3) = 3$. To see $N_{c(t)}(\pi, 3) = 0$ for every coefficient $c \not \in (\pi)$, we first note that since the coefficient $c\not\in (\pi)$, then $c\not \equiv 0$ (mod $\pi$). So now, recall from a well-known fact about subfields of finite fields that every subfield of $\mathbb{F}_{3}[t] / (\pi)$ is of order $3^r$ for some positive integer $r\mid m$, we then have the inclusion $\mathbb{F}_{3} \hookrightarrow\mathbb{F}_{3}[t] / (\pi)$ of fields, where $\mathbb{F}_{3}$ is a field of order $3$. Moreover, since (as a fact) $z^3 - z = 0$ for every $z\in \mathbb{F}_{3}\subset \mathbb{F}_{3}[t]\slash (\pi)$, it then follows $z^3 - z + c\not \equiv 0$ (mod $\pi$) for every $z\in \mathbb{F}_{3}\subset \mathbb{F}_{3}[t]\slash (\pi)$; and so the reduced polynomial $f_{c(t)}(z)\not \equiv 0$ (mod $\pi$) for every point $z\in \mathbb{F}_{3}\subset \mathbb{F}_{3}[t]\slash (\pi)$. Otherwise, if $f_{c(t)}(\alpha) \equiv 0$ (mod $\pi$) for some  point $\alpha\in \mathbb{F}_{3}[t]\slash (\pi)\setminus \mathbb{F}_{3}$ and for every $c\not\in (\pi)$. Then we note that together with the three roots proved in the first part, it then follows $f_{c(t)}(x)$ modulo $\pi$ has in total more than three roots in $\mathbb{F}_{3}[t]\slash (\pi)$; and which then yields a contradiction. This then means  $f_{c(t)}(x)=\varphi_{3,c}(x) - x$ has no roots in $\mathbb{F}_{3}[t] / (\pi)$ for every coefficient $c\not \in (\pi)$; and hence we conclude $N_{c(t)}(\pi, 3) = 0$. This then completes the whole proof, as desired.
\end{proof} 
We now wish to generalize Theorem \ref{5.1} to any polynomial map $\varphi_{p, c}$ for any given prime integer $p\geq 3$. More precisely, we prove that the number of distinct $\mathbb{F}_{p}[t]$-fixed points of any $\varphi_{p, c}$ modulo $\pi$ is either $p$ or zero:

\begin{thm} \label{5.2}
Let $p\geq 3$ be any fixed prime integer, and consider a family of polynomial maps $\varphi_{p, c}$ defined by $\varphi_{p, c}(z) = z^p + c$ for all points $c, z\in\mathbb{F}_{p}[t]$. Let $N_{c(t)}(\pi, p)$ be the number defined as in \textnormal{(\ref{N_{ct}})}. Then the number $N_{c(t)}(\pi, p)=p$ for every coefficient $c\in (\pi)$; otherwise the number $N_{c(t)}(\pi, p) = 0$ for every coefficient $c \not \in (\pi)$.
\end{thm}
\begin{proof}
By applying a similar argument as in the Proof of Theorem \ref{5.1}, we then obtain the count as desired. That is, let $f_{c(t)}(z)= z^p - z + c$ and note that for every coefficient $c\in (\pi) := \pi \mathbb{F}_{p}[t]$, reducing $f_{c(t)}(z)$ modulo prime $\pi$, it then follows $f_{c(t)}(z)\equiv z^p - z$ (mod $\pi$); and so the reduced polynomial $f_{c(t)}(z)$ modulo $\pi$ is now a polynomial defined over a field $\mathbb{F}_{p}[t] / (\pi)$ of order $p^{\text{deg}(\pi)} = p^m$. So now, as before we may recall the inclusion $\mathbb{F}_{p} \hookrightarrow\mathbb{F}_{p}[t] / (\pi)$ of fields and also recall that $z^p = z $ for every $z\in \mathbb{F}_{p}\subset \mathbb{F}_{p}[t] / (\pi)$, it then follows $f_{c(t)}(z)\equiv 0$ (mod $\pi$) for every point $z\in \mathbb{F}_{p}\subset \mathbb{F}_{p}[t] / (\pi)$. Moreover, since the reduced univariate polynomial $f_{c(t)}(x)$ modulo $\pi$ is of degree $p$ over a field $\mathbb{F}_{p}[t] / (\pi)$ and so $f_{c(t)}(x)$ modulo $\pi$ can only have $\leq p$ roots in $\mathbb{F}_{p}[t] / (\pi)$ (even counted with multiplicity), we then conclude $\#\{ z\in \mathbb{F}_{p}[t]\slash (\pi) : \varphi_{p,c}(z) - z \equiv 0 \text{ (mod $\pi$)}\} = p$ and so $N_{c(t)}(\pi, p)=p$. To see $N_{c(t)}(\pi, p) = 0$ for every coefficient $c \not \in (\pi)$, we again first note that since $c\not \in (\pi)$, then $c\not \equiv 0$ (mod $\pi$). But now recall (as a fact) that $z^p - z = 0$ for every $z\in \mathbb{F}_{p}\subset \mathbb{F}_{p}[t] / (\pi)$, it then follows $z^p - z + c\not \equiv 0$ (mod $\pi$) for every element $z\in \mathbb{F}_{p}\subset \mathbb{F}_{p}[t] / (\pi)$; and so the reduced polynomial $f_{c(t)}(z)\not \equiv 0$ (mod $\pi$) for every point $z\in \mathbb{F}_{p}[t] / (\pi)$. Otherwise if $f(\alpha) \equiv 0$ (mod $\pi$) for some point $\alpha\in \mathbb{F}_{p}[t] / (\pi)\setminus \mathbb{F}_{p}$ and for every $c\not\in (\pi)$, then by applying here the same reasoning as in the Proof of the second part of Theorem \ref{5.1}, we then obtain a similar contradiction also in this case. This then means that $f_{c(t)}(x)=\varphi_{p,c}(x) - x$ has no roots in $\mathbb{F}_{p}[t] / (\pi)$ for every coefficient $c\not \in (\pi)$; and hence we conclude $N_{c(t)}(\pi, p) = 0$. This then completes the whole proof, as desired.
\end{proof}

Finally, we wish to generalize Theorem \ref{5.2} further to any $\varphi_{p^{\ell}, c}$ for any prime integer $p\geq 3$ and any integer $\ell \geq 1$. That is, we prove that the number of distinct fixed points of any $\varphi_{p^{\ell}, c}$ modulo $\pi$ is also $p$ or zero:

\begin{thm} \label{5.3}
Let $p\geq 3$ be any fixed prime, and $\ell \geq 1$ be any integer. Consider a family of polynomial maps $\varphi_{p^{\ell}, c}$ defined by $\varphi_{p^{\ell}, c}(z) = z^{p^{\ell}}+c$ for all $c, z\in\mathbb{F}_{p}[t]$, and let $N_{c(t)}(\pi, p)$ be as in \textnormal{(\ref{N_{ct}})}. For every $c\in (\pi)$, we have $N_{c(t)}(\pi, p) = p$ if $\ell \in \{1, p\}$ or $2\leq N_{c(t)}(\pi, p) \leq \ell$ if $\ell \in \mathbb{Z}^{+}\setminus\{1, p\}$; otherwise $N_{c(t)}(\pi, p) = 0$ for every $c \not \in (\pi)$. 
\end{thm}
\begin{proof}
As before, let $f_{c(t)}(z)= z^{p^{\ell}} - z + c$ and note that for every coefficient $c\in (\pi)$, reducing $f_{c(t)}(z)$ modulo  prime $\pi$, it then follows $f_{c(t)}(z)\equiv z^{p^{\ell}} - z$ (mod $\pi$); and so $f_{c(t)}(z)$ modulo $\pi$ is now a polynomial defined over a finite field $\mathbb{F}_{p}[t]\slash (\pi)$. So now, as before we may recall $\mathbb{F}_{p}\hookrightarrow \mathbb{F}_{p}[t]\slash (\pi)$ of fields, and also that $z^p = z$ for every $z\in \mathbb{F}_{p}\subset \mathbb{F}_{p}[t]\slash (\pi)$, it then follows $z^{p^{\ell}} = (z^p)^{^\ell} = z^{\ell}$ for every $z\in \mathbb{F}_{p}$ and for every $\ell \in \mathbb{Z}_{\geq 1}$. But then we note $f_{c(t)}(z)\equiv z^{\ell} - z$ (mod $\pi$) for every $z\in \mathbb{F}_{p}\subset \mathbb{F}_{p}[t]\slash (\pi)$ and for every $\ell \in \mathbb{Z}_{\geq 1}$. Now if $\ell = 1$ or $\ell = p$, then this yields $f_{c(t)}(z)\equiv z - z$ (mod $\pi$) or $f_{c(t)}(z)\equiv z^p - z$ (mod $\pi$) for every $z\in \mathbb{F}_{p}\subset \mathbb{F}_{p}[t]\slash (\pi)$; and so $f_{c(t)}(z)\equiv 0$ (mod $\pi$) for every point $z\in \mathbb{F}_{p}\subset \mathbb{F}_{p}[t]\slash (\pi)$ and so we then conclude $N_{c(t)}(\pi, p) = p$. Otherwise, if $\ell \in \mathbb{Z}^{+}\setminus\{1, p\}$ for any prime $p$, then observe $z$ and $z-1$ are linear factors of $f_{c(t)}(z)\equiv z(z-1)(z^{\ell-2}+z^{\ell-3}+\cdots +z+1)$ (mod $\pi$) for every $z\in \mathbb{F}_{p}\subset \mathbb{F}_{p}[t]\slash (\pi)$; and so $f_{c(t)}(z)$ modulo $\pi$ vanishes at $z\equiv 0, 1$ (mod $\pi$) in $\mathbb{F}_{p}[t]\slash (\pi)$. This then means $\#\{ z\in \mathbb{F}_{p}[t]\slash (\pi) : \varphi_{p^{\ell},c}(z) -z \equiv 0 \ \text{(mod $\pi$)}\}\geq 2$ with a strict inequality depending on whether the other non-linear factor of $f_{c(t)}(z)$ modulo $\pi$ vanishes or not on $\mathbb{F}_{p}[t]\slash (\pi)$. Now since $\kappa(z)=z^{\ell-2}+z^{\ell-3}+\cdots +z+1$ (mod $\pi$) is a polynomial of degree $\ell-2$ over a field $\mathbb{F}_{p}[t]\slash (\pi)$, then $\kappa(z)$ has $\leq \ell-2$ roots in $\mathbb{F}_{p}[t]\slash (\pi)$(even counted with multiplicity). But now we note $2\leq \#\{ z\in \mathbb{F}_{p}[t]\slash (\pi) : \varphi_{p^{\ell},c}(z) -z \equiv 0 \ \text{(mod $\pi$)}\}\leq (\ell-2) + 2 = \ell$ and so $2\leq N_{c(t)}(\pi, p) \leq \ell$. Finally, we now show $N_{c(t)}(\pi, p) = 0$ for every coefficient $c \not \in (\pi)$ and every $\ell \in \mathbb{Z}_{\geq 1}$. To see this, let's for the sake of a contradiction, suppose $z^{p^{\ell}} - z + c \equiv 0$ (mod $\pi$) for some point $z\in \mathbb{F}_{p}[t]\slash (\pi)$ and for every $c \not \in (\pi)$ and every $\ell \in \mathbb{Z}_{\geq 1}$. But then we may observe that if $\ell = 1$ or $\ell = p$ and so $z^{p^{\ell}} - z =  0$ (mod $\pi$) for every $z\in \mathbb{F}_{p}\subset \mathbb{F}_{p}[t]\slash (\pi)$, it then follows that $c\equiv 0$ (mod $\pi$) and so $c\in (\pi)$; which then yields a contradiction. Otherwise, if $\ell \in \mathbb{Z}^{+}\setminus\{1, p\}$, then since in this case we proved earlier that $2\leq N_{c(t)}(\pi, p) \leq \ell$ when $c\in (\pi)$, then together with the above root of $f_{c(t)}(z)$ modulo $\pi$ assumed when $c\not \in (\pi)$, it then follows that overall the total number of distinct roots of $f_{c(t)}(z)\equiv z^{\ell}-z +c$ (mod $\pi$) can be more than $\ell=$ deg$(f_{c(t)}(z)  \text{ mod } \pi)$; which is impossible, since every univariate nonzero polynomial over any field can have at most its degree roots in that field. This then overall means $f_{c(t)}(x)=\varphi_{p^{\ell},c}(x) - x$ has no roots in $\mathbb{F}_{p}[t] / (\pi)$ for every coefficient $c\not \in (\pi)$ and for every $\ell \in \mathbb{Z}_{\geq 1}$; and so we then conclude $N_{c(t)}(\pi, p) = 0$. This then completes the whole proof, as desired
\end{proof}

\begin{rem}\label{re5.4}
As before, with now Theorem \ref{5.3}, we may then to each distinct fixed point of $\varphi_{p^{\ell},c}$ associate a fixed orbit. In doing so, we then obtain a dynamical translation of Theorem \ref{5.3}, namely, the claim that the number of distinct fixed orbits that any $\varphi_{p^{\ell},c}$ has when iterated on the space $\mathbb{F}_{p}[t] / (\pi)$ is $p$ or bounded between $2$ and $\ell$ or zero. Again, as we've already noted in Introduction \ref{sec1} that the count obtained in Theorem \ref{5.3} may on one hand depend on $p$ or $\ell$ (and so depend on deg$(\varphi_{p^{\ell},c})$) and not on $m=$ deg$(\pi)$; and on the other hand, the count obtained in Theorem \ref{5.3} may be independent of $p, \ell$ (and hence independent of deg$(\varphi_{p^{\ell},c})$) and $m$. Consequently, we may as before have that $N_{c(t)}(\pi, p)\to \infty$ or $N_{c(t)}(\pi, p)\in [2, \ell]$ or $N_{c(t)}(\pi, p)\to 0$ as $p\to \infty$; a somewhat interesting phenomenon coinciding precisely with what we remarked about earlier in Remark \ref{re2.5} and \ref{re3.4}, however, differing very significantly from the phenomenon that remark(ed) about in Rem. \ref{re4.4} and \ref{re6.4}.
\end{rem}

\section{Number of $\mathbb{F}_{p}[t]\slash (\pi)$-Fixed Points of any Family of Polynomial Maps $\varphi_{(p-1)^{\ell},c}$}\label{sec6}

As in Section \ref{sec4} and \ref{sec5}, we in this section also wish to again count the number of distinct fixed points of any $\varphi_{(p-1)^{\ell},c}$ modulo fixed irreducible monic $\pi \in \mathbb{F}_{p}[t]$ of degree $m\geq 1$, where $p\geq 5$ is any prime and $\ell \geq 1$ is any integer. As before, we again let $p\geq 5$ be any prime, $\ell \geq 1$ be any integer, $c\in \mathbb{F}_{p}[t]$ be any point and $\pi\in \mathbb{F}_{p}[t]$ be any fixed irreducible monic polynomial of degree $m\geq 1$, and then define fixed point-counting function
\begin{equation}\label{M_{c}}
M_{c(t)}(\pi, p) := \#\bigg\{ z\in \mathbb{F}_{p}[t] \slash (\pi) : \varphi_{(p-1)^{\ell},c}(z) - z \equiv 0 \ \text{(mod $\pi$)}\bigg\}.
\end{equation} \noindent Again, setting $\ell =1$ and so $\varphi_{(p-1)^{\ell}, c} = \varphi_{p-1,c}$, we first prove the following theorem and its generalization \ref{6.2}:

\begin{thm} \label{6.1}
Let $\varphi_{4, c}$ be defined by $\varphi_{4, c}(z) = z^4 + c$ for all $c, z\in\mathbb{F}_{5}[t]$, and $M_{c(t)}(\pi, 5)$ be as in \textnormal{(\ref{M_{c}})}. Then $M_{c(t)}(\pi, 5) = 1$ or $2$ for every $c\equiv 1 \ (\text{mod} \ \pi)$ or $c\in (\pi)$, resp.; or else $M_{c(t)}(\pi, 5) = 0$ for every $c\equiv -1\ (\text{mod} \ \pi)$.  
\end{thm}
\begin{proof}
Let $g_{c(t)}(z)= \varphi_{4,c}(z)-z = z^4 - z + c$ and note that for every coefficient $c\in (\pi) := \pi\mathbb{F}_{5}[t]$, reducing $g_{c(t)}(z)$ modulo prime $\pi$, we then obtain $g_{c(t)}(z)\equiv z^4 - z$ (mod $\pi$); and so the reduced polynomial $g_{c(t)}(z)$ modulo $\pi$ is now a polynomial defined over a finite field $\mathbb{F}_{5}[t] / (\pi)$ of order $5^{\text{deg}(\pi)} = 5^m$. Now recall from a well-known fact about subfields of finite fields that every subfield of $\mathbb{F}_{5}[t] / (\pi)$ is of order $5^r$ for some positive integer $r\mid m$, we then obtain the inclusion $\mathbb{F}_{5}\hookrightarrow\mathbb{F}_{5}[t] / (\pi)$ of fields; and moreover recall (as a fact) that $h(x)=x^4-1$ vanishes at every nonzero $z\in \mathbb{F}_{5}$; and so $z^4=1$ for every $z\in \mathbb{F}_{5}^{\times}=\mathbb{F}_{5}\setminus \{0\}$. But then $g_{c(t)}(z)\equiv 1 - z$ (mod $\pi$) for every nonzero $z\in \mathbb{F}_{5} \subset\mathbb{F}_{5}[t] / (\pi)$; and so $g_{c(t)}(z)$ has a root in $\mathbb{F}_{5}[t] / (\pi)$, namely, $z\equiv 1$ (mod $\pi$). Moreover, since $z$ is a linear factor of $g_{c(t)}(z)\equiv z(z^3 - 1)$ (mod $\pi$), it then follows $z\equiv 0$ (mod $\pi$) is also a root of $g_{c(t)}(z)$ modulo $\pi$ in $\mathbb{F}_{5}[t] / (\pi)$. But now we then conclude $\#\{ z\in \mathbb{F}_{5}[t]\slash (\pi) :  \varphi_{4,c}(z) - z \equiv 0 \text{ (mod $\pi$)}\} = 2$ and so $M_{c(t)}(\pi, 5) = 2$. To see $M_{c(t)}(\pi, 5) = 1$ for every coefficient $c\equiv 1$ (mod  $\pi$), we note that with $c\equiv 1$ (mod  $\pi$) and since also $z^4 = 1$ for every $z\in \mathbb{F}_{5}^{\times}$, reducing $g_{c(t)}(z)= \varphi_{4,c}(z)-z = z^4 - z + c$ modulo $\pi$, it then follows $g_{c(t)}(z)\equiv 2 - z$ (mod $\pi$) and so $g_{c(t)}(z)$ has a root in $\mathbb{F}_{5}[t] / (\pi)$, namely, $z\equiv 2$ (mod $\pi$); and so we then conclude $M_{c(t)}(\pi, 5) = 1$. Finally, to see $M_{c(t)}(\pi, 5) = 0$ for every coefficient $c\equiv -1$ ( mod $\pi)$, we note that since $c \equiv -1$ (mod $\pi$) and since also $z^4 = 1$ for every element $z\in \mathbb{F}_{5}^{\times}\subset \mathbb{F}_{5}[t] / (\pi)$, it then follows $z^4 - z + c\equiv -z$ (mod $\pi$); and so $g_{c(t)}(z) \equiv -z$ (mod $\pi$). But now we note that $z\equiv 0$ (mod $\pi$) is a root of $g_{c(t)}(z)$ modulo $\pi$ for every coefficient $c\equiv -1$ (mod $\pi$) and for every coefficient $c\equiv 0$ (mod $\pi$) as seen from the first part; which is impossible, since $-1\not \equiv 0$ (mod $\pi$). This then means that $g_{c(t)}(x)=\varphi_{4,c}(x)-x$ has no roots in $\mathbb{F}_{5}[t] / (\pi)$ for every coefficient $c\equiv -1$ (mod $\pi$); and so we then conclude $M_{c(t)}(\pi, 5) = 0$, as also desired. This then completes the whole proof, as needed.
\end{proof} 
We now wish to generalize Theorem \ref{6.1} to any map $\varphi_{p-1, c}$ for any given prime integer $p\geq 5$. More precisely, we prove that the number of distinct $\mathbb{F}_{p}[t]$-fixed points of every map $\varphi_{p-1, c}$ modulo $\pi$ is $1$ or $2$ or $0$:

\begin{thm} \label{6.2}
Let $p\geq 5$ be any fixed prime integer, and consider a family of polynomial maps $\varphi_{p-1, c}$ defined by $\varphi_{p-1, c}(z) = z^{p-1}+c$ for all $c, z\in\mathbb{F}_{p}[t]$. Let $M_{c(t)}(\pi, p)$ be the number as in \textnormal{(\ref{M_{c}})}. Then $M_{c(t)}(\pi, p) = 1$ or $2$ for every coefficient $c\equiv 1 \ (mod \ \pi)$ or $c\in (\pi)$, resp.; otherwise $M_{c(t)}(\pi, p) = 0$ for every coefficient $c\equiv -1\ (mod \ \pi)$. 
\end{thm}
\begin{proof}
By applying a similar argument as in the Proof of Theorem \ref{6.1}, we then obtain the count as desired. That is, let $g_{c(t)}(z)= z^{p-1} - z + c$ and note that for every coefficient $c\in (\pi) := \pi \mathbb{F}_{p}[t]$, reducing $g_{c(t)}(z)$ modulo prime $\pi$, it then follows $g_{c(t)}(z)\equiv z^{p-1} - z$ (mod $\pi$); and so the reduced polynomial $g_{c(t)}(z)$ modulo $\pi$ is now a polynomial defined over a finite field $\mathbb{F}_{p}[t] / (\pi)$ of order $p^{\text{deg}(\pi)} = p^m$. So now, recall the inclusion $\mathbb{F}_{p}\hookrightarrow \mathbb{F}_{p}[t] / (\pi)$ of fields; and also recall (as a fact) that $h(x)=x^{p-1}-1$ vanishes at every point $z\in \mathbb{F}_{p}^{\times}= \mathbb{F}_{p}\setminus \{0\}$; and so $z^{p-1}=1$ for every element $z\in \mathbb{F}_{p}^{\times}\subset \mathbb{F}_{p}[t] / (\pi)$. But now $g_{c(t)}(z)\equiv 1 - z$ (mod $\pi$) for every point $z\in \mathbb{F}_{p}^{\times}\subset \mathbb{F}_{p}[t] / (\pi)$; and so $g_{c(t)}(z)$ has a root in $\mathbb{F}_{p}[t] / (\pi)$, namely, $z\equiv 1$ (mod $\pi$). Moreover, since $z$ is a linear factor of $g_{c(t)}(z)\equiv z(z^{p-2} - 1)$ (mod $\pi$), it then follows $z\equiv 0$ (mod $\pi$) is also a  root of $g_{c(t)}(z)$ modulo $\pi$ in $\mathbb{F}_{p}[t] / (\pi)$. But now we then conclude $\#\{ z\in \mathbb{F}_{p}[t]\slash (\pi) : \varphi_{p-1,c}(z) - z \equiv 0 \text{ (mod $\pi$)}\} = 2$ and so $M_{c(t)}(\pi, p) = 2$. To see $M_{c(t)}(\pi, p) = 1$ for every coefficient $c\equiv 1$ (mod  $\pi$), we note that with $c\equiv 1$ (mod  $\pi$) and since also $z^{p-1} = 1$ for every $z\in \mathbb{F}_{p}^{\times}\subset \mathbb{F}_{p}[t] / (\pi)$, reducing $g_{c(t)}(z)= \varphi_{p-1,c}(z)-z = z^{p-1} - z + c$ modulo $\pi$, it then follows $g_{c(t)}(z)\equiv 2 - z$ (mod $\pi$) and so $g_{c(t)}(z)$ has a root in $\mathbb{F}_{p}[t] / (\pi)$, namely, $z\equiv 2$ (mod $\pi$); and so we conclude $M_{c(t)}(\pi, p) = 1$. Finally, to see $M_{c(t)}(\pi, p) = 0$ for every coefficient  $c\equiv -1$ (mod $\pi)$, we note that since $c \equiv -1$ (mod $\pi$) and since also $z^{p-1} = 1$ for every $z\in \mathbb{F}_{p}^{\times}\subset \mathbb{F}_{p}[t] / (\pi)$, it then follows $z^{p-1} - z + -1\equiv -z$ (mod $\pi$); and so $g_{c(t)}(z) \equiv -z$ (mod $\pi$). But now we note $z\equiv 0$ (mod $\pi$) is a root of $g_{c(t)}(z)$ modulo $\pi$ for every coefficient $c\equiv -1$ (mod $\pi$) and for every coefficient $c\equiv 0$ (mod $\pi$) as seen from the first part; and which is impossible, since $-1\not \equiv 0$ (mod $\pi$). This then means $g_{c(t)}(x)=\varphi_{p-1,c}(x)-x$ has no roots in $\mathbb{F}_{p}[t] / (\pi)$ for every coefficient $c\equiv -1$ (mod $\pi$); and so we then conclude $M_{c(t)}(\pi, p) = 0$. This then completes the whole proof, as desired.
\end{proof}

Finally, we now wish to generalize Theorem \ref{6.2} further to any $\varphi_{(p-1)^{\ell}, c}$ for any prime $p\geq 5$ and for any $\ell \in \mathbb{Z}_{\geq 1}$. That is, we prove the number of distinct fixed points of every $\varphi_{(p-1)^{\ell}, c}$ modulo $\pi$ is also $1$ or $2$ or $0$:

\begin{thm} \label{6.3}
Let $p\geq 5$ be any fixed prime integer, and $\ell \geq 1$ be any integer. Consider a family of polynomial maps $\varphi_{(p-1)^{\ell}, c}$ defined by $\varphi_{(p-1)^{\ell}, c}(z) = z^{(p-1)^{\ell}}+c$ for all $c, z\in\mathbb{F}_{p}[t]$. Let $M_{c(t)}(\pi, p)$ be defined as in \textnormal{(\ref{M_{c}})}. Then $M_{c(t)}(\pi, p) = 1$ or $2$ for every $c\equiv 1 \ (mod \ \pi)$ or $c\in (\pi)$, resp.; or else $M_{c(t)}(\pi, p) = 0$ for every $c\equiv -1\ (mod \ \pi)$. 
\end{thm}

\begin{proof}
As before, let $g_{c(t)}(z)= z^{(p-1)^{\ell}} - z + c$ and note that for every coefficient $c\in (\pi)$, reducing $g_{c(t)}(z)$ modulo  prime $\pi$, it then follows $g_{c(t)}(z)\equiv z^{(p-1)^{\ell}} - z$ (mod $\pi$); and so the reduced polynomial $g_{c(t)}(z)$ modulo $\pi$ is now a polynomial defined over a finite field $\mathbb{F}_{p}[t] / (\pi)$. Now recall the inclusion $\mathbb{F}_{p}\hookrightarrow \mathbb{F}_{p}[t] / (\pi)$ of fields and also recall (as a fact) that $z^{p-1} = 1$ for every $z \in \mathbb{F}_{p}^{\times} \subset \mathbb{F}_{p}[t] / (\pi)$, it then follows $z^{(p-1)^{\ell}} = 1$ for every $z \in \mathbb{F}_{p}^{\times} \subset \mathbb{F}_{p}[t] / (\pi)$ and for every $\ell \in \mathbb{Z}_{\geq 1}$. But then $g_{c(t)}(z)\equiv 1 - z$ (mod $\pi$) for every point $z \in \mathbb{F}_{p}^{\times} \subset \mathbb{F}_{p}[t] / (\pi)$; and so $g_{c(t)}(z)$ modulo $\pi$ has a root in $\mathbb{F}_{p}[t] / (\pi)$, namely, $z\equiv 1$ (mod $\pi$). Moreover, since $z$ is a linear factor of $g_{c(t)}(z)\equiv z(z^{(p-1)^{\ell}-1} - 1)$ (mod $\pi$), it then follows $z\equiv 0$ (mod $\pi$) is also a root of $g_{c(t)}(z)$ modulo $\pi$ in $\mathbb{F}_{p}[t] / (\pi)$. But now we then conclude $\#\{ z\in \mathbb{F}_{p}[t]\slash (\pi) : \varphi_{(p-1)^{\ell},c}(z) - z \equiv 0 \text{ (mod $\pi$)}\} = 2$ and so $M_{c(t)}(\pi, p) = 2$. To see $M_{c(t)}(\pi, p) = 1$ for every coefficient $c\equiv 1$ (mod  $\pi$), we note that with $c\equiv 1$ (mod $\pi$) and since also $z^{(p-1)^{\ell}} = 1$ for every element $z\in \mathbb{F}_{p}^{\times}\subset \mathbb{F}_{p}[t] / (\pi)$, reducing $g_{c(t)}(z)= \varphi_{(p-1)^{\ell},c}(z)-z = z^{(p-1)^{\ell}} - z + c$ modulo $\pi$, it then follows $g_{c(t)}(z)\equiv 2 - z$ (mod $\pi$) and so $g_{c(t)}(z)$ has a root in $\mathbb{F}_{p}[t] / (\pi)$, namely, $z\equiv 2$ (mod $\pi$); and so we then conclude $M_{c(t)}(\pi, p) = 1$. Finally, to see $M_{c(t)}(\pi, p) = 0$ for every coefficient $c\equiv -1$ (mod $\pi)$ and for every $\ell \in \mathbb{Z}_{\geq 1}$, we note that since $c \equiv -1$ (mod $\pi$) and since also $z^{(p-1)^{\ell}} = 1$ for every $z\in \mathbb{F}_{p}^{\times}\subset \mathbb{F}_{p}[t] / (\pi)$ and for every $\ell \geq 1$, it then follows $z^{(p-1)^{\ell}} - z + -1\equiv -z$ (mod $\pi$); and so $g_{c(t)}(z) \equiv -z$ (mod $\pi$). But now we note that $z\equiv 0$ (mod $\pi$) is a root of $g_{c(t)}(z)$ modulo $\pi$ for every coefficient $c\equiv -1$ (mod $\pi$) and for every coefficient $c\equiv 0$ (mod $\pi$) as seen from the first part; and which is not possible, since $-1\not \equiv 0$ (mod $\pi$). This then means that $g_{c(t)}(x)=\varphi_{(p-1)^{\ell},c}(x)-x$ has no roots in $\mathbb{F}_{p}[t] / (\pi)$ for every coefficient $c\equiv -1$ (mod $\pi$) and for every $\ell \in \mathbb{Z}_{\geq 1}$; and so we then conclude $M_{c(t)}(\pi, p) = 0$, as also needed. This then completes the whole proof, as desired.
\end{proof}

\begin{rem}\label{re6.4}
As before, with now Theorem \ref{6.3}, we may to each distinct fixed point of $\varphi_{(p-1)^{\ell},c}$ associate a fixed orbit. In doing so, we then obtain a dynamical translation of Theorem \ref{6.3}, namely, that the number of distinct fixed orbits that any $\varphi_{(p-1)^{\ell},c}$ has when iterated on the space $\mathbb{F}_{p}[t] / (\pi)$ is $1$ or $2$ or $0$. Furthermore, as noted in Intro. \ref{sec1} that the count obtained in Theorem \ref{6.3} in all of the cases $c\equiv 1, 0, -1$ (mod $\pi)$ is independent of $p, \ell$ (and so independent of deg$(\varphi_{(p-1)^{\ell},c})$) and $m=$ deg$(\pi)$. Moreover, we may also observe that the expected total count (namely, $1 + 2 + 0 =3$) in Theorem \ref{6.3} is not only also independent of $p, \ell$ (and so independent of deg$(\varphi_{(p-1)^{\ell},c}$) and $m$, but is also a constant $3$ even as degree $(p-1)^{\ell}\to \infty$ or $m\to \infty$; a somewhat interesting phenomenon differing significantly from a phenomenon that we remarked about in Remark \ref{re5.4}, but  coinciding precisely with a phenomenon that we remarked about in [\cite{BK2}, Remark 3.5] and also currently here in Rem. \ref{re4.4}.
\end{rem}

\section{Average Number of $\mathcal{O}_{K}\slash p\mathcal{O}_{K}$-and $\mathbb{Z}_{p}\slash p\mathbb{Z}_{p}$-Fixed Points of any $\varphi_{p^{\ell},c}$ and $\varphi_{(p-1)^{\ell},c}$} \label{sec7}

In this section, we wish to first restrict on the subring $\mathbb{Z}\subset \mathcal{O}_{K}$ and then study the behavior of $N_{c}(p)$ as $c\to \infty$. More precisely, we wish to determine: \say{\textit{What is the average value of $N_{c}(p)$ as $c \to \infty$?}} The following corollary shows that the average value of $N_{c}(p)$ is zero or bounded if $\ell \in \mathbb{Z}^{+}\setminus \{1, p\}$ or unbounded if $\ell \in \{1,p\}$ as $c\to \infty$:
\begin{cor}\label{7.1}
Let $K\slash \mathbb{Q}$ be any number field of degree $n \geq 2$ in which any prime $p\geq 3$ is inert. The average value of $N_{c}(p)$ is zero or bounded if $\ell \in \mathbb{Z}^{+}\setminus \{1, p\}$ or unbounded if $\ell \in \{1, p\}$ as $c\to\infty$. That is, we have
\begin{myitemize}
    \item[\textnormal{(a)}] \textnormal{Avg} $N_{c\neq pt}(p)= \lim\limits_{c \to\infty} \Large{\frac{\sum\limits_{3\leq p\leq c, \ p\nmid c \text{ in } \mathcal{O}_{K}}N_{c}(p)}{\Large{\sum\limits_{3\leq p\leq c, \ p\nmid c \text{ in } \mathcal{O}_{K}}1}}} =  0$. 
    
    \item[\textnormal{(b)}] $2\leq$ \textnormal{Avg} $N_{c = pt, \ell \in \mathbb{Z}^{+}\setminus \{1, p\}}(p) \leq \ell$, where $\ell \geq 2$.

     \item[\textnormal{(c)}] \textnormal{Avg} $N_{c= pt, \ell \in \{1, p\}}(p)= \lim\limits_{c \to\infty} \Large{\frac{\sum\limits_{3\leq p\leq c, \ p\mid c \text{ in } \mathcal{O}_{K}, \ell \in \{1, p\}}N_{c}(p)}{\Large{\sum\limits_{3\leq p\leq c, \ p\mid c \text{ in } \mathcal{O}_{K}, \ell \in \{1, p\}}1}}} =  \infty$.
\end{myitemize}

\end{cor}
\begin{proof}
Since by Theorem \ref{2.3} the number $N_{c}(p) = 0$ for any inert prime $p\nmid c$ in $\mathcal{O}_{K}$, we then obtain $\lim\limits_{c\to\infty} \Large{\frac{\sum\limits_{3\leq p\leq c, \ p\nmid c}N_{c}(p)}{\Large{\sum\limits_{3\leq p\leq c, \ p\nmid c}1}}} = 0$; and so the average Avg$N_{c \neq pt}(p) = 0$. To see (b), we also recall from Theorem \ref{2.3} that $2\leq N_{c}(p)\leq \ell$ for any inert $p\mid c$ in $\mathcal{O}_{K}$ and any $\ell \in \mathbb{Z}^{+}\setminus \{1, p\}$; and so taking summation over all $p\geq 3$, we then obtain $2\sum\limits_{3\leq p\leq c, \ p\mid c, \ \ell \in \mathbb{Z}^{+}\setminus \{1, p\}}1\leq \sum\limits_{3\leq p\leq c, \ p\mid c, \ \ell \in \mathbb{Z}^{+}\setminus \{1, p\}}N_{c}(p)\leq \ell\sum\limits_{3\leq p\leq c, \ p\mid c, \ \ell \in \mathbb{Z}^{+}\setminus \{1, p\}}1$. Since by assumption $\sum\limits_{3\leq p\leq c, \ p\mid c, \ \ell \in \mathbb{Z}^{+}\setminus \{1, p\}}1 \neq 0$, then scaling by $1\slash \sum\limits_{3\leq p\leq c, \ p\mid c, \ \ell \in \mathbb{Z}^{+}\setminus \{1, p\}}1$ and letting $c\to \infty$, we then obtain $2\leq \lim\limits_{c\to\infty} \Large{\frac{\sum\limits_{3\leq p\leq c, \ p\mid c, \ \ell \in \mathbb{Z}^{+}\setminus \{1, p\}}N_{c}(p)}{\Large{\sum\limits_{3\leq p\leq c, \ p\mid c, \ \ell \in \mathbb{Z}^{+}\setminus \{1, p\}}1}}}\leq \ell$; and so conclude $2\leq$ Avg $N_{c = pt, \ell \in \mathbb{Z}^{+}\setminus \{1, p\}}(p) \leq \ell$. Finally, to see (c), we again recall from Theorem \ref{2.3} that $N_{c}(p) = p$ for any inert prime $p\mid c$ and any $\ell \in \{1, p\}$. But now observe $\sum\limits_{3\leq p\leq c, \ p\mid c, \ \ell \in \{1, p\}} N_{c}(p) = \sum\limits_{3\leq p\leq c, \ p\mid c, \ \ell \in \{1, p\}}p =: \sigma_{1,p}(c)$ and  $\sum\limits_{3\leq p\leq c, \ p\mid c, \ \ell \in \{1, p\}} 1  = \omega(c)$, where $\sigma_{1}(n)$ (resp. $\omega(n)$) is the number of divisors (resp. the number of distinct prime divisors) of any integer $n\in \mathbb{Z}^{+}$; and so $\frac{\sum\limits_{3\leq p\leq c, \ p\mid c, \ \ell \in \{1, p\}} N_{c}(p)}{\sum\limits_{3\leq p\leq c, \ p\mid c, \ \ell \in \{1, p\}} 1} = \frac{\sigma_{1,p}(c)}{\omega(c)}$. Now observe $\sigma_{1,p}(c)=\sum\limits_{3\leq p \leq c, \ p\mid c, \ \ell \in \{1, p\}}p\leq \sum\limits_{3\leq p\leq c}p$ for each integer $c\geq 3$, and since from \cite{Ma} we also know that $\sum\limits_{3\leq p\leq c}p = \frac{c^2}{2}(\text{log }c + \text{log log }c - \frac{3}{2}+\frac{\text{log log }c - 5\slash 2}{\text{log }c}) + O(\frac{c^2(\text{log log }c)^2}{\text{log}^2c})$ for each integer $c\geq 3$, we then obtain that the numerator $\sigma_{1,p}(c)\leq \frac{c^2}{2}(\text{log }c + \text{log log }c - \frac{3}{2}+\frac{\text{log log }c - 5\slash 2}{\text{log }c}) + O(\frac{c^2(\text{log log }c)^2}{\text{log}^2c})$. Moreover, it is also a well-known fact that the denominator $\omega(c)\leq \frac{\text{log }c}{\text{log }2}$ for every $c\in \mathbb{Z}_{\geq 3}$. Now since $\frac{\text{log }c}{\text{log }2} \to \infty$ very slow as $c\to \infty$, it then follows $\omega(c)\to \infty$ very slow as $c\to \infty$. On the other hand, since $\frac{c^2}{2}(\text{log }c + \text{log log }c - \frac{3}{2}+\frac{\text{log log }c - 5\slash 2}{\text{log }c}) + O(\frac{c^2(\text{log log }c)^2}{\text{log}^2c})\to \infty$ very fast as $c\to \infty$, then $\sigma_{1,p}(c) \to \infty$ very fast as $c\to \infty$. But then $\frac{\sigma_{1,p}(c)}{\omega(c)} \to \infty$ very fast as $c\to \infty$ and so $\lim\limits_{c\to\infty} \Large{\frac{\sum\limits_{3\leq p\leq c, \ p\mid c, \ \ell \in \{1, p\}}N_{c}(p)}{\Large{\sum\limits_{3\leq p\leq c, \ p\mid c, \ \ell\in \{1, p\}}1}}}= \infty$; and from which we then conclude Avg $N_{c= pt, \ell \in \{1, p\}}(p) = \infty$, as also desired. This then completes the whole proof, as needed.
\end{proof}
\begin{rem} \label{re7.2}
From arithmetic statistics to arithmetic dynamics, Corollary \ref{7.1} shows that any $\varphi_{p,c}$ iterated on the space $\mathcal{O}_{K} / p\mathcal{O}_{K}$ has on average zero or a bounded or unbounded number of distinct fixed orbits as $c\to \infty$.
\end{rem}

As before, we restrict on $\mathbb{Z}\subset \mathbb{Z}_{p}$ and then determine: \say{\textit{What is the average value of $X_{c}(p)$ as $c \to \infty$?}} The following corollary shows that the average value of $X_{c}(p)$ is also zero or bounded or unbounded  as $c\to \infty$:
\begin{cor}\label{7.3}
Let $p\geq 3$ be any prime integer. Then the average value of the function $X_{c}(p)$ is zero or bounded whenever $\ell \in \mathbb{Z}^{+}\setminus \{1, p\}$ or unbounded whenever $\ell \in \{1, p\}$ as $c\to\infty$. More precisely, we have
\begin{myitemize}
    \item[\textnormal{(a)}] \textnormal{Avg} $X_{c\neq pt}(p)= \lim\limits_{c \to\infty} \Large{\frac{\sum\limits_{3\leq p\leq c, \ p\nmid c \text{ in } \mathbb{Z}_{p}}X_{c}(p)}{\Large{\sum\limits_{3\leq p\leq c, \ p\nmid c \text{ in } \mathbb{Z}_{p}}1}}} =  0$. 
    
    \item[\textnormal{(b)}] $2\leq$ \textnormal{Avg} $X_{c = pt, \ell \in \mathbb{Z}^{+}\setminus \{1, p\}}(p) \leq \ell$, where $\ell \geq 2$.

     \item[\textnormal{(c)}] \textnormal{Avg} $X_{c= pt, \ell \in \{1, p\}}(p)= \lim\limits_{c \to\infty} \Large{\frac{\sum\limits_{3\leq p\leq c, \ p\mid c \text{ in } \mathbb{Z}_{p}, \ell \in \{1, p\}}X_{c}(p)}{\Large{\sum\limits_{3\leq p\leq c, \ p\mid c \text{ in } \mathbb{Z}_{p}, \ell\in \{1, p\}}1}}} =  \infty$.
\end{myitemize}
\end{cor}

\begin{proof}
By applying a similar argument as in the Proof of Corollary \ref{7.1}, we then obtain the limits as desired.
\end{proof}
\begin{rem} \label{re7.4}
Again, we note that from arithmetic statistics to arithmetic dynamics, Cor. \ref{7.3} shows that any $\varphi_{p^{\ell},c}$ iterated on $\mathbb{Z}_{p} / p\mathbb{Z}_{p}$ has on average $0$ or a bounded or unbounded number of distinct fixed orbits as $c\to \infty$.
\end{rem}

Similarly, we also wish to determine: \say{\textit{What is the average value of the function $Y_{c}(p)$ as $c\to \infty$?}} The following corollary shows that the average value of $Y_{c}(p)$ exists and is moreover equal to $1$ or $2$ or $0$ as $c\to \infty$:
\begin{cor}\label{7.5}
Let $p\geq 5$ be any prime integer. Then the average value of the function $Y_{c}(p)$ exists and is equal to $1$ or $2$ or $0$ as $c\to\infty$. More precisely, we have 
\begin{myitemize}
    \item[\textnormal{(a)}] \textnormal{Avg} $Y_{c-1 = pt}(p) = \lim\limits_{(c-1)\to\infty} \Large{\frac{\sum\limits_{5\leq p\leq (c-1), \ p\mid (c-1) \text{ in } \mathbb{Z}_{p}}Y_{c}(p)}{\Large{\sum\limits_{5\leq p\leq (c-1), \ p\mid (c-1) \text{ in } \mathbb{Z}_{p}}1}}} = 1.$ 

    \item[\textnormal{(b)}] \textnormal{Avg} $Y_{c= pt}(p) = \lim\limits_{c\to\infty} \Large{\frac{\sum\limits_{5\leq p\leq c, \ p\mid c \text{ in } \mathbb{Z}_{p}}Y_{c}(p)}{\Large{\sum\limits_{5\leq p\leq c, \ p\mid c \text{ in } \mathbb{Z}_{p}}1}}} = 2.$
    
    \item[\textnormal{(c)}] \textnormal{Avg} $Y_{c+1= pt}(p)= \lim\limits_{(c+1) \to\infty} \Large{\frac{\sum\limits_{5\leq p\leq (c+1), \ p\mid (c+1) \text{ in } \mathbb{Z}_{p}}Y_{c}(p)}{\Large{\sum\limits_{5\leq p\leq (c+1), \ p\mid (c+1) \text{ in } \mathbb{Z}_{p}}1}}} =  0$.    
\end{myitemize}

\end{cor}
\begin{proof}
Since from Theorem \ref{4.3} we know $Y_{c}(p) = 1$ for any $p$ such that $p\mid (c-1)$ in $\mathbb{Z}_{p}$, we then have $\lim\limits_{(c-1)\to\infty} \Large{\frac{\sum\limits_{5\leq p\leq (c-1), \ p\mid (c-1) \text{ in } \mathbb{Z}_{p}}Y_{c}(p)}{\Large{\sum\limits_{5\leq p\leq (c-1), \ p\mid (c-1) \text{ in } \mathbb{Z}_{p}}1}}} = \lim\limits_{(c-1)\to\infty} \Large{\frac{\sum\limits_{5\leq p\leq (c-1), \ p\mid (c-1) \text{ in } \mathbb{Z}_{p}}1}{\Large{\sum\limits_{5\leq p\leq (c-1), \ p\mid (c-1) \text{ in } \mathbb{Z}_{p}}1}}} = 1$; and so the average Avg $Y_{c-1 = pt}(p) = 1$. Similarly, since from Theorem \ref{4.3} we know $Y_{c}(p) = 2$ or $0$ for any prime $p$ such that $p\mid c$ in $\mathbb{Z}_{p}$ or $p$ such that $p\mid (c+1)$ in  $\mathbb{Z}_{p}$, resp., we then obtain $\lim\limits_{c\to\infty} \Large{\frac{\sum\limits_{5\leq p\leq c, \ p\mid c \text{ in } \mathbb{Z}_{p}}Y_{c}(p)}{\Large{\sum\limits_{5\leq p\leq c, \ p\mid c \text{ in } \mathbb{Z}_{p}}1}}} = 2$ or $\lim\limits_{(c+1)\to\infty} \Large{\frac{\sum\limits_{5\leq p\leq (c+1), \ p\mid (c+1) \text{ in } \mathbb{Z}_{p}}Y_{c}(p)}{\Large{\sum\limits_{5\leq p\leq (c+1), \ p\mid (c+1) \text{ in } \mathbb{Z}_{p}}1}}} = 0$, resp.; and so Avg $Y_{c = pt}(p) = 2$ or Avg $Y_{c+1 = pt}(p) = 0$, resp. This then completes the whole proof, as required.
\end{proof} 
\begin{rem} \label{re5}
As before, we again note that from arithmetic statistics to arithmetic dynamics, Corollary \ref{7.5} shows that any $\varphi_{(p-1)^{\ell},c}$ iterated on the space $\mathbb{Z}_{p} / p\mathbb{Z}_{p}$ has on average one or two or no fixed orbits as $c \to \infty$.
\end{rem}

\section{On Average Number of $\mathbb{F}_{p}[t]\slash (\pi)$-Fixed Points of any Family of $\varphi_{p^{\ell},c}$ \& $\varphi_{(p-1)^{\ell},c}$}\label{sec8}

As in Section \ref{sec7}, we also wish to inspect the asymptotic behavior of the function $N_{c(t)}(\pi, p)$ as \text{deg}$(c)\to \infty$. More precisely, we wish to determine: \say{\textit{What is the average value of the function $N_{c(t)}(\pi, p)$ as \text{deg}(c)$\to \infty$?}} The following corollary shows that the average value of $N_{c(t)}(\pi, p)$ is zero or bounded or unbounded as \text{deg}$(c)\to \infty$:
\begin{cor}\label{co6}
Let $p\geq 3$ be any prime integer and deg(c) $=n\geq 3$ be any integer. Then the average value of $N_{c(t)}(\pi, p)$ is zero or bounded if $\ell \in \mathbb{Z}^{+}\setminus \{1, p\}$ or unbounded if $\ell \in \{1, p\}$ as $n\to\infty$. That is, we have
\begin{myitemize}
    \item[\textnormal{(a)}] \textnormal{Avg} $N_{c(t)\neq \pi t}(\pi, p)= \lim\limits_{n \to\infty} \Large{\frac{\sum\limits_{3\leq p \leq n, \ \pi\nmid c \text{ in } \mathbb{F}_{p}[t]}N_{c(t)}(\pi, p)}{\Large{\sum\limits_{3\leq p \leq n, \ \pi\nmid c \text{ in } \mathbb{F}_{p}[t]}1}}} =  0$. 
    
    \item[\textnormal{(b)}] $2\leq$ \textnormal{Avg} $N_{c(t) = \pi t, \ell \in \mathbb{Z}^{+}\setminus \{1, p\}}(\pi, p) \leq \ell$, where $\ell \geq 2$.

     \item[\textnormal{(c)}] \textnormal{Avg} $N_{c(t)= \pi t, \ell \in \{1, p\}}(\pi, p)= \lim\limits_{n \to\infty} \Large{\frac{\sum\limits_{3\leq p \leq n, \ \pi \mid c \text{ in } \mathbb{F}_{p}[t], \ell \in \{1, p\}}N_{c(t)}(\pi, p)}{\Large{\sum\limits_{3\leq p \leq n, \ \pi\mid c \text{ in } \mathbb{F}_{p}[t], \ell \in \{1, p\}}1}}} =  \infty$.
\end{myitemize}
\end{cor}

\begin{proof}
By applying a similar argument in the Proof of Cor. \ref{7.1} with limit as $c\to \infty$ replaced with deg $(c)\to \infty$, we then obtain the limits. That is, since from Thm. \ref{5.3} we know $N_{c(t)}(\pi, p) = 0$ for any $\pi\in \mathbb{F}_{p}[t]$ such that $\pi \nmid c$, we then obtain $\lim\limits_{n\to\infty} \Large{\frac{\sum\limits_{3\leq p \leq n, \ \pi \nmid c\text{ in }\mathbb{F}_{p}[t]}N_{c(t)}(\pi, p)}{\Large{\sum\limits_{3\leq p \leq n, \ \pi\nmid c\text{ in }\mathbb{F}_{p}[t]}1}}} = 0$ and so Avg $N_{c(t) \neq \pi t}(\pi, p) = 0$. Similarly, since from Thm. \ref{5.3} we know $2\leq N_{c(t)}(\pi, p) \leq \ell$ for any $\pi\in \mathbb{F}_{p}[t]$ such that $\pi \mid c$ and any $\ell \in \mathbb{Z}^{+}\setminus \{1, p\}$, we have $2\leq \lim\limits_{n \to\infty} \Large{\frac{\sum\limits_{3\leq p \leq n, \ \pi \mid c \text{ in } \mathbb{F}_{p}[t], \ell \not \in \{1, p\}}N_{c(t)}(\pi, p)}{\Large{\sum\limits_{3\leq p \leq n, \ \pi\mid c \text{ in } \mathbb{F}_{p}[t], \ell \not \in \{1, p\}}1}}} \leq \ell$; and so $2\leq$ \textnormal{Avg} $N_{c(t) = \pi t, \ell\not \in \{1, p\}}(\pi, p) \leq \ell$. To see \textnormal{(c)}, we recall from Thm. \ref{5.3} that $N_{c(t)}(\pi, p) = p$ for any $\pi\in \mathbb{F}_{p}[t]$ such that $\pi \mid c$ and any $\ell \in \{1, p\}$. Now observe $\sum\limits_{3\leq p \leq n, \ \pi \mid c \text{ in } \mathbb{F}_{p}[t], \ell \in \{1, p\}}N_{c(t)}(\pi, p) = \sum\limits_{3\leq p \leq n, \ \pi \mid c \text{ in } \mathbb{F}_{p}[t], \ell \in \{1, p\}}p = \sigma_{1,\pi}(c)$ and  $\sum\limits_{3\leq p \leq n, \ \pi\mid c \text{ in } \mathbb{F}_{p}[t], \ell \in \{1, p\}}1  =: \omega_{\pi}(c)$, where recalling from function field number theory [\cite{Rose}, Page 15] that the divisor function $\sigma_{1}(f)$ is defined as $\sigma_{1}(f)=\sum\limits_{g\mid f}|g|$ where $|g|=\# \mathbb{F}_{p}[t]\slash (g)$ for any monics $g, f\in \mathbb{F}_{p}[t]$ and then setting deg $\pi = 1$ (and so the size $|\pi|=\# \mathbb{F}_{p}[t]\slash (\pi) = p$) we then have $\sigma_{1, \pi}(c):=\sigma_{1}(c) =\sum\limits_{3\leq p \leq n, \ \pi \mid c \text{ in } \mathbb{F}_{p}[t], \ell \in \{1, p\}}|\pi| = \sum\limits_{3\leq p \leq n, \ \pi \mid c \text{ in } \mathbb{F}_{p}[t], \ell \in \{1, p\}}p$. So now, since we are varying deg$(c) = n$ (and hence varying $c=c(t)$) and so defining $\sigma_{1, \pi}(n):=\sigma_{1, \pi}(c)$ and also $\omega(n):= \omega_{\pi}(c)$, we then obtain  $\lim\limits_{n\to\infty}\frac{\sum\limits_{3\leq p \leq n, \ \pi \mid c \text{ in } \mathbb{F}_{p}[t], \ell \in \{1, p\}}N_{c(t)}(\pi, p)}{\sum\limits_{3\leq p \leq n, \ \pi\mid c \text{ in } \mathbb{F}_{p}[t], \ell \in \{1, p\}}1} = \lim\limits_{n\to\infty}\frac{\sigma_{1,\pi}(c)}{\omega_{\pi}(c)} = \lim\limits_{n\to\infty}\frac{\sigma_{1,\pi}(n)}{\omega_{\pi}(n)}$. Now since the partial sum $\sum\limits_{3\leq p \leq n, \ \pi \mid c \text{ in } \mathbb{F}_{p}[t], \ell \in \{1, p\}}p=\sigma_{1, \pi}(n)$ and $\sum\limits_{3\leq p \leq c, \ p \mid c \text{ in } \mathbb{Z}, \ell \in \{1, p\}}p=\sigma_{1, p}(c)$ are summed over the same  divisibility condition and moreover have the same summand, we then obtain that $\sigma_{1, \pi}(c)=\sigma_{1, p}(c)$ and $\omega_{\pi}(c)=\omega(c)$ for each $c$; and from which it then follows that $\frac{\sigma_{1, \pi}(c)}{\omega_{\pi}(c)}= \frac{\sigma_{1,p}(c)}{\omega(c)}$ for each $c$. But then recall from the Proof of Cor. \ref{7.1} that $\frac{\sigma_{1,p}(c)}{\omega(c)}\to \infty$ as $c\to \infty$; and so have that $\frac{\sigma_{1,\pi}(n)}{\omega_{\pi}(n)}\to \infty$ as $n=$deg$(c)\to \infty$, as desired. 
\end{proof}

\begin{rem} 
Again, from arithmetic statistics to arithmetic dynamics, Cor. \ref{co6} shows any $\varphi_{p^{\ell},c}$ iterated on $\mathbb{F}_{p}[t] / (\pi)$ has on average $0$ or a positive bounded or unbounded number of distinct fixed orbits as \text{deg}$(c)\to \infty$.
\end{rem}

Similarly, we also wish to determine: \say{\textit{What is the average value of $M_{c(t)}(\pi, p)$ as \text{deg}(c)$\to \infty$?}} The following corollary shows that the average value of $M_{c(t)}(\pi, p)$ exists and is equal to $1$ or $2$ or zero as \text{deg}$(c)\to \infty$:
\begin{cor}\label{cor6.1}
Let $p\geq 5$ be any prime integer, and deg(c) $=n\geq 5$ be any integer. Then the average value of the function $M_{c(t)}(\pi, p)$ exits and is equal to $1$ or $2$ or $0$ as $n\to\infty$. More precisely, we have 
\begin{myitemize}
    \item[\textnormal{(a)}] \textnormal{Avg} $M_{c(t)-1 = \pi t}(\pi, p) = \lim\limits_{n\to\infty} \Large{\frac{\sum\limits_{5\leq p \leq n, \ \pi\mid (c-1)\text{ in }\mathbb{F}_{p}[t]}M_{c(t)}(\pi, p)}{\Large{\sum\limits_{5\leq p \leq n, \ \pi\mid (c-1) \text{ in }\mathbb{F}_{p}[t]}1}}} = 1.$ 

    \item[\textnormal{(b)}] \textnormal{Avg} $M_{c(t)= \pi t}(\pi, p) = \lim\limits_{n\to\infty} \Large{\frac{\sum\limits_{5\leq p \leq n, \ \pi\mid c \text{ in }\mathbb{F}_{p}[t]}M_{c(t)}(\pi, p)}{\Large{\sum\limits_{5\leq p \leq n, \ \pi\mid c\text{ in }\mathbb{F}_{p}[t]}1}}} = 2.$
    
    \item[\textnormal{(c)}] \textnormal{Avg} $M_{c(t)+1= \pi t}(\pi, p)= \lim\limits_{n \to\infty} \Large{\frac{\sum\limits_{5\leq p \leq n, \ \pi\mid (c+1)\text{ in }\mathbb{F}_{p}[t]}M_{c(t)}(\pi, p)}{\Large{\sum\limits_{5\leq p \leq n, \ \pi\mid (c+1)\text{ in }\mathbb{F}_{p}[t]}1}}} =  0$.    
\end{myitemize}

\end{cor}
\begin{proof}
Since by Theorem \ref{6.3} we know $M_{c(t)}(\pi, p) = 1$ for any $\pi\in \mathbb{F}_{p}[t]$ such that $\pi\mid (c-1)$, it then follows that $\lim\limits_{n\to\infty} \Large{\frac{\sum\limits_{5\leq p \leq n, \ \pi \mid (c-1)\text{ in }\mathbb{F}_{p}[t]}M_{c(t)}(\pi, p)}{\Large{\sum\limits_{5\leq p \leq n, \ \pi \mid (c-1)\text{ in }\mathbb{F}_{p}[t]}1}}} = \lim\limits_{n\to\infty} \Large{\frac{\sum\limits_{5\leq p \leq n, \ \pi \mid (c-1)\text{ in }\mathbb{F}_{p}[t]}1}{\Large{\sum\limits_{5\leq p \leq n, \ \pi \mid (c-1)\text{ in }\mathbb{F}_{p}[t]}1}}} = 1$; and so Avg $M_{c(t)-1 = \pi t}(\pi, p) = 1$. Similarly, since from Theorem \ref{6.3} we know $M_{c(t)}(\pi, p) = 2$ or $0$ for any $\pi\in \mathbb{F}_{p}[t]$ such that $\pi\mid c$ or $\pi\mid (c+1)$ respectively, we then obtain $\lim\limits_{n\to\infty} \Large{\frac{\sum\limits_{5\leq p \leq n, \ \pi\mid c \text{ in }\mathbb{F}_{p}[t]}M_{c(t)}(\pi, p)}{\Large{\sum\limits_{5\leq p \leq n, \ \pi\mid c \text{ in }\mathbb{F}_{p}[t]}1}}} = 2$ or $\lim\limits_{n\to\infty} \Large{\frac{\sum\limits_{5\leq p \leq n, \ \pi \mid (c+1)\text{ in }\mathbb{F}_{p}[t]}M_{c(t)}(\pi, p)}{\Large{\sum\limits_{5\leq p \leq n, \ \pi\mid (c+1)\text{ in }\mathbb{F}_{p}[t]}1}}} = 0$, respectively; and so Avg $M_{c(t) = \pi t}(\pi, p) = 2$ or Avg $M_{c(t)+1 = \pi t}(\pi, p) = 0$, respectively; which then completes the whole proof.
\end{proof} 
\begin{rem} \label{re8.4}
As before, we note that from arithmetic statistics to arithmetic dynamics, Corollary \ref{cor6.1} shows that any $\varphi_{(p-1)^{\ell},c}$ iterated on the space $\mathbb{F}_{p}[t] / (\pi)$ has on average one or two or no fixed orbits as \text{deg}$(c)\to \infty$.
\end{rem}

\section{On Density of Polynomials $\varphi_{p^{\ell},c}(x)$ with $N_{c}(p) = X_{c}(p) = p$ and $N_{c}(p) = X_{c}(p)=0$}\label{sec9}

In this and the next section, we focus our counting on $\mathbb{Z}\subset \mathcal{O}_{K}$ and then ask: \say{\textit{For any fixed $\ell \in \mathbb{Z}^{+}$, what is the density of integer polynomials $\varphi_{p^{\ell},c}(x) = x^{p^{\ell}} + c\in \mathcal{O}_{K}[x]$ with exactly $p$ integral fixed points modulo $p$?}} The following corollary shows there are very few polynomials $\varphi_{p^{\ell},c}(x)\in \mathbb{Z}[x]$ with $p$ distinct fixed points modulo $p$:

\begin{cor}\label{9.1}
Let $K\slash \mathbb{Q}$ be any number field of degree $n\geq 2$ with the ring of integers $\mathcal{O}_{K}$, and in which any prime integer $p\geq 3$ is inert. Let $\ell \geq 1$ be any fixed integer. Then the density of monic integer polynomials $\varphi_{p^{\ell},c}(x) = x^{p^{\ell}} + c\in \mathcal{O}_{K}[x]$ with $N_{c}(p) = p$ exists and is equal to $0 \%$ as $c\to \infty$. More precisely, we have 
\begin{center}
    $\lim\limits_{c\to\infty} \Large{\frac{\# \{\varphi_{p^{\ell},c}(x)\in \mathbb{Z}[x] \ : \ 3\leq p\leq c \ \text{and} \ N_{c}(p) \ = \ p\}}{\Large{\# \{\varphi_{p^{\ell},c}(x) \in \mathbb{Z}[x] \ : \ 3\leq p\leq c \}}}} = \ 0.$
\end{center}
\end{cor}
\begin{proof}
Since the defining condition $N_{c}(p) = p$ is as we proved in Corollary \ref{cor2.4}, determined whenever the coefficient $c$ is divisible by any prime $p\geq 3$, we may then count $\# \{\varphi_{p^{\ell},c}(x) \in \mathbb{Z}[x] : 3\leq p\leq c \ \text{and} \ N_{c}(p) \ = \ p\}$ by counting the number $\# \{\varphi_{p,c}(x)\in \mathbb{Z}[x] : 3\leq p\leq c \ \text{and} \ p\mid c \ \text{for \ any \ fixed} \ c \}$. In that case, we then write 
\begin{center}
$\Large{\frac{\# \{\varphi_{p^{\ell},c}(x) \in \mathbb{Z}[x] \ : \ 3\leq p\leq c \ \text{and} \ N_{c}(p) \ = \ p\}}{\Large{\# \{\varphi_{p^{\ell},c}(x) \in \mathbb{Z}[x] \ : \ 3\leq p\leq c \}}}} = \Large{\frac{\# \{\varphi_{p^{\ell},c}(x)\in \mathbb{Z}[x] \ : \ 3\leq p\leq c \ \text{and} \ p\mid c \ \text{for any fixed} \ c \}}{\Large{\# \{\varphi_{p^{\ell},c}(x) \in \mathbb{Z}[x] \ : \ 3\leq p\leq c \}}}}$. 
\end{center}\indent Moreover, for any fixed integer $c\geq 3$, the numerator of the foregoing quotient may be rewritten to then obtain
\begin{center}
$\# \{\varphi_{p^{\ell},c}(x) \in \mathbb{Z}[x] : 3\leq p\leq c \ \text{and} \ p\mid c \} = \# \{p : 3\leq p\leq c \text{ and } p\mid c \} = \sum_{3\leq p\leq c, \ p\mid c}1 = \omega (c)$, 
\end{center}where $\omega(n)$ is by definition the number of distinct prime factors of $n$. Writing $\# \{\varphi_{p^{\ell},c}(x) \in \mathbb{Z}[x]  : 3\leq p\leq c \} = \sum_{3\leq p\leq c} 1 = \pi(c)$, where $\pi(n)$ is by definition the number of primes at most $n$, we then note that the quotient 
\begin{center}
$\Large{\frac{\# \{\varphi_{p^{\ell},c}(x)\in \mathbb{Z}[x] \ : \ 3\leq p\leq c \ \text{and} \ p\mid c \ \text{for any fixed} \ c \}}{\Large{\# \{\varphi_{p^{\ell},c}(x)\in \mathbb{Z}[x] \ : \ 3\leq p\leq c \}}}} = \frac{\omega(c)}{\pi(c)}$.
\end{center}So now, recall (from a well-known fact) that for any $c\in \mathbb{Z}_{\geq 3}$, we have $2^{\omega(c)}\leq \sigma (n) \leq 2^{\Omega(c)}$, where $\sigma(n)$ is by definition the divisor function and $\Omega(n)$ is by definition the total number of prime factors of $n$, with respect to their multiplicity. Note that taking logarithms, we then obtain $\omega(c)\leq \frac{\text{log} \ \sigma(c)}{\text{log} \ 2}$; and so $\frac{\omega(c)}{\pi(c)} \leq \frac{\text{log} \ \sigma(c)}{\text{log} \ 2 \cdot \pi(c)}$. Moreover, for every $\epsilon >0$, it is well-known that $\sigma(c) = o(c^{\epsilon})$; and so log $\sigma(c) =$ log $o(c^{\epsilon})$ and so have $\frac{\omega(c)}{\pi(c)} \leq \frac{\text{log} \ o(c^{\epsilon})}{\text{log} \ 2 \cdot \pi(c)}$. Now for every fixed $\epsilon>0$, we then note $\lim\limits_{c\to\infty} \frac{\text{log} \ o(c^{\epsilon})}{\text{log} \ 2 \cdot \pi(c)} = 0$ and so $\lim\limits_{c\to\infty} \frac{\omega(c)}{\pi(c)} \leq 0$. But now 
\begin{center}
$\lim\limits_{c\to\infty} \Large{\frac{\# \{\varphi_{p^{\ell},c}(x)\in \mathbb{Z}[x] \ : \ 3\leq p\leq c \ \text{and} \ N_{c}(p) \ = \ p\}}{\Large{\# \{\varphi_{p^{\ell},c}(x) \in \mathbb{Z}[x] \ : \ 3\leq p\leq c \}}}} =\lim\limits_{c\to\infty} \frac{\omega(c)}{\pi(c)} \leq 0$.
\end{center}Moreover, we also observe that the number $\# \{\varphi_{p^{\ell},c}(x)\in \mathbb{Z}[x] : 3\leq p\leq c \ \text{and} \ N_{c}(p) \ = \ p\}\geq 1$, and so have 
\begin{center}
$\lim\limits_{c\to\infty}\Large{\frac{\# \{\varphi_{p^{\ell},c}(x) \in \mathbb{Z}[x] \ : \ 3\leq p\leq c \ \text{and} \ N_{c}(p) \ = \ p\}}{\Large{\# \{\varphi_{p^{\ell},c}(x) \in \mathbb{Z}[x] \ : \ 3\leq p\leq c \}}}}\geq \lim\limits_{c\to\infty}\frac{1}{\pi(c)} = 0$. But now, we then conclude that the limit 
\end{center}  $\lim\limits_{c\to\infty} \Large{\frac{\# \{\varphi_{p^{\ell},c}(x) \in \mathbb{Z}[x] \ : \ 3\leq p\leq c \ \text{and} \ N_{c}(p) \ = \ p\}}{\Large{\# \{\varphi_{p^{\ell},c}(x) \in \mathbb{Z}[x] \ : \ 3\leq p\leq c \}}}} = 0$ as needed. This then completes the whole proof, as desired.
\end{proof}\noindent Note that one may also interpret Corollary \ref{9.1} as saying that for any fixed $\ell \in \mathbb{Z}_{ \geq 1}$, the probability of choosing randomly a monic polynomial $\varphi_{p^{\ell},c}(x)\in \mathbb{Z}[x]\subset \mathcal{O}_{K}[x]$ having exactly $p$ distinct fixed points modulo $p$ is zero.

As before, we also have the following corollary which shows that for any fixed $\ell \in \mathbb{Z}_{\geq 1}$, the probability of choosing randomly a monic polynomial $\varphi_{p^{\ell},c}(x)=x^{p^{\ell}}+c\in \mathbb{Z}[x]\subset \mathcal{O}_{K}[x]$ having $N_{c}(p)\in [2, \ell]$ is also zero:

\begin{cor}\label{9.2}
Let $K\slash \mathbb{Q}$ be any number field of degree $n\geq 2$ with the ring of integers $\mathcal{O}_{K}$, and in which any prime integer $p\geq 3$ is inert. Let $\ell \geq 1$ be any fixed integer. Then the density of monic integer polynomials $\varphi_{p^{\ell},c}(x) = x^{p^{\ell}} + c\in \mathcal{O}_{K}[x]$ with $N_{c}(p)\in [2, \ell]$ exists and is equal to $0 \%$ as $c\to \infty$. More precisely, we have 
\begin{center}
    $\lim\limits_{c\to\infty} \Large{\frac{\# \{\varphi_{p^{\ell},c}(x)\in \mathbb{Z}[x] \ : \ 3\leq p\leq c \ \text{and} \ N_{c}(p)\in [2, \ell]\}}{\Large{\# \{\varphi_{p^{\ell},c}(x) \in \mathbb{Z}[x] \ : \ 3\leq p\leq c \}}}} = \ 0.$
\end{center}
\end{cor}
\begin{proof}
Applying a similar argument as in the Proof of Corollary \ref{9.1}, we obtain that the limit is equal to zero.
\end{proof}

\noindent Recall in Corollary \ref{9.1} or \ref{9.2} that a density of $0\%$ of monic integer polynomials $\varphi_{p^{\ell},c}(x)\in\mathcal{O}_{K}[x]$ have the number $N_{c}(p) = p$ or $N_{c}(p) \in [2, \ell]$, respectively; and so the density of monic integer polynomials $\varphi_{p^{\ell},c}(x)-x\in \mathcal{O}_{K}[x]$ that are reducible modulo $p$ is $0\%$. So now, we also wish to determine: \say{\textit{For any fixed integer $\ell \geq 1$, what is the density of monic integer polynomials $\varphi_{p^{\ell},c}(x)\in \mathcal{O}_{K}[x]$ with no integral fixed points modulo $p\mathbb{Z}$?}} The following corollary shows that for any fixed integer $\ell \geq 1$, the probability of choosing randomly a monic integer polynomial $\varphi_{p^{\ell},c}(x)\in \mathbb{Z}[x]\subset \mathcal{O}_{K}[x]$ such that $\mathbb{Q}[x]\slash (\varphi_{p^{\ell}, c}(x)-x)$ is an algebraic number field of degree $p^{\ell}$ is 1: 
\begin{cor}\label{9.3}
Let $K\slash \mathbb{Q}$ be any number field of degree $n\geq 2$ with the ring of integers $\mathcal{O}_{K}$, and in which any prime integer $p\geq 3$ is inert. Let $\ell \geq 1$ be any fixed integer. Then the density of monic integer polynomials $\varphi_{p^{\ell},c}(x)=x^{p^{\ell}} + c\in \mathcal{O}_{K}[x]$ with $N_{c}(p) = 0$ exists and is equal to $100 \%$ as $c\to \infty$. More precisely, we have
\begin{center}
    $\lim\limits_{c\to\infty} \Large{\frac{\# \{\varphi_{p^{\ell},c}(x)\in \mathbb{Z}[x] \ : \ 3\leq p\leq c \ and \ N_{c}(p) \ = \ 0 \}}{\Large{\# \{\varphi_{p^{\ell},c}(x) \in \mathbb{Z}[x] \ : \ 3\leq p\leq c \}}}} = \ 1.$
\end{center}
\end{cor}
\begin{proof}
Since $N_{c}(p) = p$ or $N_{c}(p)\in [2, \ell]$ or $N_{c}(p) = 0$ for any given prime $p\geq 3$ and since we also proved the densities in Corollary \ref{9.1} and \ref{9.2}, we now obtain the desired density (i.e., the desired limit is equal to 1). 
\end{proof}
\noindent Note that the foregoing corollary also shows that for any fixed $\ell \in \mathbb{Z}_{\geq 1}$, there are infinitely many polynomials $\varphi_{p^{\ell},c}(x)$ over $\mathbb{Z}$ (and hence over $\mathbb{Q}$) such that for $f(x) = \varphi_{p^{\ell},c}(x)-x = x^{p^{\ell}}-x+c$, we then have that the quotient ring $\mathbb{Q}_{f} = \mathbb{Q}[x]\slash (f(x))$ induced by $f$ is an algebraic  number field of odd degree deg$(f) = p^{\ell}$. Comparing the densities in Corollaries \ref{9.1}, \ref{9.2} and \ref{9.3}, one may also observe that in the whole family of monic integer polynomials $\varphi_{p^{\ell},c}(x) = x^{p^{\ell}} +c$, almost all such monic polynomials have no integral fixed points modulo $p$ (i.e., have no rational roots); and hence almost all such monic polynomials $f(x)$ are irreducible over $\mathbb{Q}$. Consequently, this may also imply that the average value of $N_{c}(p)$ in the whole family of polynomials $\varphi_{p^{\ell},c}(x)\in \mathbb{Z}[x]$ is zero.

As before, we now also wish to focus our counting on $\mathbb{Z}\subset \mathbb{Z}_{p}$ and then determine: \say{\textit{For any fixed $\ell \in \mathbb{Z}_{ \geq 1}$, what is the density of monic integer polynomials $\varphi_{p^{\ell},c}(x)\in \mathbb{Z}_{p}[x]$ with $p$ fixed points modulo $p$?}} The following corollary shows that very few $p$-adic integer polynomials $\varphi_{p^{\ell},c}(x)\in \mathbb{Z}[x]$ have $p$ distinct fixed points modulo $p$:
\begin{cor}\label{9.4}
Let $p\geq 3$ be any prime, and $\ell \geq 1$ be any fixed integer. Then the density of integer polynomials $\varphi_{p^{\ell},c}(x) = x^{p^{\ell}} + c\in \mathbb{Z}_{p}[x]$ with $X_{c}(p) = p$ exists and is equal to $0 \%$ as $c\to \infty$. That is, we have 
\begin{center}
    $\lim\limits_{c\to\infty} \Large{\frac{\# \{\varphi_{p^{\ell},c}(x)\in \mathbb{Z}[x] \ : \ 3\leq p\leq c \ \text{and} \ X_{c}(p) \ = \ p\}}{\Large{\# \{\varphi_{p^{\ell},c}(x) \in \mathbb{Z}[x] \ : \ 3\leq p\leq c \}}}} = \ 0.$
\end{center}
\end{cor}

\begin{proof}
By applying a similar argument as in Proof of Cor. \ref{9.1}, we then obtain that the limit is equal to zero.
\end{proof}

\noindent Note that we may also interpret Corollary \ref{9.4} as saying that for any fixed $\ell \in \mathbb{Z}_{ \geq 1}$, the probability of choosing randomly a $p$-adic integer polynomial $\varphi_{p^{\ell},c}(x)\in \mathbb{Z}[x]\subset \mathbb{Z}_{p}[x]$ having $p$ distinct fixed points modulo $p$ is zero.

Similarly, we also have the following corollary showing that for any fixed $\ell \in \mathbb{Z}_{\geq 1}$, the probability of choosing randomly a $p$-adic integer polynomial $\varphi_{p^{\ell},c}(x)=x^{p^{\ell}}+c$ in $\mathbb{Z}[x]\subset \mathbb{Z}_{p}[x]$ with $X_{c}(p)\in [2, \ell]$ is also zero:

\begin{cor}\label{9.5}
Let $p\geq 3$ be any prime, and $\ell \geq 1$ be any fixed integer. The density of integer polynomials $\varphi_{p^{\ell},c}(x) = x^{p^{\ell}} + c\in \mathbb{Z}_{p}[x]$ with $X_{c}(p)\in [2, \ell]$ exists and is equal to $0 \%$ as $c\to \infty$. More precisely, we have 
\begin{center}
    $\lim\limits_{c\to\infty} \Large{\frac{\# \{\varphi_{p^{\ell},c}(x)\in \mathbb{Z}[x] \ : \ 3\leq p\leq c \ \text{and} \ X_{c}(p)\in [2, \ell]\}}{\Large{\# \{\varphi_{p^{\ell},c}(x) \in \mathbb{Z}[x] \ : \ 3\leq p\leq c \}}}} = \ 0.$
\end{center}
\end{cor}
\begin{proof}
By the same reasoning as in the Proof of Cor. \ref{9.4}, we obtain that the limit exists and is equal to zero.
\end{proof}

As before, we recall in Corollary \ref{9.4} or \ref{9.5} that a density of $0\%$ of $p$-adic integer polynomials $\varphi_{p^{\ell},c}(x)\in\mathbb{Z}[x]$ have $X_{c}(p) = p$ or $X_{c}(p) \in [2, \ell]$, resp.; and so the density of $p$-adic integer polynomials $\varphi_{p^{\ell},c}(x)-x\in \mathbb{Z}[x]$ that are reducible modulo $p$ is $0\%$. So now, we also wish to determine: \say{\textit{For any fixed $\ell \in \mathbb{Z}_{\geq 1}$, what is the density of monic $p$-adic integer polynomials $\varphi_{p^{\ell},c}(x)\in \mathbb{Z}[x] \subset \mathbb{Z}_{p}[x]$ with no integral fixed points modulo $p$?}} The following corollary shows that for any fixed $\ell \in \mathbb{Z}_{\geq 1}$, the probability of choosing randomly a monic $p$-adic integer polynomial $\varphi_{p^{\ell},c}(x)\in \mathbb{Z}[x]\subset \mathbb{Z}_{p}[x]$ so that $\mathbb{Q}[x]\slash (\varphi_{p^{\ell}, c}(x)-x)$ is a degree-$p^{\ell}$ algebraic number field is 1: 
\begin{cor}\label{9.6}
Let $p\geq 3$ be a prime integer, and $\ell \geq 1$ be any fixed integer. The density of integer polynomials $\varphi_{p^{\ell},c}(x) = x^{p^{\ell}} + c\in \mathbb{Z}_{p}[x]$ with $X_{c}(p) = 0$ exists and is equal to $100 \%$ as $c\to \infty$. More precisely, we have 
\begin{center}
    $\lim\limits_{c\to\infty} \Large{\frac{\# \{\varphi_{p^{\ell},c}(x)\in \mathbb{Z}[x] \ : \ 3\leq p\leq c \ and \ X_{c}(p) \ = \ 0 \}}{\Large{\# \{\varphi_{p^{\ell},c}(x) \in \mathbb{Z}[x] \ : \ 3\leq p\leq c \}}}} = \ 1.$
\end{center}
\end{cor}
\begin{proof}
Since the number $X_{c}(p) = p$ or $X_{c}(p)\in [2, \ell]$ or $X_{c}(p)=0$ for any given prime $p\geq 3$ and since we also proved the densities in Corollary \ref{9.4} and \ref{9.5}, it then immediately follows that the desired limit is equal to 1. 
\end{proof}

\noindent As before, the foregoing corollary shows that for any fixed $\ell \in \mathbb{Z}_{\geq 1}$, there are infinitely many $p$-adic integer polynomials $\varphi_{p^{\ell},c}(x)\in \mathbb{Z}[x]\subset\mathbb{Q}[x]$ such that for $f(x) = x^{p^{\ell}}-x+c$, the attached ring $\mathbb{Q}_{f} = \mathbb{Q}[x]\slash (f(x))$ is a number field of degree $p^{\ell}$. Again, comparing the densities in Cor. \ref{9.4}, \ref{9.5} and \ref{9.6}, we may then observe that in the whole family of monic integer polynomials $\varphi_{p^{\ell},c}(x) = x^{p^{\ell}} +c$, almost all such monics have no integral fixed point modulo $p$ (i.e., have no $\mathbb{Q}$-roots); and so almost all polynomials $f(x)$ are irreducible over $\mathbb{Q}$. This may also imply that the average value of $X_{c}(p)$ in the whole family of polynomials $\varphi_{p^{\ell},c}(x)\in \mathbb{Z}[x]\subset \mathbb{Z}_{p}[x]$ is zero.

\section{The Densities of Integer Polynomials $\varphi_{(p-1)^{\ell},c}(x)\in \mathbb{Z}_{p}[x]$ with $Y_{c}(p) = 1$ or $2$ or $0$}\label{sec10}

As in Section \ref{sec9}, we also wish to determine: \say{\textit{For any fixed $\ell \in \mathbb{Z}_{ \geq 1}$, what is the density of monic integer polynomials $\varphi_{(p-1)^{\ell},c}(x) = x^{(p-1)^{\ell}} + c\in \mathbb{Z}_{p}[x]$ with two points modulo $p$?}} The following corollary shows that there are also very few $p$-adic integer polynomials $\varphi_{(p-1)^{\ell},c}(x)\in \mathbb{Z}[x]$ with two distinct fixed points modulo $p$:

\begin{cor}\label{10.1}
Let $p\geq 5$ be any prime, and $\ell \geq 1$ be any fixed integer. The density of integer polynomials $\varphi_{(p-1)^{\ell},c}(x) = x^{(p-1)^{\ell}} + c\in \mathbb{Z}_{p}[x]$ with $Y_{c}(p) = 2$ exists and is equal to $0 \%$ as $c\to \infty$. Specifically, we have 
\begin{center}
    $\lim\limits_{c\to\infty} \Large{\frac{\# \{\varphi_{(p-1)^{\ell},c}(x) \in \mathbb{Z}[x]\ : \ 5\leq p\leq c \ and \ Y_{c}(p) \ = \ 2\}}{\Large{\# \{\varphi_{(p-1)^{\ell},c}(x) \in \mathbb{Z}[x]\ : \ 5\leq p\leq c \}}}} = \ 0.$
\end{center}
\end{cor}
\begin{proof}
By applying a similar argument as in the Proof of Corollary \ref{9.1}, we then obtain the limit as desired.
\end{proof}

We also have the following corollary showing that for any fixed $\ell \in \mathbb{Z}_{\geq 1}$, the probability of choosing randomly a monic $p$-adic integer polynomial $\varphi_{(p-1)^{\ell},c}(x)\in\mathbb{Z}[x]\subset \mathbb{Z}_{p}[x]$ with $Y_{c}(p) = 1$ exists and is also zero:

\begin{cor}\label{10.2}
Let $p\geq 5$ be any prime, and $\ell \geq 1$ be any fixed integer. The density of integer polynomials $\varphi_{(p-1)^{\ell},c}(x) = x^{(p-1)^{\ell}} + c\in \mathbb{Z}_{p}[x]$ with $Y_{c}(p) = 1$ exists and is equal to $0 \%$ as $c\to \infty$. That is, we have 
\begin{center}
    $\lim\limits_{c\to\infty} \Large{\frac{\# \{\varphi_{(p-1)^{\ell},c}(x) \in \mathbb{Z}[x]\ : \ 5\leq p\leq c \ and \ Y_{c}(p) \ = \ 1\}}{\Large{\# \{\varphi_{(p-1)^{\ell},c}(x) \in \mathbb{Z}[x]\ : \ 5\leq p\leq c \}}}} = \ 0.$
\end{center}
\end{cor}
\begin{proof}
By applying a similar argument as in [\cite{BK2}, Proof of Cor. 6.2], we then obtain the limit as desired.
\end{proof}

Similarly, we may also recall in Corollary \ref{10.1} or \ref{10.2} that a density of $0\%$ of monic integer polynomials $\varphi_{(p-1)^{\ell},c}(x)\in \mathbb{Z}_{p}[x]$ have $Y_{c}(p) = 2$ or $1$, respectively; and so the density of integer polynomials $\varphi_{(p-1)^{\ell},c}(x)-x\in \mathbb{Z}_{p}[x]$ that are reducible modulo $p$ is $0\%$. So now, we also wish to determine: \say{\textit{For any fixed $\ell \in \mathbb{Z}_{\geq 1}$, what is the density of integer polynomials $\varphi_{(p-1)^{\ell},c}(x)\in \mathbb{Z}_{p}[x]$ with no integral fixed points modulo $p$?}} The following corollary shows that for any fixed $\ell \in \mathbb{Z}_{\geq 1}$, the probability of choosing randomly a monic $p$-adic integer polynomial $\varphi_{(p-1)^{\ell},c}(x)\in \mathbb{Z}[x]$ such that $\mathbb{Q}[x]\slash (\varphi_{(p-1)^{\ell}, c}(x)-x)$ is a number field of degree $(p-1)^{\ell}$ is also 1:
\begin{cor} \label{10.3}
Let $p\geq 5$ be a prime integer, and $\ell \geq 1$ be any integer. Then the density of polynomials $\varphi_{(p-1)^{\ell}, c}(x) = x^{(p-1)^{\ell}}+c\in \mathbb{Z}_{p}[x]$ with $Y_{c}(p) = 0$ exists and is equal to $100 \%$ as $c\to \infty$. Specifically, we have 
\begin{center}
    $\lim\limits_{c\to\infty} \Large{\frac{\# \{\varphi_{(p-1)^{\ell}, c}(x)\in \mathbb{Z}[x] \ : \ 5\leq p\leq c \ and \ Y_{c}(p) \ = \ 0 \}}{\Large{\# \{\varphi_{(p-1)^{\ell},c}(x) \in \mathbb{Z}[x] \ : \ 5\leq p\leq c \}}}} = \ 1.$
\end{center}
\end{cor}
\begin{proof}
Recall that the number $Y_{c}(p) = 1, 2$ or $0$ for any given prime $p\geq 5$ and since we also proved the densities in Cor. \ref{10.1} and \ref{10.2}, we now obtain the desired density (i.e., we get that the limit exists and is equal to 1).
\end{proof}

\noindent As before, the foregoing corollary shows that for any fixed $\ell \in \mathbb{Z}_{\geq 1}$, there are infinitely many $p$-adic integer monic polynomials $\varphi_{(p-1)^{\ell},c}(x)\in \mathbb{Z}[x]$ such that for $g(x) = x^{(p-1)^{\ell}}-x+c$, we then have that the quotient ring $\mathbb{Q}_{g} = \mathbb{Q}[x]\slash (g(x))$ induced by $g$ is an algebraic number field of even degree $(p-1)^{\ell}$. Comparing the densities in Corollary \ref{10.1}, \ref{10.2} and \ref{10.3}, we may again observe that in the whole family of monic integer polynomials $\varphi_{(p-1)^{\ell},c}(x) = x^{(p-1)^{\ell}} +c$, almost all such monics have no $p$-adic integral fixed point modulo $p$ (i.e., have no rational roots); and so almost all monic polynomials $g$ are irreducible over $\mathbb{Q}$. Consequently, this may then also imply that the average value of $Y_{c}(p)$ in the whole family of polynomials $\varphi_{(p-1)^{\ell},c}(x)\in \mathbb{Z}[x]\subset \mathbb{Z}_{p}[x]$ is also zero.

A recurring theme in algebraic number theory that shows up when one is studying an algebraic number field $K$ of some interest, is that one that must also simultaneously try to describe precisely what the associated ring $\mathcal{O}_{K}$ of integers is; which though is classically known to describe very naturally the arithmetic of the underlying number field, however, accessing $\mathcal{O}_{K}$ from a computational point of view is known to be an extremely involved computing problem. In our case here, it then follows $\mathbb{Q}_{f}$ has a ring of integers $\mathcal{O}_{\mathbb{Q}_{f}}$, and moreover by Bhargava-Shankar-Wang \cite{sch1}, we then also obtain the following corollary showing the probability of choosing randomly a monic $p$-adic integer polynomial $f\in \mathbb{Z}[x]$ arising from a polynomial discrete dynamical in Section \ref{sec2} or \ref{sec3} (and ascertained by Corollary \ref{9.3} or \ref{9.6}), such that $\mathbb{Z}[x]\slash (f(x))$ is the ring of integers of $\mathbb{Q}_{f}$, is $\approx 60.7927\%$:
\begin{cor}\label{10.4}
Assume Corollary \ref{9.3} or \ref{9.6}. When integer polynomials $f(x)$ are ordered by height $H(f)$ as in \textnormal{\cite{sch1}}, the density of monic polynomials $f(x)$ such that $\mathbb{Z}_{f}=\mathbb{Z}[x]\slash (f(x))$ is the ring of integers of $\mathbb{Q}_{f}$ is $\zeta(2)^{-1}$. 
\end{cor}

\begin{proof}
Since from Corollary \ref{9.3} or \ref{9.6} we know that for any fixed integer $\ell \geq 1$, there are infinitely many polynomials $f(x)\in \mathbb{Z}[x]\subset \mathbb{Q}[x]$ such that $\mathbb{Q}_{f} = \mathbb{Q}[x]\slash (f(x))$ is a number field of degree $p^{\ell}$; and moreover associated to $\mathbb{Q}_{f}$ is the ring $\mathcal{O}_{\mathbb{Q}_{f}}$ of integers. This then means that the set of irreducible monic integer polynomials $f(x) = x^{p^{\ell}}-x + c\in \mathbb{Q}[x]$ such that $\mathbb{Q}_{f}$ is an algebraic number field of degree $p^{\ell}$ is not empty. So now, applying a remarkable theorem of Bhargava-Shankar-Wang [\cite{sch1}, Theorem 1.2] to the underlying family of monic integer polynomials $f(x)$ ordered by height $H(f) = |c|^{1\slash p^{\ell}}$ such that the ring of integers $\mathcal{O}_{\mathbb{Q}_{f}} = \mathbb{Z}[x]\slash (f(x))$, it then follows that the density of such monic polynomials $f(x)\in \mathbb{Z}[x]$ is equal to $\zeta(2)^{-1} \approx 60.7927\%$, as needed.
\end{proof}

Similarly, every number field $\mathbb{Q}_{g}$ induced by a polynomial $g$, is naturally equipped with the ring of integers $\mathcal{O}_{\mathbb{Q}_{g}}$, and which as before may be difficult to compute in practice. In this case, with now the polynomials $g$, we again take great advantage of [\cite{sch1}, Theorem 1.2] to then also show in the following corollary that the probability of choosing randomly a monic $p$-adic integer $g\in \mathbb{Z}[x]$ arising from a polynomial discrete dynamical in Section \ref{sec4} (and also ascertained by Corollary \ref{10.3}), such that $\mathbb{Z}[x]\slash (g(x))$ is the ring of integers of $\mathbb{Q}_{g}$, is also $\approx 60.7927\%$:

\begin{cor}\label{7.4}
Assume Corollary \ref{10.3}. When monic polynomials $g(x)\in \mathbb{Z}[x]$ are ordered by height $H(g)$ as in \textnormal{\cite{sch1}}, the density of monic polynomials $g(x)$ such that $\mathbb{Z}_{g}=\mathbb{Z}[x]\slash (g(x))$ is the ring of integers of $\mathbb{Q}_{g}$ is $\zeta(2)^{-1}$. 
\end{cor}

\begin{proof}
By applying a similar argument as in the Proof of Corollary \ref{10.4}, we then obtain the desired density.
\end{proof}

\section{On Local Densities of $f, g\in \mathbb{Z}_{p}[x]$ inducing Maximal orders in Associated Fields}

Recall in algebraic number theory that an \say{\textit{order}} in an algebraic number field $K$ is any subring $R\subset K$ that is free of rank $n = [K :\mathbb{Q}]$ over $\mathbb{Z}$. It is well-known in algebraic number theory that the ring of integers $\mathcal{O}_{K}$ in any algebraic number field $K$ is the union of all orders in $K$, and moreover the ring $\mathcal{O}_{K}$ is not only an order in $K$ but is in fact also the maximal order in $K$. (See more of these important facts in Stevenhagen's paper \cite{pet}.) 

As with many other arithmetic objects whose distributions are of great importance and interest of study in arithmetic statistics, orders are among those objects whose statistics is again of serious importance and interest in arithmetic statistics. (For instance, see the seminal work \cite{gav} of Bhargava on counting all orders in $S_{4}$-quartic fields.) So now, recall from Corollary \ref{9.6} the existence of infinitely many monic irreducible polynomials $f(x)$ over $\mathbb{Z}\subset \mathbb{Z}_{p}\subset \mathbb{Q}_{p}$ such that $\mathbb{Q}_{p(f)} := \mathbb{Q}_{p}[x]\slash (f(x))$ is a degree-$p^{\ell}$ field extension of $\mathbb{Q}_{p}$ (i.e., $\mathbb{Q}_{p(f)}\slash \mathbb{Q}_{p}$ is an algebraic $p$-adic number field and thus has ring of integers $\mathcal{O}_{\mathbb{Q}_{p(f)}}$). Meanwhile, we may also recall that the second part of Theorem \ref{3.3} (i.e., the part in which we proved $X_{c}(p) = 0$ for every $c\not \in p\mathbb{Z}_{p}$) implies that the polynomial $f(x) = x^{p^{\ell}} - x + c \in \mathbb{Z}_{p}[x]\subset \mathbb{Q}_{p}[x]$ is irreducible modulo fixed $p\mathbb{Z}_{p}$; and so we may as before to each such irreducible monic $p$-adic integer polynomial $f\in \mathbb{Q}_{p}[x]$ associate a field, say, $\mathbb{Q}_{p(f)}$. Now inspired by \cite{gav, as, sch1}, we may also ask for the density of monic $p$-adic integer polynomials $f$ arising from a polynomial discrete dynamical system in Section \ref{sec3}, such that $\mathbb{Z}_{p}[x]\slash (f(x))$ is the maximal order in $\mathbb{Q}_{p(f)}$. A $p$-adic density result of Hendrik Lenstra \cite{as} applied on our irreducible monic $p$-adic integer polynomials, enables us to obtain the following corollary showing the probability of choosing randomly a monic polynomial $f(x)$ over $\mathbb{Z}_{p}$ such that $\mathbb{Z}_{p}[x]\slash (f(x))$ is the maximal order in $\mathbb{Q}_{p(f)}$; and moreover this probability tends to 1 in the large-$p$ limit: 

\begin{cor}\label{10.6}
Assume Corollary \ref{9.6} or second part of Theorem \ref{3.3}. Then the density of monic integer polynomials $f(x)$ over $\mathbb{Z}_{p}$ ordered by height $H(f)$ as defined in \textnormal{\cite{as}} such that $\mathbb{Z}_{p(f)} = \mathbb{Z}_{p}[x]\slash (f(x))$ is the maximal order in $\mathbb{Q}_{p(f)}$ exists and is equal to $\rho_{\text{deg(f)}}(p):= 1 - p^{-2}$. Moreover, this density tends to $1$ as $p\to \infty$.   
\end{cor}
\begin{proof}
To see the density, we recall from Corollary \ref{9.6} the existence of infinitely many polynomials $f(x)\in \mathbb{Z}[x]\subset \mathbb{Z}_{p}[x]\subset \mathbb{Q}_{p}[x]$ such that $\mathbb{Q}_{p(f)}\slash \mathbb{Q}_{p}$ is a number field of degree $p^{\ell}$, or recall that the second part of Theorem \ref{3.3} implies that the polynomial $f(x) = x^{p^{\ell}} - x + c \in \mathbb{Z}_{p}[x]\subset \mathbb{Q}_{p}[x]$ is irreducible modulo any fixed $p\mathbb{Z}_{p}$ for every coefficient $c\not \in \mathbb{Z}_{p}$, and so induces a degree-$p^{\ell}$ number field $\mathbb{Q}_{p(f)}\slash \mathbb{Q}_{p}$. This then means that the set of fields $\mathbb{Q}_{p(f)}$ is not empty. So now, as pointed out in the work of Bhargava-Shankar-Wang [\cite{sch1}, Page 2], we may then apply [\cite{as}, Prop. 3.5] on the family of irreducible polynomials $f(x) = x^{p^{\ell}}-x+c \in \mathbb{Z}_{p}[x]$ resulting from Corollary \ref{9.6} or from the second part of Theorem \ref{3.3} when we've ordered them by height $H(f)$ as in \cite{as}, to then obtain the first part. Letting $p\to \infty$, we then obtain that the function $\rho_{\text{deg($f$)}}(p)\to 1$, as also desired. 
\end{proof}

Similarly, we may also recall that from Corollary \ref{10.3} that there are infinitely many monic irreducible polynomials $g(x)$ over $\mathbb{Z}\subset \mathbb{Z}_{p}\subset \mathbb{Q}_{p}$ such that $\mathbb{Q}_{p(g)} := \mathbb{Q}_{p}[x]\slash (g(x))$ is a degree-$(p-1)^{\ell}$ field extension of $\mathbb{Q}_{p}$ (and again $\mathbb{Q}_{p(g)}\slash \mathbb{Q}_{p}$ is an algebraic $p$-adic number field and so has ring of integers $\mathcal{O}_{\mathbb{Q}_{p(g)}}$)). Moreover, we may also recall that the second part of Theorem \ref{4.3} the part in which we proved $Y_{c}(p) = 0$ for every $c\equiv -1\ (\text{mod} \ p\mathbb{Z}_{p})$) implies $g(x) = x^{(p-1)^{\ell}} - x + c \in \mathbb{Z}_{p}[x]\subset \mathbb{Q}_{p}[x]$ is irreducible modulo any fixed $p\mathbb{Z}_{p}$; and so we may to each such irreducible monic $p$-adic integer polynomial $g(x)\in \mathbb{Q}_{p}[x]$ associate a field, say, $\mathbb{Q}_{p(g)}$. In this case, with now the polynomials $g$, we may again wish to determine the density of monic polynomials $g$ arising from a polynomial discrete dynamical system in Section \ref{sec4}, such that $\mathbb{Z}_{p}[x]\slash (g(x))$ is the maximal order in $\mathbb{Q}_{p(g)}$. To do so, we again take great advantage of the aforementioned $p$-adic density result of Lenstra \cite{as} to then obtain the following corollary showing the probability of choosing randomly a $p$-adic integer polynomial $g$ such that $\mathbb{Z}_{p}[x]\slash (g(x))$ is the maximal order in $\mathbb{Q}_{p(g)}$; and moreover this probability again tends to 1 as $p\to \infty$: 

\begin{cor}
Assume Corollary \ref{10.3} or second part of Theorem \ref{4.3}. Then the density of monic $p$-adic integer polynomials $g(x)$ over $\mathbb{Z}_{p}$ ordered by height $H(g)$ as defined in \textnormal{\cite{as}} such that $\mathbb{Z}_{p(g)} = \mathbb{Z}_{p}[x]\slash (g(x))$ is the maximal order in $\mathbb{Q}_{p(g)}$ exists and is equal to $\rho_{\text{deg(g)}}(p):= 1 - p^{-2}$. Moreover, this density tends to $1$ as $p\to \infty$.   
\end{cor}
\begin{proof}
Applying a similar argument as in the Proof of Corollary \ref{10.6}, we then obtain the density, as desired. 
\end{proof}

\section{On the Number of Number fields $K_{f}\slash \mathbb{Q}$ with Bounded Absolute Discriminant}\label{sec11}
\subsection{On Fields $\mathbb{Q}_{f}\slash \mathbb{Q}$ with Bounded Absolute Discriminant and Prescribed Galois group}
Recall we saw from Corollary \ref{9.3} that there is an infinite family of irreducible monic integer polynomials $f(x) = x^{p^{\ell}}-x + c$ such that the field $\mathbb{Q}_{f}$ induced by $f$ is an algebraic number field of odd prime degree $p^{\ell}$. Moreover, we also saw from Corollary \ref{9.6} that we can always find an infinite family of irreducible monic integer polynomials $g(x) = x^{(p-1)^{\ell}}-x + c$ such that the field extension $\mathbb{Q}_{g}$ over $\mathbb{Q}$ induced by $g$ is an algebraic number field of even degree $(p-1)^{\ell}$. In this subsection, we wish to study the problem of counting number fields; a problem that's originally from and is of very serious interest in arithmetic statistics. Inspired (as in \cite{BK2}) by Bhargava-Shankar-Wang (BSW) \cite{sch}, we then wish to count here the number of primitive number fields $\mathbb{Q}_{f}$ induced by irreducible monic polynomials $f\in \mathbb{Z}[x]$ arising from a polynomial discrete dynamical in Section \ref{sec2} or \ref{sec3} (and ascertained by Corollary \ref{9.3} or \ref{9.6}), with bounded absolute discriminant. To this end, we then obtain:
\begin{cor}\label{11.1}
Assume Corollary \ref{9.3} or \ref{9.6} and let $\mathbb{Q}_{f} = \mathbb{Q}[x]\slash (f(x))$ be a primitive number field with discriminant $\Delta(\mathbb{Q}_{f})$. Up to isomorphism classes of number fields, we have $\# \{ \mathbb{Q}_{f} : |\Delta(\mathbb{Q}_{f})| < X \} \ll  X^{p^{\ell}\slash(2p^{\ell}-2)}$. 
\end{cor}

\begin{proof}
From Corollary \ref{9.3} or \ref{9.6}, there are infinitely many irreducible monic $p$-adic integer polynomials $f(x)\in \mathbb{Z}[x]\subset \mathbb{Q}[x]$ such that $\mathbb{Q}_{f}=\mathbb{Q}[x]\slash (f(x))$ is a number field of degree $p^{\ell}$. This then means that the set of number fields $\mathbb{Q}_{f}$ is not empty. Now assuming $\mathbb{Q}_{f}$ are primitive, we may apply an argument of (BSW)[\cite{sch}, Page 2] to show that up to isomorphism classes of number fields the total number of such primitive fields $\mathbb{Q}_{f}$ with $|\Delta(\mathbb{Q}_{f})| < X$, is bounded above by the number of monics $f\in \mathbb{Z}[x]$ with height $H(f) \ll X^{1\slash(2p^{\ell}-2)}$ and vanishing subleading coefficient. So now, since $H(f) = |c|^{1\slash p^{\ell}}$ (as by definition of height in \cite{sch}) and so $|c| \ll X^{p^{\ell}\slash(2p^{\ell}-2)}$, we then obtain that $\# \{f\in \mathbb{Z}[x] : H(f) \ll X^{1\slash(2p^{\ell}-2)} \} = \# \{f\in \mathbb{Z}[x] : |c| \ll X^{p^{\ell}\slash(2p^{\ell}-2)} \} \ll X^{p^{\ell}\slash(2p^{\ell}-2)}$. It then follows $\# \{ \mathbb{Q}_{f} : |\Delta(\mathbb{Q}_{f})| < X \} \ll  X^{p^{\ell}\slash(2p^{\ell}-2)}$ up to isomorphism classes of number fields, as desired.
\end{proof}

By applying a similar argument as in Cor. \ref{11.1}, we then also obtain the following corollary on the number of primitive number fields $\mathbb{Q}_{g}$ induced by irreducible monic $p$-adic integer polynomials $g\in \mathbb{Z}[x]$ arising from a polynomial discrete dynamical in Section \ref{sec4} (and ascertained by Cor. \ref{10.3}), with bounded absolute discriminant:

\begin{cor}\label{11.2}
Assume Corollary \ref{10.3} and let $\mathbb{Q}_{g} = \mathbb{Q}[x]\slash (g(x))$ be a primitive number field with discriminant $\Delta(\mathbb{Q}_{g})$. Then up to isomorphism classes of number fields, we have $\# \{ \mathbb{Q}_{g} : |\Delta(\mathbb{Q}_{g})| < X \} \ll  X^{(p-1)^{\ell}\slash(2(p-1)^{\ell}-2)}$. 
\end{cor}

\begin{proof}
Applying a similar argument as in the Proof of Corollary \ref{11.1}, we then obtain the count, as desired.
\end{proof}

We recall in algebraic number theory that a number field $K$ is called \say{\textit{monogenic}} if there exists an algebraic number $\alpha \in K$ such that the ring $\mathcal{O}_{K}$ of integers is the subring $\mathbb{Z}[\alpha]$ generated by $\alpha$ over $\mathbb{Z}$, i.e., $\mathcal{O}_{K}= \mathbb{Z}[\alpha]$. Now recall in Corollary \ref{11.1} that we counted primitive number fields $\mathbb{Q}_{f}$ with absolute discriminant $|\Delta(\mathbb{Q}_{f})| < X$. So now, we wish to count the number of number fields $\mathbb{Q}_{f}$ induced by irreducible monic polynomials $f\in \mathbb{Z}[x]$ arising from a polynomial discrete dynamical in Section \ref{sec2} (and ascertained by Corollary \ref{9.3}), that are monogenic with $|\Delta(\mathbb{Q}_{f})| < X$ and with Galois group Gal$(\mathbb{Q}_{f}\slash \mathbb{Q})$ equal to the symmetric group $S_{p^{\ell}}$. 
By taking great advantage of a result of Bhargava-Shankar-Wang [\cite{sch1}, Corollary 1.3], we then obtain here:

\begin{cor}\label{11.3}
Assume Corollary \ref{9.3}. Then the number of isomorphism classes of algebraic number fields $\mathbb{Q}_{f}$ of degree $p^{\ell}$ and with $|\Delta(\mathbb{Q}_{f})| < X$ that are monogenic and have associated Galois group $S_{p^{\ell}}$ is $\gg X^{\frac{1}{2} + \frac{1}{p^{\ell}}}$.
\end{cor}

\begin{proof}
Applying a similar argument as in [\cite{BK2}, Proof of Cor. 8.3], we then obtain the lower bound, as desired.
\end{proof}

Similarly, by again taking great advantage of that same result of (BSW)[\cite{sch1}, Corollary 1.3], we then also obtain the following corollary on the number of number fields $\mathbb{Q}_{g}$ induced by irreducible monic $p$-adic integer polynomials $g\in \mathbb{Z}[x]$ arising from a polynomial discrete dynamical in Section \ref{sec4} (and ascertained by Corollary \ref{10.3}), that are monogenic with $|\Delta(\mathbb{Q}_{g})| < X$ and Galois group Gal$(\mathbb{Q}_{g}\slash \mathbb{Q})$ equal to symmetric group $S_{(p-1)^{\ell}}$: 

\begin{cor}
Assume Corollary \ref{10.3}. The number of isomorphism classes of algebraic number fields $\mathbb{Q}_{g}$ of degree $(p-1)^{\ell}$ and $|\Delta(\mathbb{Q}_{g})| < X$ that are monogenic and have associated Galois group $S_{(p-1)^{\ell}}$ is $\gg X^{\frac{1}{2} + \frac{1}{(p-1)^{\ell}}}$.
\end{cor}

\begin{proof}
Applying a similar argument as in Proof of Corollary \ref{11.3}, we then obtain the lower bound, as desired.
\end{proof}

\subsection{On Fields $K_{f}\slash \mathbb{Q}$ with Bounded Absolute Discriminant and Prescribed Galois group}
Recall that we proved in Cor. \ref{9.3} or \ref{9.6} the existence of an infinite family of irreducible monic integer polynomials $f(x) = x^{p^{\ell}} - x + c\in \mathbb{Q}[x]\subset K[x]$ for any fixed $\ell \in \mathbb{Z}_{\geq 1}$; and more to this, we may also recall that the second part of Theorem \ref{2.3} (i.e., the part in which we proved that $N_{c}(p) = 0$ for every $c\not \equiv 0\ (\text{mod} \ p\mathcal{O}_{K})$) implies that $f(x) = x^{p^{\ell}} - x + c \in \mathcal{O}_{K}[x]\subset K[x]$ is irreducible modulo any fixed prime ideal $p\mathcal{O}_{K}$. So now, as in Section \ref{sec9}, we may to each irreducible polynomial $f$ associate a field $K_{f} = K[x]\slash (f(x))$, which is again a number field of deg$(f) = p^{\ell}$ over $K$; and more to this, we also recall from algebraic number theory that associated to $K_{f}$ is an integer Disc$(K_{f})$ called the discriminant. Moreover, bearing in mind that we now obtain an inclusion $\mathbb{Q}\hookrightarrow K \hookrightarrow K_{f}$ of fields, we then also note $m:=[K_{f} : \mathbb{Q}] = [K : \mathbb{Q}] \cdot [K_{f} : K] = np^{\ell}$, for fixed degree $n\geq 2$ of any number field $K$. So now, inspired (as in \cite{BK2}) by field-counting advances in arithmetic statistics, we also wish to count the number of fields $K_{f}$ induced by irreducible polynomials $f$ arising from a polynomial discrete dynamical systems in Section \ref{sec2} or \ref{sec3}. To do so, we define and then determine the asymptotic behavior of 
\begin{equation}\label{N_{m}}
N_{m}(X) := \# \bigg\{K_{f}\slash \mathbb{Q} : [K_{f} : \mathbb{Q}] = m \textnormal{ and} \ |\text{Disc}(K_{f})|\leq X \bigg\}
\end{equation} as a positive real number $X\to \infty$. To this end, motivated greatly by more recent work of Lemke Oliver-Thorne \cite{lem} and then applying the first part of their [\cite{lem}, Theorem 1.2] on the function $N_{m}(X)$, we then obtain here:

\begin{cor} Fix any number field $K\slash \mathbb{Q}$ of degree $n\geq 2$ with the ring of integers $\mathcal{O}_{K}$. Assume Corollary \ref{9.3} or second part of Theorem \ref{2.3}, and let $N_{m}(X)$ be the number defined as in \textnormal{(\ref{N_{m}})}. Then we have 
\begin{equation}\label{N_{m}(x)} 
N_{m}(X)\ll_{m}X^{2d - \frac{d(d-1)(d+4)}{6m}}\ll X^{\frac{8\sqrt{m}}{3}}, \text{where d is the least integer for which } \binom{d+2}{2}\geq 2m + 1.
\end{equation}
\end{cor}

\begin{proof}
To see the inequality \textnormal{(\ref{N_{m}(x)})}, we first recall from Corollary \ref{9.3} the existence of infinitely many monic polynomials $f(x)$ over $\mathbb{Q}\subset K_{f}$ such that $K_{f}\slash \mathbb{Q}$ is an algebraic number field of degree $m=np^{\ell}$, or recall from the second part of Theorem \ref{2.3} the existence of monic integral polynomials $f(x) = x^{p^{\ell}}-x + c\in K[x]$ that are irreducible modulo any fixed prime ideal $p\mathcal{O}_{K}$ for every coefficient $c\not \in p\mathcal{O}_{K}$ and hence induce degree-$m$ number fields $K_{f}\slash \mathbb{Q}$. This then means that the set of algebraic number fields $K_{f}\slash \mathbb{Q}$ is not empty. So now, we may apply [\cite{lem}, Theorem 1.2 (1)] on the number $N_{m}(X)$, and doing so we then obtain inequality \textnormal{(\ref{N_{m}(x)})}, as needed.
\end{proof}

By applying again a result of Bhargava-Shankar-Wang [\cite{sch1}, Corollary 1.3] on the algebraic number fields $K_{f}\slash \mathbb{Q}$ induced by irreducible monic polynomials $f$ (defined over any fixed algebraic number field $K$ and hence over $\mathbb{Q}$) arising a polynomial discrete dynamical system in Section \ref{sec2}, we then also obtain the following corollary:

\begin{cor}\label{8.7}
Assume Corollary \ref{9.3} or second part of Theorem \ref{2.3}. The number of isomorphism classes of number fields $K_{f}\slash \mathbb{Q}$ of degree $m$ and $|\Delta(K_{f})| < X$ that are monogenic and have Galois group $S_{m}$ is $\gg X^{\frac{1}{2} + \frac{1}{m}}$.
\end{cor}

\begin{proof}
To see this, we first recall from Corollary \ref{9.3} the existence of infinitely many monic polynomials $f(x)$ over $\mathbb{Q}\subset K_{f}$ such that $K_{f}\slash \mathbb{Q}$ is an algebraic  number field of degree $m=np^{\ell}$, or recall from the second part of Theorem \ref{2.3} the existence of monic integral polynomials $f(x) = x^{p^{\ell}}-x + c\in K[x]$ that are irreducible modulo any fixed prime ideal $p\mathcal{O}_{K}$ for every coefficient $c\not \in p\mathcal{O}_{K}$ and hence induce degree-$m$ number fields $K_{f}\slash \mathbb{Q}$. This then means that the set of algebraic number fields $K_{f}\slash \mathbb{Q}$ is not empty. So now, applying [\cite{sch1}, Corollary 1.3] to the underlying number fields $K_{f}$ with $|\Delta(K_{f})| < X$ that are monogenic and have associated Galois group $S_{m}$, we then obtain that the number of isomorphism classes of such number fields $K_{f}$ is $\gg X^{\frac{1}{2} + \frac{1}{m}}$, as desired.
\end{proof}

\section{On the Number of Intermediate fields $L$ of $H_{f_{c(t)}}\slash \mathbb{F}_{p}(t)$ and  also $\Tilde{L}$ of $H_{g_{c(t)}}\slash \mathbb{F}_{p}(t)$}\label{sec12}

Recall that the second part of Theorem \ref{5.3} (i.e., the part in which we proved that the number $N_{c(t)}(\pi, p) = 0$ for every coefficient $c\not \equiv 0\ (\text{mod} \ \pi)$) implies that $f_{c(t)}(x) = x^{p^{\ell}} - x + c\in \mathbb{F}_{p}[t][x]$ is irreducible modulo prime $\pi \in \mathbb{F}_{p}[t]$ for any fixed prime $p\geq 3$ and for any integer $\ell \in \mathbb{Z}_{\geq 1}$. Similarly, the second part of Theorem \ref{6.3} (i.e., the part in which we proved that the number $M_{c(t)}(\pi, p) = 0$ for every coefficient $c\not \equiv -1\ (\text{mod} \ \pi)$) also implies that $g_{c(t)}(x) = x^{(p-1)^{\ell}} - x + c\in \mathbb{F}_{p}[t][x]$ is irreducible modulo prime $\pi \in \mathbb{F}_{p}[t]$ for any fixed prime $p\geq 5$ and for any $\ell \in \mathbb{Z}_{\geq 1}$. So now, since we have the inclusion $\mathbb{F}_{p}[t]\hookrightarrow \mathbb{F}_{p}(t)$ of rings and so viewing each such coefficient $c$ of $f_{c(t)}(x)$ as now an element in a rational function field $\mathbb{F}_{p}(t)$, we may then to each irreducible monic polynomial $f_{c(t)}$ associate a field $H_{f_{c(t)}}= \mathbb{F}_{p}(t)[x]\slash (f_{c(t)}(x))$. Similarly, viewing each such coefficient $c$ of $g_{c(t)}(x)$ as an element in $\mathbb{F}_{p}(t)$, then to each irreducible monic polynomial $g_{c(t)}$ we may also associate a field $H_{g_{c(t)}}= \mathbb{F}_{p}(t)[x]\slash (g_{c(t)}(x))$. It then follows from standard theory of algebraic extensions of function fields that each of the naturally constructed fields $H_{f_{c(t)}}\slash \mathbb{F}_{p}(t)$ and $H_{g_{c(t)}}\slash \mathbb{F}_{p}(t)$ is an algebraic function field, and moreover we note the degree $[H_{f_{c(t)}}: \mathbb{F}_{p}(t)]=$ deg$(f_{c(t)}) =p^{\ell}$ and also the degree $[H_{g_{c(t)}}: \mathbb{F}_{p}(t)]=$ deg$(g_{c(t)}) =(p-1)^{\ell}$. 

So now, in this section, we wish to study the problem of counting subextensions of a fixed field extension; a problem that's again of great interest and importance in arithmetic statistics. Inspired by work of Bhargava \cite{gav} on the number of quartic orders in $S_{4}$-quartic fields and of Thunder-Widmer [\cite{Jef}, Lemma 6] on the number of intermediate fields of a fixed extension of function fields, we then wish here (as a consequence of the second part of Theorem \ref{5.3} and \ref{6.3}) to count the number of subfields  $L$ of a function field $H_{f_{c(t)}}$ (resp. $\Tilde{L}$ of a function $H_{g_{c(t)}}$) induced by irreducible monic polynomials $f_{c(t)}$ (resp. $g_{c(t)}$) arising from polynomial discrete dynamical systems in Section \ref{sec5} (resp. Section \ref{sec6}), and such that each of $L$ and $\Tilde{L}$ contains a fixed function field $\mathbb{F}_{p}(t)$. With that in end, we simply take a great advantage of [\cite{Jef}, Lemma 6] and then obtain here the following corollaries:

\begin{cor}\label{12.1}
Fix $\mathbb{F}_{p}(t)$ and assume second part of Theorem \ref{5.3}. Let $H_{f_{c(t)}}$ be as before with degree $d:=[H_{f_{c(t)}}: \mathbb{F}_{p}(t)] = p^{\ell}$. Let $N(d)$ be the number of intermediate fields $L$ where $\mathbb{F}_{p}(t)\subset L \subset H_{f_{c(t)}}$. Then we have  
\begin{center}
$N(d)\leq d2^{d!}$, where $d!\sim \frac{d^d}{e^d}\sqrt{2\pi d}$ as $d\to \infty$.
\end{center}
\end{cor}
\begin{proof}
Setting $K = H_{f_{c(t)}}$ and $k = \mathbb{F}_{p}(t)$, so that the degree $[K: k] = d$, and then applying [\cite{Jef}, Lemma 6] to the extension $K \supset k$ of function fields, it then follows immediately that the number $N(d)\leq d2^{d!}$ as desired.
\end{proof}
Similarly, we also have the following corollary on the number of subfields $\Tilde{L}$ of $H_{g_{c(t)}}$ such that $\Tilde{L}\supset\mathbb{F}_{p}(t)$:
\begin{cor}
Fix $\mathbb{F}_{p}(t)$ and assume second part of Theorem \ref{6.3}. Let $H_{g_{c(t)}}$ be as before with degree $r:=[H_{g_{c(t)}}: \mathbb{F}_{p}(t)] = (p-1)^{\ell}$. Let $M(r)$ be the number of intermediate fields $\Tilde{L}$ where $\mathbb{F}_{p}(t)\subset \Tilde{L} \subset H_{g_{c(t)}}$. We have  
\begin{center}
$M(r)\leq r2^{r!}$, where $r!\sim \frac{r^r}{e^r}\sqrt{2\pi r}$ as $r\to \infty$.
\end{center}
\end{cor}
\begin{proof}
By applying a similar argument as in the Proof of Cor. \ref{12.1}, we then obtain the inequality as desired.
\end{proof} 

\addcontentsline{toc}{section}{Acknowledgments}
\section*{\textbf{Acknowledgments}}
I'm truly grateful and deeply indebted to my long-time great advisors, Dr. Ilia Binder and Dr. Arul Shankar, for all their boundless generosity, friendship and for the weekly conversations, and along with Dr. Jacob Tsimerman for always very uncomprimisingly supporting my professional and philosophical-mathematical research endeavours. I'm truly very grateful to Prof. Tsimerman for offering a very enlightening reading course on Local fields during the summer 2024, and also for his invaluable recommendation of Rosen’s amazing text book \cite{Rose} on function fields; which in the end kick-started and paved greatly the way for the thinking that was necessary for Section \ref{sec3}, \ref{sec4}, \ref{sec5}, \ref{sec6}, \ref{sec12} and part of \ref{sec10}. I'm truly very grateful to both the Department of Mathematical and Computational Sciences (in particular, Prof. Yael Karshon, Prof. Ilia Binder, Prof. Arul Shankar, Prof. Marina Tvalavadze, Prof. Alex Rennet, Prof. Michael Gr\"{o}echenig, Prof. Julie Desjardins, Prof. Duncan Dauvergne, Prof. Ke Zhang, Prof. Jaimal Thind) and Department of Mathematics (in particular, Prof. Yael Karshon, Prof. Ilia Binder, Prof. Arul Shankar, Prof. Jacob Tsimerman, Prof. Michael Gr\"{o}echenig, Prof. Florian Herzig, Prof. Daniel Litt, Prof. Konstantin Khanin and Prof. Ignacio Uriarte-Tuero) at the University of Toronto for everything. Last but not the least, I'm truly very grateful and indebted to my great life coach, Dr. Michael Bumby, and along with (N\&S) for the great conversations, friendship and everything. As a graduate research student, this work and my studies are hugely and wholeheartedly funded by Dr. Binder and Dr. Shankar. This article is very happily dedicated to all students in the whole world who embrace learning and discovering with humility and integrity! Any opinions expressed in this article belong solely to me, the author, Brian Kintu; and should never be taken at all as a reflection of the views of anyone that has been happily acknowledged by the author.

\bibliography{References}
\bibliographystyle{plain}

\noindent Dept. of Math. and Comp. Sciences (MCS), University of Toronto, Mississauga, Canada \newline
\textit{E-mail address:} \textbf{brian.kintu@mail.utoronto.ca}\newline 
\date{\small{\textit{January 13, 2026.}}}

\end{document}